\documentclass[reqno,english,11pt]{amsart}

\usepackage[utf8]{inputenc}
\usepackage[T1]{fontenc}

\usepackage{amsmath, amssymb, amsthm}
\usepackage{geometry}
\usepackage{enumitem}
\usepackage{hyperref}
\hypersetup{
  colorlinks,
  citecolor=blue,
  linkcolor=blue,
  urlcolor=green}
\usepackage{marginnote}
\usepackage{xcolor}
\usepackage{upgreek}
\usepackage{tikz}
\usetikzlibrary{patterns}
\usepackage{mathtools}

\numberwithin{equation}{section}

\theoremstyle{plain}
\newtheorem{theorem}{Theorem}
\newtheorem{lemma}{Lemma}
\newtheorem{proposition}{Proposition}

\theoremstyle{definition}

\newtheorem{remark}{Remark}[section]

\newcommand{\N}{\mathcal{N}}

\newcommand{\C}{\mathbb{C}}
\newcommand{\R}{\mathbb{R}}

\newcommand{\z}{\boldsymbol{z}}

\newcommand{\round}[1]{\left({#1}\right)}
\newcommand{\squared}[1]{\left[{#1}\right]}
\newcommand{\braces}[1]{\left\lbrace{#1}\right\rbrace}
\newcommand{\abs}[1]{\left\lvert{#1}\right\rvert}
\newcommand{\norm}[1]{\left\lVert{#1}\right\rVert}
\newcommand{\bracket}[1]{\left\langle{#1}\right\rangle}

\newcommand{\sdf}[1]{\mathsf{#1}}
\newcommand{\twovec}[2]{\begin{pmatrix}{#1}\\{#2}\end{pmatrix}}

\newcommand{\Pssi}[1]{\boldsymbol{\Psi}}

\newcommand{\ring}[1]{\mathring{#1}}
\newcommand{\Riem}{\textrm{Riem}}
\newcommand{\A}{\mathcal{A}}
\newcommand{\eps}{\varepsilon}
\newcommand{\J}{\mathcal{J}}
\newcommand{\Tr}{\mathrm{Tr}}

\newcommand{\vphi}{\varPhi}

\newcommand{\V}{\mathcal{V}}

\definecolor{ao}{rgb}{0.0, 0.5, 0.0}

\begin{document}

\title[Entire solutions to 4 dimensional Ginzburg-Landau equations]{Entire solutions to 4 dimensional Ginzburg--Landau equations and codimension 2  minimal submanifolds}
    \author[M. Badran]{Marco Badran}
    
	\address{Department of Mathematical Sciences, University of Bath, Bath, BA2 7AY, United Kingdom. }
 
    \email{mb2747@bath.ac.uk}	

	\author[M. del Pino]{Manuel del Pino}
	
	\address{Department of Mathematical Sciences, University of Bath, Bath, BA2 7AY, United Kingdom.}
	
	\email{mdp59@bath.ac.uk}
\maketitle
\begin{abstract}
We consider the magnetic Ginzburg--Landau equations in $\R^4$
\begin{equation*}
\begin{cases}
    -\eps^2(\nabla-iA)^2u =
  \frac{1}{2}(1-|u|^{2})u,\\
    \eps^2d^{*}dA=\bracket{(\nabla-iA)u,iu},
    \end{cases}
\end{equation*}
formally corresponding to the Euler--Lagrange equations for the energy functional
\begin{equation*}
     E(u,A)=\frac{1}{2}\int_{\R^4}|(\nabla-iA)u|^{2}+\eps^2|dA|^{2}+\frac{1}{4\eps^2}(1-|u|^{2})^{2}.
\end{equation*}
Here $u\colon\R^4\to \C$, $A\colon\R^4\to\R^4$ and $d$ denotes the exterior derivative acting on the one-form dual to $A$.
Given a minimal surface $M^2$ in $\R^3$ with finite total curvature and non-degenerate, we construct a solution  $(u_\eps,A_\eps)$
which has a
zero set consisting of a smooth surface close to 
$M\times \{0\}\subset \R^4$. Away from the latter surface we have
$|u_\eps| \to 1$ and  
$$\begin{aligned} 
u_\eps(x)\, \to\, \frac {z}{|z|},\quad
A_\eps(x)\, \to\, \frac 1{|z|^2} ( -z_2 \nu(y) + z_1 {\bf e}_4), \quad x = y + z_1 \nu(y) + z_2 {\bf e}_4
\end{aligned} 
$$
for all sufficiently small $z\ne 0$. Here $y\in M$ and $\nu(y)$ is a unit normal vector field to $M$ in $\R^3$.
\end{abstract}


\section{Introduction}\label{Introduction}

We consider the magnetic Ginzburg--Landau energy in $\R^n$
	\begin{equation}\label{Energy_no_domain}
        E(u,A)=\frac{1}{2}\int_{\R^n}|(\nabla-iA)u|^{2}+\eps^2|dA|^{2}+\frac{1}{4\eps^2}(1-|u|^{2})^{2}
	\end{equation}
whose argument is a pair $(u,A)$ where $u\colon\R^n\to \C$ is an order parameter and $A\colon\R^n\to \R^n$ is a vector field 
which we also regard as a one form $A(x)=A_j(x)dx^j$ in $\R^n$,
so that its exterior derivative is
\begin{equation*}
    dA = \sum_{j<k}(\partial_jA_k-\partial_kA_j)dx^j\wedge dx^k.
\end{equation*}
and $|dA|^2$ in \eqref{Energy_no_domain} is given by $ |dA|^2=\frac12\sum_{j,k=1}^n\abs{\partial_kA_j-\partial_jA_k}^2$.

\medskip

Energy \eqref{Energy_no_domain} arises in the 
classical Ginzburg--Landau theory of superconductivity; the quantity $|u|^2$ measures density of Cooper pairs of superconducting electrons and $A$ is the induced magnetic potential. The zero set of $u$ is interpreted as that of defects of the underlying material where superconductivity is lost. The functional $E$ arises from a $U(1)$-gauge theory, meaning that it is invariant under $U(1)$ gauge transformations
\begin{equation*}
	(u,A)\mapsto (ue^{i\gamma},A+\nabla\gamma)
\end{equation*}
for any $\gamma\colon\R^n\to\R$.

\medskip

The associated Euler--Lagrange equations for $E$ are given by
\begin{equation}\label{full_equations}
\begin{cases}
    -\eps^2(\nabla-iA)^2u =
  \frac{1}{2}(1-|u|^{2})u,\\
   \ \ \, \eps^2d^{*}dA=\bracket{(\nabla-iA)u,iu}
\end{cases}
\quad\text{in }\R^n.
\end{equation}
In the above expression the brackets $\langle\cdot,\cdot\rangle$ represent the standard inner product of complex numbers. The operator in the left-hand side of the second equation reads, in Euclidean coordinates,
\begin{align*}
	d^*dA&=-\sum_{j,k=1}^n\partial_j(\partial_jA_k-\partial_k A_j)dx^k=\sum_{k=1}^n(-\Delta A_k+\partial_k\mathrm{div}A)dx^k.
\end{align*}
The term $\bracket{(\nabla-iA)u,iu}$ is dual to a gauge-invariant real-valued 1-form called superconducting current. 

\medskip

Taubes \cite{taubes1980arbitraryn} and Berger--Chen \cite{berger1989symmetric} 
have found solutions of \eqref{full_equations} in the planar case $n=2$ with isolated zeros of $u(z)$ (vortices). 
When $\eps=1$ the degree $1$ vortex solution $(u_0, A_0)$ found in \cite{berger1989symmetric} takes the form 
\begin{equation}\label{radialansatz}
    u_0(z)=  f(r)e^{i\theta},\quad A_0(z) =a(r)\nabla\theta, \quad z=re^{i\theta}
\end{equation}
where $r=|z|$, \ $\nabla\theta= |z|^{-2}(-z_2,z_1)$.  
The functions $f(r)$ and $a(r)$ are positive solutions to the system of ODE 
\begin{equation}\label{system_fa}
\begin{cases}
		-f''-\frac{f'}{r}+\frac{(1-a)^2f}{r^2}-\frac{1}{2}f(1-f^2)=0,\\
-a''+\frac{a'}{r}-f^2(1-a)=0
\end{cases}
\quad\text{in }(0,+\infty)
\end{equation}
with $f(0)=a(0)=0$. Remark that by self-duality the pair $(f,a)$ satisfies an equivalent (see \cite{taubes1980equivalence}) system of first order equations
\begin{equation}\label{first order}
	\begin{cases}
		f'=\frac{(1-a)f}r\\
		a'=\frac r2(1-f^2)
	\end{cases}
\end{equation} 
see for example \cite{gustafsonrid2000stability}.
The pair $(u_0,A_0)$ is the unique (up to gauge transformations) solution of \eqref{full_equations} with $\eps=1$ and exactly one zero with topological degree 1 at the origin. This solution is linearly stable as established in \cite{stuart1994dynamics,gustafsonrid2000stability}.  Also, $f$ and $a$ are strictly monotone increasing and
\begin{equation*}
	f(r)-1=O(e^{-{r}}),\quad a(r)-1= O(e^{-r}),\quad\text{as }r\to \infty,
\end{equation*}
see \cite{berger1989symmetric,plohr1981behavior}. 
We observe that, for any $\eps>0$, the scaling $(f(r/\eps)e^{i\theta}, a(r/\eps)d\theta)$ solves \eqref{full_equations} for $n=2$.

\medskip

In what follows we restrict ourselves to the four dimensional case $n=4$ in \eqref{full_equations}. Denote
\begin{equation*}
	\nabla^A\coloneqq\nabla-iA,\quad \Delta^A\coloneqq (\nabla-iA)^2.
\end{equation*}
We find entire solutions $(u_\eps,A_\eps)$ to the system
 \begin{equation}\label{Energy_4d}
        \begin{cases}
    -\eps^2\Delta^Au =
  \frac{1}{2}(1-|u|^{2})u,\\
   \eps^2d^{*}dA=\bracket{\nabla^Au,iu}
\end{cases}
\quad\text{in }\R^4
	\end{equation}
that exhibit a zero set $\eps$-close to a prescribed codimension 2 smooth minimal surface $\widetilde M$ embedded in $\R^4$. More precisely, writing $A_0=(A_{01},A_{02})$, then
for all sufficiently small $|z|$ and $\eps \to 0$ these solutions satisfy 
\begin{equation}\label{whatever}
    u_\eps(x) = u_0(z/\eps) +o(1), \quad A_\eps(x) = 
    \eps^{-1}A_{01}(z/\eps)\nu_1(y)  + 
     \eps^{-1}A_{02}(z/\eps)\nu_2(y)+ 
    o(1),
\end{equation}
where $x= y + z_1 \nu_1(y) +  z_2 \nu_2(y)$, being $y\in \widetilde M$ and $\{\nu_1$,$\nu_2\}$ an orthonormal frame of the normal space to $\widetilde M$ at the point $y$. In particular 
we obtain $|u_\eps|\to 1$ away from $\widetilde M$ and
\begin{equation*}
	    u_\eps(x)\, \to\, \frac {z}{|z|},\quad
A_\eps(x)\, \to\, \frac 1{|z|^2} (-z_2 \nu_1(y) + z_1 \nu_2(y)).
\end{equation*}

This type of connection between solution of semilinear PDEs and minimal submanifolds is well understood in the Allen--Cahn case
\begin{equation} \label{ac}
-\eps \Delta u =\frac{1}{\eps} (1-u^2)u
\end{equation}
for real-valued functions $u$. Solutions that concentrate along minimal hypersurfaces have been broadly studied after the pioneering work by Modica--Mortola \cite{mortola5esempio} and De Giorgi's conjecture \cite{de1979convergence}.
In general one expects existence of solutions $u_\eps(x)$ whose zero set lies close to a given minimal hypersurface $M$ and 
$u_\eps (x) \approx \tanh (z/\sqrt{2}\eps)$, where $z$ is a normal coordinate to $M$. This is the principle behind various constructions including the works by Pacard and   Ritor\'e \cite{pacard2003constant} in compact manifolds and by the second author, M.\ Kowalczyk and J.\ Wei \cite{del2011giorgi, del2013entire} for entire solutions in $\R^n$. 

\medskip 
In \cite{del2013entire} such a construction has been achieved for $n=3$ in  \eqref{ac} and    
each given embedded minimal surface $M$ in $\R^3$ with finite total curvature and non-degenerate in suitable sense, the simplest example being the catenoid. 
The purpose of this paper is to construct entire solutions to \eqref{Energy_4d} with the asymptotic behaviour \eqref{whatever} around the class of codimension 2 minimal submanifolds in $\R^4$ corresponding to the embedding $M\times \{0\} \subset  \R^4$ of the class of minimal surfaces $M$ in $\R^3$ considered in  \cite{del2013entire}.   

\smallskip
Thus,  we look for entire solutions $(u_\eps, A_\eps)$ of system \eqref{Energy_4d} such that the zero set of $u_\eps$  is a codimension 2 manifold 
close to $\widetilde M= M\times \{0\}$ and $|u_\eps|\to 1$ as $\eps \to 0$ in compact subsets of $\R^4\setminus \widetilde M$.
Next we recall some major features of this celebrated class of minimal surfaces.

\subsection*{Complete minimal surfaces of finite total curvature in $\R^3$}

A surface $M$ embedded in $\R^3$ is said to be a minimal surface if it is a critical point of the area functional, or equivalently if its mean curvature
$ k_1 + k_2 \equiv 0$ vanishes on $M$, where $k_1, k_2$ are the principal curvatures. Denoting with $K= k_1k_2 \le 0 $ the Gauss curvature, we say that $M$ has finite total curvature if 
\begin{equation*}
	\int_M|K|<\infty.
\end{equation*}
The theory of embedded, complete minimal surfaces $M$ with finite total curvature in $\R^3$, has reached a notable development in the last four decades. For more than a century, the plane and the catenoid were the only two known examples of such surfaces. While a general theory for those surfaces was available for a long time, only in 1981 Costa discovered a first non-trivial example for $M$ embedded, orientable, with genus 1, later generalised by Hoffman and Meeks to arbitrary genus, see \cite{costa1984example, hoffman1990embedded}. These surfaces have three ends, two catenoidal and one planar.
Many other examples of multiple-end embedded minimal surfaces
have been found since; see for instance \cite{kapouleas1997complete, traizet2002embedded} and references therein.
Outside a large cylinder, a general manifold  $M$ in this class decomposes into the disjoint union of $m$ unbounded  connected components $M_1,\dots,M_m$, called its ends, which are asymptotic to 
either catenoids or planes, all of them with parallel axes, see 
\cite{osserman2013survey, jorge1983topology, schoen1983uniqueness}. 
After a rotation, we can choose coordinates $x=(x_1,x_2,x_3)=(x',x_3)$ in $\R^3$ and a large number $R_0$ such that
\begin{equation*}
	M\setminus\{|x'| < R_0\}=\bigcup_{j=1}^m M_j.
\end{equation*}
Each end $M_k$ can asymptotically be represented as 
\begin{equation*}
	M_k=\braces{x\in\R^3\,:\, |x'| \geq R_0,\ x_3=F_k(x')}
\end{equation*}
where for suitable constants $a_k,b_k,b_{ik}$, the function $F_k$ satisfies 
\begin{equation}\label{expansion F_k}
	F_k(x')=a_k\log |x'| +b_k+b_{ik}\frac{x^i}{|x'|^2}+O(|x'|^{-2})\quad |x'| \geq R_0.
\end{equation}
The coefficients $a_k$ are ordered and balanced, in the sense that 
\begin{equation}\label{a_ks}
	a_1\leq a_2\leq \cdots\leq a_m,\quad a_1+\cdots +a_m=0.
\end{equation}
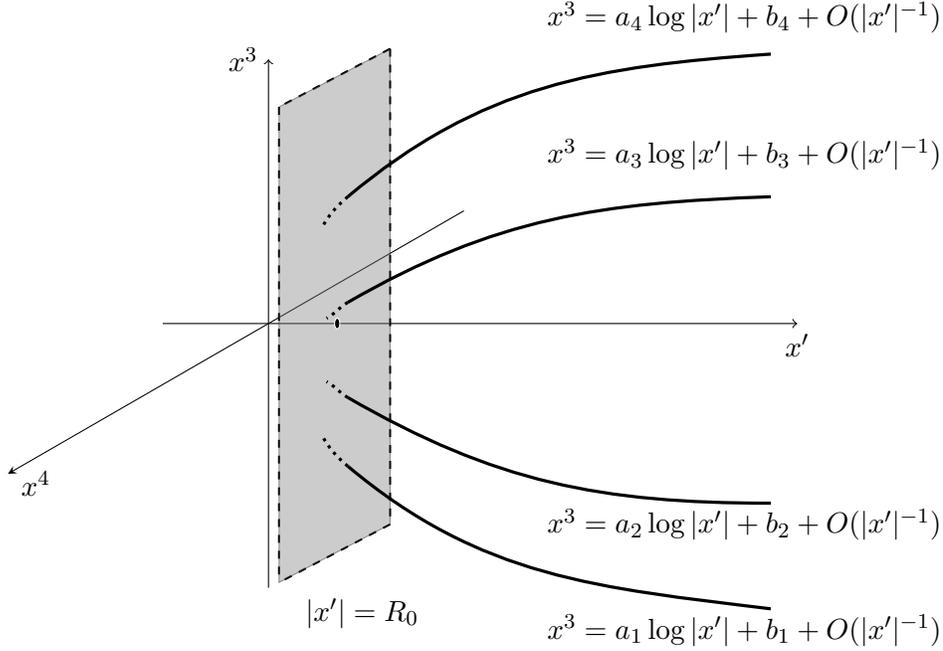
\begin{figure}
    \centering
         \begin{tikzpicture}[x=20pt,y=20pt,yscale=1,xscale=1]
    \draw[black, thin, ->] (0,-5) -- (0,5);
    \draw[black, thin, ->] (-2,0) -- (10,0);
    \draw (0,0) -- (30:3cm);
    \draw[-stealth] (0,0) -- (210:4cm);
    \draw[thick ,dashed] (2.3,5.2) -- (2.3,-3.8);
    \draw[thick ,dashed] (.2,4.1) -- (.2,-4.9);
    \draw[thick ,dashed] (2.3,5.2) -- (.2,4.1);
    \draw[thick ,dashed] (2.3,-3.8) -- (.2,-4.9);
    \draw[fill=gray, opacity=0.4] (2.3,5.2) -- (2.3,-3.8) -- (.2,-4.9) -- (.2,4.1) -- cycle;
    
    \shadedraw[inner color=black,outer color=black, draw=white] (1.3,0) ellipse (.05 and .1);
    
    \draw[black, very thick, out=30,in=182] (1.5,0.4) to (9.5,2.4);
    \draw[black, very thick, dotted, out=210, in = 28] (1.5,0.4) to (1.1,0.1);
    
    \draw[black, very thick, out=40,in=184] (1.5,2.4)to (9.5,5.1);
    \draw[black, very thick, dotted, out=220, in = 60] (1.5,2.4) to (1,1.8);
    
    \draw[black, very thick, out=-30,in=-180] (1.5,-1.4) to (9.5,-3.4);
    \draw[black, very thick, dotted, out=-210, in = -28] (1.5,-1.4) to (1.1,-1.1);
    
    \draw[black, very thick, out=-40,in=-187] (1.5,-2.7) to (9.5,-5.4);
    \draw[black, very thick, dotted, out=-220, in = -60] (1.5,-2.7) to (1,-2.1);
    
    \draw[black, thick] (0,5) circle (0pt) node[anchor=east] {$x^3$};
    \draw[black, thick] (-4.4,-2.6) circle (0pt) node[anchor=north] {$x^4$};
    \draw[black, thick] (10,0) circle (0pt) node[anchor=north] {$x'$};
    \draw[black, thick] (.5,-5) circle (0pt) node[anchor=north west]  {$|x'|=R_0$};
    \draw[black, thick] (9,5.3) circle (0pt) node[anchor=south] {$x^3=a_4\log |x'|+b_4+O(|x'|^{-1})$};
    \draw[black, thick] (9,2.7) circle (0pt) node[anchor=south] {$x^3=a_3\log |x'|+b_3+O(|x'|^{-1})$};
    \draw[black, thick] (9,-4.3) circle (0pt) node[anchor=south] {$x^3=a_2\log |x'|+b_2+O(|x'|^{-1})$};
    \draw[black, thick] (9,-6.3) circle (0pt) node[anchor=south] {$x^3=a_1\log |x'|+b_1+O(|x'|^{-1})$};
       \end{tikzpicture}
       \caption{An outline of the asymptotic behaviour of the ends of the surface $M$ in the case where $m=4$.}
\end{figure}

The second variation of the area functional corresponds to the Jacobi operator of $M$
\begin{equation*}
	J_M[h]=\Delta_M h+|\A_M|^2h,
\end{equation*} 
where $|\A_M|^2=-K$ is the square norm of the second fundamental form. For later purposes, we record the curvature decay along the ends 
\begin{equation}\label{estimate A_M}
	|\A_M|^2=k_1^2+k_2^2=O(r^{-4}),\quad r=|x'|
\end{equation}
for $r$ large. This is a consequence of the expansion \eqref{expansion F_k}, see \cite{del2013entire}.

\medskip

Translations along the coordinate axes and rotation  around the $x_3$-axis induce four bounded Jacobi fields $z_i$, $i=0,\dots,3$, that is $J_M[z_i]=0$ 
\begin{equation}\label{jacobi fields}
	\begin{split}
	z_0(y)&=\nu(y)\cdot\round{-y_2,y_1,0},\\
	z_i(y)&=\nu(y)\cdot e_i\quad i=1,2,3.
	\end{split}
\end{equation}
We say that $M$ is non-degenerate if the  second variation of the area functional has no bounded kernel other than that induced by rigid motions, that is
\begin{equation}\label{nondeg}
	\braces{z\in L^\infty(M)\,:\, \Delta_M z+|\A_M|^2z=0}=\mathrm{span}\braces{z_0,z_1,z_2,z_3}.
\end{equation}
The assumption of non-degeneracy is known to hold in some notable cases of embedded, minimal surfaces in $\R^3$, like the catenoid or the Costa--Hoffman--Meeks surface of any genus, see \cite{nayatani1992morse, nayatani1993morse, morabito2009index}. 
\subsection*{Main results}
In what follows, $M$ designates a complete, minimal surface with finite total curvature embedded in $\R^3$, which is also non-degenerate in the sense  \eqref{nondeg}. For the moment we assume that the order in the ends  \eqref{a_ks} is strict, namely
\begin{equation}\label{strict a}
	a_1< a_2< \cdots< a_m.
\end{equation}
Our main result states the existence of a solution of Problem \eqref{Energy_4d} 
with a profile as in \eqref{whatever} with $\widetilde M= M\times \{0\}$ and $(\nu_1,\nu_2)=(\nu,{\bf e}_4)$, where $\nu$ is a unit normal vector field on $M$. We will identify $\widetilde M=M\times\{0\}$ with  $M$.

\begin{theorem}\label{Less general theorem}
Let $M$ satisfy assumption  \eqref{strict a}.
	Then there is a number $\delta>0$ such that for all  sufficiently small $\eps>0$, there exists a solution $(u_\eps,A_\eps)$ of problem \eqref{Energy_4d} which as $\eps\to 0 $
	satisfies 
	\begin{equation} \label{aaa}   
		u_\eps(x)=u_0\round{\tfrac{z-\eps h_1(y)}{\eps}}+O(\eps^2),\quad A_\eps(x)=  \eps^{-1}A_{01}\round{\tfrac{z-\eps h_1(y)}{\eps}}  \nu(y) +  \eps^{-1}A_{02}\round{\tfrac{z-\eps h_1(y)}{\eps}}   {\bf e}_4 +O(\eps)
	\end{equation}
	for all points $$  x=y+z_1\nu(y)+z_2{\bf e}_4, \quad   y\in M, \quad 
 \abs{z}<\delta,  $$
	and a smooth function $h_1: M\to \R^2$ that satisfies
	\begin{equation}
		\norm{h_1}_{L^\infty(M)} \leq C\eps.
	\end{equation}
Besides $|u_\eps|\to1$ uniformly in compact subsets of $\R^4\setminus M$.
\end{theorem}
From \eqref{aaa} the asymptotic behavior of the predicted solution near $M$ for $z\ne 0$ is given by    
\begin{equation*}
	    u_\eps(x)\, \to\, \frac {z}{|z|},\quad
A_\eps(x)\, \to\, \frac 1{|z|^2} (-z_2 \nu(y) + z_1 {\bf e}_4). 
\end{equation*}
The proof yields that corresponding bounds in the statement of the theorem can also be obtained for derivatives of any order. In fact, the contraction mapping principle, the fact that $D_zu_0(0)$ is invertible and the  implicit function theorem yields that the set of values $x$ with  $u_\eps(x)=0$ can be described as a smooth surface of the form
$$
x=   y + h_1^1(y) \nu(y) +  h_1^2(y){\bf e_4}  + O(\eps^2),  \quad y\in M.
$$

It is possible to generalise Theorem \ref{Less general theorem} by removing the hypothesis of non-parallel ends \eqref{strict a}. 
In this case, we choose a balanced, ordered vector of real numbers
\begin{equation*}
	\lambda= (\lambda_1,\dots,\lambda_m),\quad \lambda_1\leq\lambda_2\leq \cdots\leq \lambda_m
\end{equation*}
with 
\begin{equation}\label{balendcond}
	\sum_{i=1}^m\lambda_i=0.
\end{equation}
We will prove the existence of a solution $(u_{\eps},A_\eps)$ such that $u_\eps^{-1}(0)$ lies uniformly close to $M$, for $r(x)=O(1)$, and its $k$-th end, $k=1,\dots,m$ lies at uniformly bounded distance from the graph 
\begin{equation}\label{graph ends with beta}
	 (x_3, x_4)=\left(F_k(x')+\eps\lambda_k\log|x'|,\  0\right),\quad |x'|\geq R_0.
\end{equation}
As done in \cite{del2013entire}, we need to make a further geometric requirement on the $\lambda_k$'s. If two ends are parallel, say $a_k=a_{k+1}$, then we will need that $\lambda_{k+1}>\lambda_k$ for otherwise the graphs \eqref{graph ends with beta} will eventually intersect. For some $\sigma\in (0,1)$, we require
\begin{equation}\label{divendcond}
	\lambda_{k+1}-\lambda_k>4/\sigma\quad\text{if}\quad a_{k+1}=a_k
\end{equation}
With these assumptions we can state the general result. 
\begin{theorem}\label{main theorem}
    Let $M$ be as above. There is a number $\delta>0$ such that for all  sufficiently small $\eps>0$ and $\lambda=(\lambda_1,\dots,\lambda_m)$ satisfying conditions \eqref{balendcond} and \eqref{divendcond} there exists a solution $(u_\eps,A_\eps)$ of problem \eqref{Energy_4d} which as $\eps\to 0 $
	satisfies 
	\begin{equation*}   
		u_\eps(x)=u_0\round{\tfrac{z-\eps h(y)}{\eps}}+O(\eps^2),\quad A_\eps(x)=  \eps^{-1}A_{01}\round{\tfrac{z-\eps h(y)}{\eps}}  \nu(y) +  \eps^{-1}A_{02}\round{\tfrac{z-\eps h(y)}{\eps}}   {\bf e}_4 +O(\eps)
	\end{equation*}
	for all points $$  x=y+z_1\nu(y)+z_2{\bf e}_4, \quad   y\in M, \quad 
 \abs{z}<\delta,  $$
	and where the smooth function $h: M\to \R^2$ satisfies
    \begin{equation*}
		h(y)=\round{(-1)^k\lambda_k\log |y'|,\, 0}+h_1(y)\quad y=(y',y^3,0)\in M_k\times\{0\}.
	\end{equation*}
 with 
 \begin{equation}
		\norm{h_1}_{L^\infty(M)} \leq C\eps.
	\end{equation}
Besides, $|u_\eps|\to1$ uniformly in compact subsets of $\R^4\setminus M$.
\end{theorem}

The connection between critical points of functionals 
of Allen--Cahn or Ginzburg--Landau types and, respectively, codimension 1 and 2 minimal submanifolds has long been known, and a large literature has been dedicated to that subject. 
Among other works,  in the Allen--Cahn case we can mention \cite{modica1979gamma, sternberg1988effect, kohn1989local, tonegawa2012stable, roger2008convergence, del2011giorgi, del2013entire, del2015serrin, pacard2003constant}.  
Limits of critical points of min-max type to build codimension 1 minimal surfaces on compact manifolds have recently  been used in \cite{chodosh2020minimal, bellettini2020inhomogeneous, guaraco2018min} as a PDE alternative to the Almgren--Pitts min-max approach  \cite{pitts1976existence, marques2017existence, schoen1981regularity}.

\medskip

As for complex-valued Ginzburg--Landau type equations, 
point vortex concentration in the two dimensional case has been analyzed in many works, among  them the classical  books \cite{bethuel1994ginzburg, sandier2008vortices,pacard2012linear}. 
In higher dimensions, limits towards geodesics in  three dimensional space or more generally codimension 2 minimal surfaces have been analyzed from the optics of the calculus of variations in \cite{riviere1996lines, contreras2017nearly, lin1999complex, bethuel2001asymptotics, montero2004local, jerrard2004local, pigati2021minimal}.
 
\medskip

In the self-dual magnetic Ginzburg--Landau context of energy \eqref{Energy_no_domain} in a compact manifold, Pigati and Stern \cite{pigati2021minimal}
proved that critical points with uniformly bounded energies 
approach the
weight measure of a stationary, codimension 2 integral varifold. Our result can be interpreted as a form of converse of their statement in the (non-compact) entire space for a special class of minimal submanifolds. 
In the recent work \cite{liu2021entire} Liu, Ma, Wei and Wu have found a different class of entire solutions to \eqref{Energy_4d} 
connected with the so-called saddle solution of Allen--Cahn equation and a special minimal surface built by Arezzo and Pacard \cite{arezzo2003complete}. Symmetries allow them to reduce the problem to one in two variables. 

\medskip
After completing this paper, we became aware of the very interesting work by De Philippis and Pigati \cite{pigati-dephilippis}. They have established a result that complements the findings in \cite{pigati2021minimal} for the scenario of a non-degenerate codimension 2 minimal submanifold. Their method, based on variational techniques, does not provide detailed asymptotic information. However, they have successfully resolved the more challenging case of Ginzburg--Landau equations where no induced magnetic field is present. Our techniques do not extend to cover that particular case.

\section{Preliminaries }

In this Section we discuss relevant features of the model that will be used in the proof of Theorems \ref{Less general theorem} and \ref{main theorem}. We will discuss gauge invariance and its effect on the linearised operator, from which follows the important decomposition \eqref{decomposition_of_L_cmp} below. Moreover, we set notation that will be used throughout the paper. 

\medskip

Here and in what follows we define the dot product between pairs as 
\begin{equation*}
	\twovec{\phi}{\omega}\cdot\twovec{\psi}{\eta}=\bracket{\phi,\psi}+\eps^2\omega\cdot\eta.
\end{equation*}
For a pair $W=(u,A)$, we define
\begin{equation}\label{error_function_S}
	S(W)\coloneqq \twovec{-\eps^2\Delta^{A}u-\frac{1}{2}(1-|u|^{2})u}{\eps^2d^{*}dA-\bracket{\nabla^Au,iu}},
\end{equation}
so that $W$ is a solution to \eqref{Energy_4d} if and only if $S(W)=0$.

\subsection{Gauge invariance} 

One of the main features of energy \eqref{Energy_no_domain} is its invariance under the action of the $U(1)$ gauge group, namely
\begin{equation*}
	E(\mathrm{G}_\gamma(u,A))=E(u,A)
\end{equation*}
where $\mathrm{G}_\gamma$ is defined, for any real-valued map $\gamma$, as
\begin{equation}\label{gauge_transform}
	\mathrm{G}_\gamma(u,A)\coloneqq (ue^{i\gamma},A+d\gamma).
\end{equation}
One immediate consequence of this fact is that the global minimiser $(1,0)$ is moved into other global minimisers by gauge transformations, providing directly a collection of solutions to $S(W)=0$ given by 
\begin{equation}\label{pure gauges}
	(e^{i\gamma},d\gamma).
\end{equation}
Solutions of the form \eqref{pure gauges} are called pure gauges. Formally, another way of writing pure gauges is 
\begin{equation}\label{pure gauges v2}
	(\psi,d\psi/i\psi),\quad \psi\colon \Omega\to\mathbb{S}^1.
\end{equation}
where $\Omega$ is the domain considered.

\subsection{The linearised operator}
The Fr\'echet derivative $S'(W)$ around $W$ is given by 
\begin{equation}\label{S'(U)}
	S'(W) \twovec{\phi}{\omega}=\twovec{-\eps^2\Delta^{A}\phi -\frac{1}{2}(1-|u|^{2})\phi+u\bracket{u,\phi} +2i\eps^2\nabla^{A}u\cdot \omega-iud^{*}\omega}{\eps^2d^{*}d\omega+|u|^{2}\omega-\bracket{\nabla^Au,i\phi}-\bracket{\nabla^A\phi,iu}}.
\end{equation}
If $W$ is a solution, part of the kernel of $S'(W)$ is given by the invariance of $E$ under $U(1)$-gauge transformations $\mathrm{G}_\gamma$, given by \eqref{gauge_transform}, for any real-valued $\gamma$. By direct linearisation, it is easy to see that for any real-valued map $\gamma$
\begin{equation}\label{gauge_zero-mode}
	\Theta_W[\gamma]\coloneqq \twovec{iu\gamma}{d\gamma}
\end{equation}
belongs to $\ker S'(W)$.
We can give a characterization of the orthogonality to all elements in the form \eqref{gauge_zero-mode} by computing the formal adjoint of $\Theta_W$: if $\Phi=(\phi, \omega)$
\begin{align*}
	\int\Phi\cdot\Theta_W[\gamma]&=\int\twovec{\phi}{\omega}\cdot\twovec{iu\gamma}{d\gamma}\\
	&=\int\bracket{\phi,iu\gamma}+\eps^2\omega\cdot d\gamma\\
	&=\int\gamma\squared{\eps^2d^*\omega+\bracket{\phi,iu}},
\end{align*}
from where we infer that $\Phi$ is orthogonal to the gauge-kernel if and only if
\begin{equation}\label{goc}
	\Theta_W^*[\Phi]\coloneqq \eps^2d^*\omega+\bracket{\phi,iu}=0.
\end{equation}
The linearised operator \eqref{S'(U)} can be expressed in a more useful form, by the decomposition
\begin{align}\label{decomposition_of_L}
	S'(W) \twovec{\phi}{\omega}=&\twovec{-\eps^2\Delta^{A}\phi-\frac{1}{2}(1-3\left|u\right|^{2})\phi+2i\eps^2\nabla^{A}u\cdot\omega}{
-\eps^2\Delta\omega+\left|u\right|^{2}\omega-2\bracket{\nabla^Au,i\phi}}-\begin{pmatrix}iu\left(\eps^2d^*\omega+\bracket{\phi,iu}\right)\\
d\left(\eps^2d^*\omega+\bracket{\phi,iu}\right)
\end{pmatrix}
\end{align}
where $-\Delta\omega\coloneqq \round{d^*d+dd^*}\omega$ is the Hodge Laplacian on forms. Defining 
\begin{equation}
	L_W\twovec{\phi}{\omega}\coloneqq \twovec{-\eps^2\Delta^{A}\phi-\frac{1}{2}(1-3\left|u\right|^{2})\phi+2i\eps^2\nabla^{A}u\cdot\omega}{
-\eps^2\Delta\omega+\left|u\right|^{2}\omega-2\bracket{\nabla^Au,i\phi}}
\end{equation}
we can write \eqref{decomposition_of_L} as
\begin{equation}\label{decomposition_of_L_cmp}
	S'(W)[\Phi]=L_W[\Phi]-\Theta_W\Theta_W^*[\Phi]
\end{equation}
and obtain a much clearer insight on the linearised. Indeed, for increments satisfying \eqref{goc} we have $S'(W)[\Phi]=L_W[\Phi]$ where $L_W$ is an elliptic operator which behaves at infinity like $-\Delta+\text{Id}$ for all finite energy configurations $W$. For this reason, we refer to $L_W$ as the gauge-corrected linearised.

\medskip

The natural space where the operator $L_W$ is defined and continuous is the space $H^1_W(M)$ of functions  $\Phi=(\phi, \omega)$ for which the the covariant Sobolev norm
\begin{equation}\label{norm H1U}
	\norm{\Phi}^2_{H^{1}_W(M)}=\norm{\phi}^2_{L^2(M)}+\norm{\nabla^A\phi}^2_{L^2(M)}+\norm{\omega}^2_{L^2(M)}+\norm{\nabla\omega}^2_{L^2(M)}
\end{equation}
is finite, being $\nabla$ the Levi-Civita connection on $M$ and $A$ the 1-form of $W$. We introduce a notation that allows us to write in a compact way the gauge-corrected linearised $L_W$. Define the following gradient-like operator
\begin{equation}\label{gradlike}
	\nabla_W\twovec{\phi}{\omega}=\twovec{\nabla^A\phi}{(d+d^*)\omega}
\end{equation}
where again $W=(u, A)$. Observe that applying the operator $d^*+d$ to a $p$ form one obtains the formal sum of a $(p-1)$-form and a $(p+1)$-form. In what follows this will not play a role; we only use it to be able to define the Hodge operator as a ``square'' of some first order operator 
\begin{equation*}
	\Delta\omega=(d+d^*)^*(d+d^*)\omega.
\end{equation*}
If in particular $\omega$ is a 1-form, then we define for any vector $v$
\begin{equation*}
	\nabla_{v,W}\twovec{\phi}{\omega}\coloneqq \nabla_W\twovec{\phi}{\omega}(v)= \twovec{\nabla^A_v\phi}{d\omega(\cdot,v)}.
\end{equation*}
Composing $\nabla_W$ with its formal adjoint we obtain a laplacian-like operator 
\begin{equation*}
	-\Delta_W=\nabla_W^*\nabla_W
\end{equation*}
which explicitly reads 
\begin{equation*}
	-\Delta_W\twovec{\phi}{\omega}=\twovec{-\Delta^A\phi}{-\Delta\omega}.
\end{equation*}
With this notation, the linearised can be written as 
\begin{equation}\label{expanded_linearized}
	L_W[\Phi]=-\eps^2\Delta_W\Phi+\Phi+T_W\Phi
\end{equation} 
where
\begin{equation*}
	T_W\twovec{\phi}{\omega}\coloneqq \twovec{-\frac{3}{2}(1-\left|u\right|^{2})\phi+2i\eps^2\nabla^{A}u\cdot\omega}{
-(1-\left|u\right|^{2})\omega-2\bracket{\nabla^Au,i\phi}}.
\end{equation*}
The sense of this expression is that it highlights the behaviour at infinity of $L_W$, namely $L_W\sim -\Delta_W+\mathrm{Id}$. This is due to the fact that $T_W$ vanishes at infinity whenever $W$ is a finite energy configuration, that is 
\begin{equation*}
	|u|^2\to 1,\quad \nabla^Au\to 0\quad\text{near }\infty.
\end{equation*}

\subsection{The planar case}
Suppose we have a solution $W_0=(u_0,A_0)$ to $S(W_0)=0$ in the Euclidean space $\R^n$. As pointed out in \cite{stuart1994dynamics}, in this case there is a $n$-parameter family of isometries of the space, namely the translations, that generate elements of the kernel of $S'(W_0)$. Direct differentiation along the $j$-th direction produces elements $\widetilde{V}_j=(\partial_ju_0,\partial_jA_0)$ which are not in $L^2$ and in general have non-vanishing projection along the gauge-kernel. On the other hand, setting 
\begin{equation}\label{elements Vj}
	V_j\coloneqq \widetilde{V}_j-\Theta_{W_0}[A_{0j}]
\end{equation}
amounts to projecting $\widetilde V_j$ onto the orthogonal of the gauge-kernel, in the sense that 
\begin{align*}
	\Theta_{W_0}^*[V_j]&=-\partial_i(\partial_jA_{0i})+\partial_{ii}A_{0j}+\langle\nabla_j^{A_0}u_0,iu_0\rangle\\
	&=-\round{d^*dA_0-\bracket{\nabla^{A_0}u_0,iu_0}}_j=0.
\end{align*}
for $j=1,\dots,n$. Moreover
\begin{equation*}
	V_j=\twovec{\nabla_j^{A_0}u_0}{(\partial_jA_{0i}-\partial_iA_{0j})dz^i}\in L^2(\R^n),\qquad \forall j\in\{1,\dots,n\}.
\end{equation*}

\medskip

As already mentioned, one case that will be of particular interest for us is the planar case $\R^2$ with the single-vortex solution centred at the origin, that is, the solution $U_0=(u_0,A_0)$ introduced in \eqref{radialansatz}. The gauge-corrected linearised around ${U_0}$, which will be denoted with 
\begin{equation}\label{Lstraight}
	\sdf{L}\coloneqq S'({U_0})+\Theta_{U_0}\Theta_{U_0}^*,
\end{equation}
is non-degenerate and stable, see \cite{stuart1994dynamics} and \cite{gustafsonrid2000stability}, that its kernel is generated only by the elements \eqref{elements Vj}, which in this case are given by 
\begin{equation}\label{expression sdfVj}
	\sdf{V}_1=\twovec{f'}{\frac{a'}{r}dt^2},\quad \sdf{V}_2=\twovec{if'}{-\frac{a'}{r}dt^1},
\end{equation}
being $f$ and $a$ the solutions to system \eqref{system_fa}. We claim that
\begin{equation*}
\sdf{V}_\alpha=(\nabla_{U_0}{U_0})(e_\alpha)\eqqcolon\nabla_{\alpha,{U_0}}{U_0},\quad\alpha=1,2.
\end{equation*} 
Indeed, using that $d^*{A_0}=0$, we have 
\begin{equation}\label{grad U=V}
	\nabla_{\alpha,{U_0}}\twovec{u_0}{{A_0}}=\twovec{\nabla^{{A_0}}u_0}{d{A_0}}(e_\alpha)=\twovec{\partial_\alpha u_0-i{A_0}_\alpha u_0}{(\partial_\beta{A_0}_\alpha-\partial_\alpha{A_0}_\beta)dt^\beta}.
\end{equation}
Let $\alpha=1$; then, using the first order system \eqref{first order}, we have
\begin{align*}
	\partial_1u_0-iA_{01}u_0&=e^{i\theta}\left[f'\partial_1r+if\partial_1\theta-iaf\partial_1\theta\right]\\
	&=e^{i\theta}\left[f'\cos\theta-if(1-a)\frac{\sin\theta}r\right]\\
	&=f'e^{i\theta}[\cos\theta-i\sin\theta]\\
	&=f'
\end{align*}
and 
\begin{align*}
	(\partial_\beta{A_{01}}-\partial_1{A_0}_\beta)dt^\beta&=[\partial_1(a\partial_2\theta)-\partial_2(a\partial_1\theta)]dt^2\\
	&=a'(\partial_1r\partial_2\theta-\partial_2r\partial_1\theta)dt^2\\
	&=(a'/r)dt^2
\end{align*}
With a similar calculation for the case $\alpha=2$ we find
\begin{equation*}
	\nabla_{\alpha,{U_0}}\twovec{u_0}{{A_0}}=\begin{cases}
		(f',\frac{a'}{r}dt^2)^T&\text{if }\alpha=1\\
		(if',-\frac{a'}{r}dt^1)^T&\text{if }\alpha=2
	\end{cases}
\end{equation*}
and the claim is proved.

\medskip

Let us denote $Z_{U_0}\coloneqq \ker\sdf{L}=\mathrm{span}\{\sdf{V}_1,\sdf{V}_2\}$. As shown in \cite{stuart1994dynamics}, it holds the coercivity estimate on $Z_{U_0}^\perp$
\begin{equation}\label{coercivity}
	\langle \sdf{L}[\Phi],\Phi\rangle_{L^2}\geq c\norm{\Phi}^2_{H^{1}_{U_0}},\quad\forall\Phi\in Z_{{U_0}}^\perp
\end{equation}
for some $c>0$.
Lastly, we remark that the coercivity estimate \eqref{coercivity} yields directly, via Lax--Milgram Theorem, to a solvability theory for the equation 
\begin{equation}\label{lstreq}
	L[\Phi]=\Psi\quad \text{on }\R^2
\end{equation}
for a right-hand side $\Psi\in Z_{{U_0}}^\perp$. Precisely it holds the following result.
\begin{lemma}\label{invertibility lstr}
	Let $\Psi\in L^2(\R^2)$ satisfy
	\begin{equation}\label{lemma1 ort cond}
		\int_{\R^2}\Psi(t)\cdot\sdf{V}_\alpha(t) dt=0,\quad \alpha=1,2,
	\end{equation} 
	where $\sdf{V}_\alpha$ is as in \eqref{expression sdfVj}. Then there exists a unique solution $\Phi\in H^1_{U_0}(\R^2)$ to problem \eqref{lstreq} satisfying 
	\begin{equation}
		\int_{\R^2}\Phi(t)\cdot\sdf{V}_\alpha(t) dt=0,\quad \alpha=1,2.
	\end{equation}
	Besides,
	\begin{equation*}
		\|\Phi\|_{H^1_{U_0}(\R^2)}\leq C\|\Psi\|_{L^2(\R^2)}.
	\end{equation*}
	for some $C>0$.
\end{lemma}

\section{Outline of the proof}\label{outline}

In this section we sketch the proof of Theorems \ref{Less general theorem} and  \ref{main theorem} and explain the main ideas.  We follow the lines of  \cite{del2006variational, del2011giorgi, del2013entire,del2015serrin, pacard2003constant, pacard2012linear, ting2013multi}, which consists of finding a precise global approximation starting from a non-degenerate lower dimensional solution and then we build an actual solution as a perturbation of such approximation. A major difference in this case is that the equation for the perturbation is a fixed point formulation whose main term has an infinite-dimensional kernel due to gauge invariance.

\medskip

Firstly, we define the approximate solution, which in our case will be a pair denoted $W$. This will encode alone the local behaviour \eqref{aaa} around the vortex set. By embeddedness, any point sufficiently close to $M$ can be written in Fermi coordinates
\begin{equation*}
	x=X(y,z)=y+z_1\nu_1(y)+z_2\nu_2(y),\quad y\in M,\ |z|<\delta.
\end{equation*}
We choose as a local approximation
\begin{equation*}
	W_0(y,z)=U_0\left(z/\eps-h(y)\right),
\end{equation*}
where we left implicit the composition with $X$ and $h\colon M\to \R^2$ is a parameter, which depends on $\eps$. A precise calculation of the error of approximation $S(W_0)$, performed in \S\ref{Size S(W0)}, suggests that a better approximation is obtained by setting $W_1=W_0+\eps^2\Lambda$, where $\Lambda$ is a decaying term. Using the minimality of $M$, we compute
\begin{equation}\label{expansion S(W_1)}
	S(W_1)=\eps^2\J[h]\sdf{V}_\alpha(t)+\eps^3(k_1^3+k_2^3)t_1^2\sdf{V}_\alpha(t)+\mathsf{Rem}
\end{equation}
where $k_1,k_2$ are the principal curvatures of $M$ and $\mathsf{Rem}=O(\eps^3e^{-\sigma|t|})$. Here $t=z/\eps-h(y)$ and $\J$ is the Jacobi operator of $M$, namely the second variation of the area functional. This approximation will be good enough for our purposes. It remains to extend it beyond the support of the Fermi coordinates. For a suitable cut-off function $\zeta$ supported in a neighbourhood of $M$, we set
\begin{equation*}
	W=\zeta W_1+(1-\zeta)\boldsymbol{\Psi}
\end{equation*}
where $\boldsymbol{\Psi}$ is a pure gauge, carefully chosen in order to create only very small extra terms in the global error $S(W)$.

\medskip

Next, we set up the perturbation scheme. We seek a true solution of the form $W+\varPhi$, which is equivalent to solving 
\begin{equation}\label{sketch fuleq}
	S(W+\vphi)=0\quad \text{in }\R^4,
\end{equation}
where $S$ is given by \eqref{error_function_S}.
We write, equivalently
\begin{equation}\label{sketch S'+E+N}
	S'(W)[\vphi]=-E-N(\vphi)
\end{equation}
where $E=S(W)$ is the error of approximation, $S'(W)$ is the linearised operator of \eqref{error_function_S} and
\begin{equation*}
	N(\vphi)=S(W+\vphi)-S(W)-S'(W)[\vphi]
\end{equation*} 
is formally a quadratic operator in $\vphi$. Starting from \eqref{sketch S'+E+N}, it is evident that finding a bounded (in a suitable topology) inverse to the linearised operator $S'(W)$ will allow us rephrase the equation for $\vphi$ as a fixed point problem. However, as already mentioned, gauge-invariance makes the invertibility of $S'(W)$ challenging. Rather than this, we develop an invertibility theory for
\begin{equation*}
	S'(W)+\Theta_W\Theta_W^*=L_W
\end{equation*}
which is the ``gauge-corrected'' linearised. More precisely, the key equation to be solved is 
\begin{equation}\label{linearised fixed point}
	L_W[\vphi]+E+N(\vphi)=0\quad \text{in }\R^4.
\end{equation}
By doing so, for such $\vphi$ the it will hold
\begin{equation}\label{solution with correction}
	S(W+\vphi)=-\Theta_W\Theta_W[\vphi]
\end{equation}
So the price to pay for this correction is the presence of a right-hand side in \eqref{sketch fuleq}. Actually, as we will see shortly, the natural degeneracies of the manifold due to ambient isometries prevent, in principle, the possibility of solving \eqref{linearised fixed point} as it is. Again, the equation is solvable up to a correction term on the right-hand side, which will be added to the one already present in \eqref{solution with correction}. Lastly, we will show \emph{a posteriori} that the natural symmetries of the solution constructed will make all these correction terms vanish, thus finding a true solution.

\medskip

The main technical result of the paper is the invertibility theory for $L_W$. The proof is based on the fact that, close to the manifold, it holds
\begin{align*}
	L_{W}[\vphi]=\sdf{L}[\vphi]-\eps^2\Delta_{M}\vphi+O(\eps^2D\vphi)
\end{align*}
where $D$ is a first order operator and we denoted $\Delta_{M}(\varphi,\eta)=(\Delta_{M}\varphi,\Delta_{H,M}\eta)$, being $\Delta_{H,M}=d^*d+dd^*$ the Hodge Laplacian on $M$. Here , $\sdf{L}$ is the 2-dimensional linearised given by \eqref{Lstraight} and it is independent by the $y$ variable. From this expression, it can be seen the role of the Jacobi operator in the resolution of \eqref{linearised fixed point}: we can formally write it as 
\begin{equation}\label{sketch eq to solve}
	\sdf{L}[\vphi]=-E-N(\vphi)+\eps^2\Delta_{M}\vphi+O(\eps^2D\vphi).
\end{equation}
We can obtain an improvement of the approximation by solving $\sdf{L}[\vphi]=-E$. According to Lemma \ref{invertibility lstr}, this is doable if for every $y\in M$
\begin{equation}\label{orthogonality condition E}
		\int_{\R^2}E(y,t)\cdot\sdf{V}_\beta(t)=0,\quad \beta=1,2.
	\end{equation}
Using \eqref{expansion S(W_1)}, \eqref{orthogonality condition E} becomes 
\begin{equation*}
	\eps^2\J[h](y)\int_{\R^2}|\sdf{V}_\alpha(t)|^2dt+\eps^3(k^3_1(y)+k^3_2(y))\int_{\R^2}t_1^2|\sdf{V}_\alpha(t)|^2dt+\int_{\R^2}\sdf{Rem}\cdot\sdf{V}_\alpha(t)dt=0
\end{equation*}
 An equation of this form can indeed be solved, using the nondegeneracy hypothesis of the Jacobi operator, up to adding corrections terms proportional to the Jacobi fields \eqref{jacobi fields}. 
 
\medskip

To sum up, the proof is divided into the following steps:
\begin{itemize}
	\item Define an approximation to a solution $W_0$, locally around $M$, using the canonical profile $U_0$.
	\item Add a correction term of order $\eps^2$ to $W_0$, obtaining an improved approximation $W_1$.
	\item Glue $W_1$ to a suitable pure gauge defined away from $M$, to get a global approximation $W$.
	\item After formulating the problem as a fixed point, solve using the invertibility theory of the gauge-corrected linearised, up to a correction term that depends on $h$.
	\item Use the nondegeneracy hypothesis for the Jacobi operator to find a suitable perturbation $h$ that cancels the correction term above. We can do so up to adding new correction terms.
	\item Show that the corrections automatically vanish.
\end{itemize}

\section{Norms}
In this section we introduce several norms which will be used throughout the rest of the paper. For $0<\gamma<1$ we denote
\begin{equation}
	\norm{\varphi}_{C^{0,\gamma}(\Lambda)}\coloneqq \norm{\varphi}_{L^\infty(\Lambda)}+\sup_{p\in\Lambda}\ [\varphi]_{\gamma,B(p,1)}
\end{equation}
where 
\begin{equation}\label{holder seminorm}
	[\varphi]_{\gamma,X}\coloneqq\sup_{x,y\in X,x\ne y}\frac{|\varphi(x)-\varphi(y)|}{|x-y|^\gamma}
\end{equation}
is the usual H\"older box. The $C^{k,\gamma}$ norm is defined as 
\begin{equation}\label{norm Ck}
	\norm{\varphi}_{C^{k,\gamma}(\Lambda)}\coloneqq\norm{\varphi}_{C^{0,\gamma}(\Lambda)}+\|D^k\varphi\|_{C^{0,\gamma}(\Lambda)}.
\end{equation}
In what follows we will need to measure the size and decay of three types of functions, namely those defined on $M$, those defined on the product $M\times\R^2$ and those defined on $\R^4$.
\subsection*{Norms on $M$} Let $\phi$ be a function defined on $M$. We want to define a norm to account the decay of $\phi$ along the manifold. We defined the map
\begin{align*}
	r(y)&=\sqrt{1+|\iota(y)|^2},\quad y\in M
\end{align*}
where $\iota$ is the embedding $M\to\R^3$. We set
\begin{equation*}
	\|\phi\|_{C^{0,\gamma}_\mu(M)}\coloneqq \|r^\mu\phi\|_{C^{0,\gamma}(M)}
\end{equation*}
and, for $k\geq 1$,
\begin{equation}\label{norm Ckgammamu}
	\norm{\phi}_{C^{k,\gamma}_{\mu}(M)}\coloneqq\norm{\phi}_{C^{0,\gamma}_{\mu}(M)}+\|D^k\phi\|_{C^{0,\gamma}_{\mu}(M)}.
\end{equation}
\subsection*{Norms on $M\times\R^2$}
We will use a weighted version of \eqref{norm Ck}. For $\mu\geq0$ and $\sigma\in (0,1)$, let 
\begin{equation}\label{norm C0gammamusigma}
	\norm{\varphi}_{C^{0,\gamma}_{\mu,\sigma}(M\times\R^2)}\coloneqq\|r^\mu e^{\sigma|t|}\varphi\|_{C^{0,\gamma}(M\times\R^2)}
\end{equation}
and, for any $k\geq 1$,
\begin{equation}\label{norm Ckgammamusigma}
	\norm{\varphi}_{C^{k,\gamma}_{\mu,\sigma}(M\times\R^2)}\coloneqq\norm{\varphi}_{C^{0,\gamma}_{\mu,\sigma}(M\times\R^2)}+\|D^k\varphi\|_{C^{0,\gamma}_{\mu,\sigma}(M\times\R^2)}
\end{equation}
We will use the norms \eqref{norm C0gammamusigma} and \eqref{norm Ckgammamusigma} indistinctly for functions $\varphi=\varphi(y,x)$ defined on $M\times\R^2$ and for functions $\varphi=\varphi(X_{h}(y,x))$, possibly defined in the whole space $\R^4$ but supported in a region where Fermi coordinates are defined, and hence understood as defined on $M\times\R^2$.
\subsection*{Norms on $\R^4$}
Finally we introduce a standard weighted norm in the whole space $\R^4$, simply by setting 
\begin{equation*}
	\|\psi\|_{C^{0,\gamma}_\mu(\R^4)}\coloneqq\|(1+|x|)^\mu\psi\|_{C^{0,\gamma}(\R^4)}
\end{equation*}
and
\begin{equation*}
	\|\psi\|_{C^{k,\gamma}_\mu(\R^4)}\coloneqq\|\psi\|_{C^{0,\gamma}_\mu(\R^4)}+\|D^k\psi\|_{C^{0,\gamma}_\mu(\R^4)}
\end{equation*}
where $\psi$ is defined in $\R^4$ and $\mu\geq 0$.

\section{The approximate solution}\label{section the first app}

Now we start carrying out the details of the proof, beginning with the construction of the first, local approximation $W_0$. In what follows, $M\subset\R^4$ is a two dimensional complete, minimal surface with finite total curvature, which is non-degenerate in the sense of \eqref{nondeg}. We consider the case in which $M$ is embedded in $\R^3$ and that the subsequent immersion in $\R^4$ is done in the canonical way. 

\medskip

 In the following calculations we will always adopt the following convention unless differently specified: Latin letters classically used for indexing  $(i,j,k,\ldots)$ will be used for tangential coordinates to the manifold, while Greek letters $(\alpha,\beta,\gamma,\ldots)$
 will denote coordinates in the normal direction. If needed, we will use the first letters of the Latin alphabet $(a,b,c\ldots)$ 
 to indicate all coordinates at once.



\subsection{First local approximation}
We parametrise a neighbourhood of $M$ in $\R^4$ as follows: let $(\nu_1,\nu_2)\coloneqq (\nu,{\bf e}_4)$ and let $X\colon\mathcal{O}\to \N$, given by
\begin{equation}
	X(y,z)=y+z^\beta\nu_\beta(y)
\end{equation}
be the Fermi coordinates around $M$. We can choose
\begin{equation*}
	\mathcal{O}=\braces{\round{y,z}\in M\times\R^2\ :\ |z|<\tau(y)}
\end{equation*}
where $\tau : M\to \R^+$ is given by 
\begin{equation*}
	\tau(y)=\delta\log(1+ r(y)),\quad  r(y)=\sqrt{1+|\iota(y)|^2},
\end{equation*}
being $\iota:M\to\R^4$ the immersion, and 
\begin{equation}
    \mathcal{N}\coloneqq X(\mathcal{O}).
\end{equation}
Let now $h=(h^1,h^2)$ be a vector-valued function defined on $M$ and consider the change of coordinates
\begin{equation*}
	z=\eps(t+h(y)),
\end{equation*}
which generates a new parametrisation $X_{h}\colon\mathcal{O}_h\to \mathcal{N}$ given by 
\begin{equation}\label{X_epsh}
 	X_{h}(y,t)=y+\eps(t^\beta+h^\beta(y))\nu_\beta(y)
 \end{equation}
 where 
 \begin{equation*}
		\mathcal{O}_{h}=\braces{\round{y,t}\in M\times\R^2\ :\ |t+h(y)|<\tau(y)/\eps}.
\end{equation*}
Observe that, in principle, with this choice of $\tau$ the map $X$ is not necessarily one-to-one. On the other hand, in \S\ref{preciseassumh} below we will make a precise assumption on $h$ that will guarantee injectivity for $X_h$.
We define the first approximation to be
\begin{equation}
	W_0(y,t)\coloneqq {U_0}(t)
\end{equation}
 where $U_0=(u_0,A_0)$ is given by \eqref{radialansatz} and we left implicit the composition with the chart $X_{h}$. 
 
 \subsection{Precise assumptions on $h$}\label{preciseassumh}
 Here we state precisely what the assumptions on the vertical perturbation $h$ are. As we have seen in Section \ref{outline} the step of making the projections vanish will amount to solve a system having as a principal part the Jacobi operator of $h$, which in our case reads
\begin{equation}\label{jacobi in codim 2}
	\mathcal{J}[h]=\twovec{\Delta_Mh^1+|\A_M|^2h^1}{\Delta_Mh^2}
\end{equation}
and such system will be solved via a fixed point argument in a space of functions whose size is small in $\eps$. This procedure is not doable in a straightforward way as the right-hand side of the fixed point is not automatically small. 

\medskip

This issue has already been addressed in \cite[\S3.3]{del2013entire} and is due to the fact that if two consecutive ends of $M$ are parallel, the error of approximation created in the region between the two ends is very small but does not decay with the distance from the origin, and eventually dominates the whole right-hand side.

\medskip

This issue can be solved by adjusting the parameter $h$. Consider any $m$-tuple of real numbers $(\lambda_1,\dots,\lambda_m)$ satisfying
\begin{equation}\label{balancing condition}
	\lambda_1\leq \lambda_2\leq\cdots\leq\lambda_m,\qquad\sum_{j=1}^m\lambda_j=0.
\end{equation} 
We have the following result, the proof of which is postponed to Section \ref{Proof_left}.
\begin{lemma}\label{behaviour h0}
	For any real numbers $\lambda_1,\dots,\lambda_m$ satisfying \eqref{balancing condition} there exists a smooth function $h_0^1$, defined on $M$, such that
	\begin{equation*}
		\Delta_Mh_0^1+|\A_M|^2h_0^1=0\quad\text{on }M
	\end{equation*}
	and such that on each end $M_j$
	\begin{equation*}
		h_0^1(y)=(-1)^j\lambda_j\log r+\eta\quad \text{on }M_j,
	\end{equation*}
	where $\eta$ satisfies 
	\begin{equation*}
		\|\eta\|_{L^{\infty}(M)}+\|r^2D\eta\|_{L^{\infty}(M)}<+\infty.
	\end{equation*}
\end{lemma}
We remark that the gap condition \eqref{divendcond} is not necessary for the proof of Lemma \ref{behaviour h0} but will be useful later to obtain estimates.

\medskip

At this point we make the following assumption: we write 
\begin{equation*}
	h\coloneqq h_0+h_1.
\end{equation*}
For any vector $\lambda$ satisfying \eqref{balancing condition} let $h_0^1$ be the function predicted by Lemma \ref{behaviour h0} and let $h_0\coloneqq (h_0^1,0)$.
As explained in \cite[\S3.3]{del2013entire}, such choice of $h_0^1$ will let the ends of the $h$-perturbed manifold drift apart fast enough to prevent the creation of the non-decaying error term. On the other hand, we assume for $h_1$ that for some $C>0$
\begin{equation}
	\|h_1\|_*\coloneqq \|h_1\|_\infty+\|D_Mh_1\|_{C^{0,\gamma}_2(M)}+\|D^2_Mh_1\|_{C^{0,\gamma}_4(M)}\leq C\eps.
\end{equation}

\subsection{The size of the error $S(W_0)$}\label{Size S(W0)}
In order to measure how good $W_0$ is as an approximation we compute the error of approximation $S(W_0)$ in $\N$, where $S$ is given by \eqref{error_function_S}.
We begin by expressing the differential operators, namely the connection Laplacian and $d^*d$, in coordinates $X_{h}$. 

\medskip

Given a 1-form $\omega=\omega_adx^a$ the general coordinate expression for $-\Delta^\omega$ is 
\begin{equation*}
	-\Delta^\omega=\frac{1}{\sqrt{\det g}}\partial^\omega_a\round{\sqrt{\det g}g^{ab}\partial^\omega_b}
\end{equation*}
being $\partial_a^\omega=\partial_a-i\omega_a$ the components of the connection gradient $\nabla^\omega$. Similarly, the operator $d^*d$ acting on 1-forms can be written in the following way: let $\omega_{ab}\coloneqq\partial_a\omega_b-\partial_b\omega_a$. Then
\begin{equation*}
	d^*d\omega=-\frac{1}{\sqrt{\det g}}g_{ab}\partial_{c}\left(\sqrt{\det g}g^{dc}g^{eb}\omega_{de}\right)dx^{a}.
\end{equation*}
We will first write the operators in coordinates $X$. First, remark that the Euclidean metric on $\N$ can be expressed as a function of Fermi coordinates as a block matrix
\begin{equation*}
	g=\begin{pmatrix}g_{z} & 0\\
0 & I_{2}
\end{pmatrix},
\end{equation*}
where $I_2$ is the $2\times 2$ identity matrix. This is a consequence of the well-known corresponding expression in codimension 1 (see, for instance, \cite[Lemma 11.1]{del2015serrin}) and our choice of embedding in $\R^4$. In particular, we can write
\begin{equation*}
	g^{ij}=g_z^{ij},\quad g^{j\alpha}=0,\quad g^{\alpha\beta}=\delta_{\alpha\beta}, \quad \text{on }\N
\end{equation*}
where $g^{ij}_z$ is the metric restricted to the manifold
\begin{equation}\label{Mz}
	M_z=\braces{y+z^1\nu(y)\ :\ y\in M}.
\end{equation}
We introduce the functions
\begin{align*}
	a^{ij}&=g^{ij},\\ b_{s}^{ik}&=\frac{1}{\sqrt{\det g}}g_{st}\partial_{j}\left(g^{ij}g^{kt}\sqrt{\det g}\right),\\
	c^{k}&=\frac{1}{\sqrt{\det g}}\partial_{j}\left(\sqrt{\det g}g^{jk}\right),\\ d_{j}^{\beta k}&=\frac{1}{\sqrt{\det g}}g_{ij}\partial_{\beta}\left(g^{ik}\sqrt{\det g}\right),
\end{align*}
and let 
\begin{equation*}
	H_z^\beta=-\frac{1}{\sqrt{\det g}}\partial_{\beta}\left(\sqrt{\det g}\right)
\end{equation*}
be the mean curvature in the direction of $\nu_\beta$. Observe that by the flatness of the immersion, $H_z^2\equiv 0$. The mean curvature vector $H_z$ can be expanded in terms of the principal curvatures of $M$ as
\begin{equation*}
	H_z(y)=H^1_z(y)\nu(y)=\left(\sum_{\ell=1}^2\frac{k_{\ell}(y)}{1-z^1 k_{\ell}(y)}\right)\nu(y),
\end{equation*}
see \cite{del2011giorgi}. In particular, we write
\begin{equation}
	H^1_z(y)=\sum_{k=1}^\infty(z^1)^{j-1}H_j(y),\quad H_j(y)\coloneqq k_1^j+k_2^j.
\end{equation}
We will use a truncated expansion
\begin{equation}\label{expansion meancurvature}
	H_z^1(y)=z^1|\A_M|^2+(z^1)^2H_3(y)+\bar{H}(y)
\end{equation}
where we used the minimality of $M$.
See for instance \cite{del2011giorgi}.
Let us consider local coordinates for $M$ around a generic point $p$, namely
\begin{equation*}
	y=\varphi_p(\xi_1,\xi_2),\quad |\xi|<\rho.
\end{equation*}
Using these, we end up with the following coordinate expression for the operators 
\begin{align}
	-\Delta^\omega\phi&=-a^{ij}\partial^\omega_{ij}\phi-c^j\partial^\omega_j\phi-\partial^\omega_{\alpha\alpha}\phi+H_z^\alpha\partial^\omega_\alpha\phi\label{expansion delta}\\
	d^{*}d\omega&=-a^{ij}\partial_{j}\omega_{ik}d\xi^{k}-b_{k}^{ij}\omega_{ij}d\xi^{k}-\partial_{\beta}\omega_{\beta\gamma}dz^{\gamma}+H_{{z}}^{\beta}\omega_{\beta\gamma}dz^{\gamma}\label{expansion d*d}\\
	&-a^{ij}\partial_{j}\omega_{i\gamma}dz^{\gamma}-c^{i}\omega_{i\gamma}dz^{\gamma}-\partial_{\beta}\omega_{\beta k}d\xi^{k}-d_{k}^{\beta j}\omega_{\beta j}d\xi^{k}\nonumber
\end{align}
where we recall that all coefficients are evaluated at $(\xi,z)$. Here we are abusing again the notation, leaving implicit the composition with the chart. We can also expand
\begin{align*}
	a^{ij}(\xi,z)&=a^{ij}_{0}(\xi)+z^\beta a^{ij}_{1,\beta}(\xi,z)\\
	c^j(\xi,z)&=c^j_{0}(\xi)+z^\beta c^j_{1,\beta}(\xi,z)
\end{align*}
and obtain
\begin{equation}\label{expansion laplacian}
\begin{split}
	a^{ij}\partial_{ij}+c^j\partial_j&=\Delta_M+\eps(t^\beta+h^\beta)\round{a^{ij}_{1,\beta}\partial_{ij}+c^j_{1,\beta}\partial_j}\\
	&\eqqcolon\Delta_M+\eps(t^\beta+h^\beta)D_{1,\beta}.
\end{split}
\end{equation}
Next we analyse how these expansions transform when we switch to $X_{h}$. In what follows we will denote $h_j^\beta=\partial_jh^\beta$, $h_{ij}^\beta=\partial_{ij}h^\beta$ and so on. Changing coordinates yields, for a function $\phi=\phi(\xi,t)$ and a 1-form $\omega=\omega_{k}(\xi,t)d\xi^k+\omega_\alpha(\xi,t)dt^\alpha$, to 
\begin{align*}
	-\eps^2\Delta^{\omega}\phi&=-\eps^2\Delta_{M_{z}}^{\omega}\phi-\partial_{\alpha\alpha}^{\omega}\phi+\eps H_{{z}}^{\alpha}\partial_{\alpha}^{\omega}\phi+\eps^2(\Delta_{M_{z}}h^{\beta})\partial_{\beta}^{\omega}\phi\\
	&+\eps a^{jk}\left[h_{k}^{\beta}\partial_{j\beta}^{\omega}\phi+h_{j}^{\gamma}\partial_{\gamma k}^{\omega}\phi\right]-\eps^2a^{jk}h_{j}^{\gamma}h_{k}^{\beta}\partial_{\gamma\beta}^{\omega}\phi.
\end{align*}
and
\begin{align*}
	\eps^2 d^{*}d_{\eps,h}\omega&=-a^{ij}\Big(\eps^2\partial_{j}\omega_{ik}-\eps h_{j}^{\beta}\partial_{\beta}\omega_{ik}-\eps h_{i}^{\beta}\partial_{j}\omega_{\beta k}-\eps^2h_{ji}^{\beta}\omega_{\beta k}\\
	&-\eps^2h_{jk}^{\beta}\omega_{i\beta}+\eps^3h_{i}^{\gamma}h_{jk}^{\beta}\omega_{\gamma\beta}+\eps^3h_{i}^{\beta}h_{j}^{\gamma}\partial_{\gamma}\omega_{\beta k}\Big)d\xi^{k}\\
	&-\eps b_{k}^{ij}\left(\eps^2\omega_{ij}-\eps h_{i}^{\beta}\omega_{\beta j}-\eps h_{j}^{\beta}\omega_{i\beta}+\eps^2h_{i}^{\gamma}h_{j}^{\beta}\omega_{\gamma\beta}\right)d\xi^{k}-\partial_{\beta}\omega_{\beta k}d\xi^{k}\\
	&-\eps^2 d_{k}^{\beta j}\left(\omega_{\beta j}- h_{j}^{\gamma}\omega_{\beta\gamma}\right)d\xi^{k}+\eps^2h_{k}^{\gamma}H_z^{\beta}\omega_{\beta\gamma}d\xi^{k}-\eps^2h_{k}^{\gamma}c^{i}\left(\omega_{i\gamma}-\eps h_{i}^{\beta}\omega_{\beta\gamma}\right)d\xi^{k}\\
	&-\partial_{\beta}\omega_{\beta\gamma}dt^{\gamma}+\eps H_{z}^{\beta}\omega_{\beta\gamma}dt^{\gamma}-\eps c^{i}\left(\omega_{i\gamma}-\eps h_{i}^{\beta}\omega_{\beta\gamma}\right)dt^{\gamma}\\
	&-a^{ij}\left(\eps\partial_{j}\omega_{i\gamma}-\eps h_{j}^{\beta}\partial_{\beta}\omega_{i\gamma}-\eps^2h_{ij}^{\beta}\omega_{\beta\gamma}-\eps h_{i}^{\beta}\partial_{j}\omega_{\beta\gamma}+\eps^2h_{i}^{\beta}h_{j}^{\delta}\partial_{\delta}\omega_{\beta\gamma}\right)dt^{\gamma}
\end{align*}
where all coefficients above have to be evaluated either at $(\xi)$ or $(\xi,t+h)$. A useful remark which applies to our case is that if the 1-form $\omega$ is purely orthogonal to $M$, in the sense that
$\omega=\omega_\alpha(t)dt^\alpha$, then only the terms of the form $\omega_{\beta\gamma}$ don't vanish and hence
\begin{align*}
	\eps^2d^{*}d\omega&=\left(-\partial_{\beta}\omega_{\beta\gamma}+\eps H_{{z}}^{\beta}\omega_{\beta\gamma}+\eps^2(\Delta_{M_{z}}h^{\beta})\omega_{\beta\gamma}-\eps^2a^{ij}h_{i}^{\beta}h_{j}^{\delta}\partial_{\delta}\omega_{\beta\gamma}\right)dt^{\gamma}\\
	&+\Big(-\eps^{3}a^{ij}h_{i}^{\gamma}h_{jk}^{\beta}\omega_{\gamma\beta}-\eps^{3}b_{k}^{ij}h_{i}^{\gamma}h_{j}^{\beta}\omega_{\gamma\beta}\\
	&+\eps^3d_{k}^{\beta j}h_{j}^{\gamma}\omega_{\beta\gamma}+\eps^2h_{k}^{\gamma}H_z^{\beta}\omega_{\beta\gamma}+\eps^3h_{k}^{\gamma}c^{i}h_{i}^{\beta}\omega_{\beta\gamma}\Big)d\xi^{k}.
\end{align*}
Evaluating $S(W_0)$ using the formulas just found and the fact that ${U_0}$ solves \eqref{full_equations}, we find 
\begin{equation}\label{error of U0 non expanded}
	S(W_0)=\eps^2(\Delta_{M_{z}}h^{\beta})\sdf{V}_\beta+\eps H^\beta_z \sdf{V}_\beta-\eps^2a^{ij}h_i^\beta h_j^\gamma\nabla_{\beta\gamma,{U_0}}{U_0}+R_\xi(h)
\end{equation}
where $\sdf{V}_\beta(t)$ is as in \eqref{expression sdfVj},  
\begin{equation}\label{nablabetagamma}
	\nabla_{\beta\gamma,{U_0}}{U_0}\coloneqq\nabla_{U_0}(\nabla_{U_0}{U_0}(e_\beta))(e_\gamma)
\end{equation} 
being $\nabla_{U_0}$ as in \eqref{gradlike},
and  
\begin{equation*}
	R(h)=O(\eps^2(h+\nabla h)\nabla_{\beta\gamma,{U_0}}{U_0}).
\end{equation*}
Now, let $\sigma\in (0,1)$. Using \eqref{expansion meancurvature} and \eqref{expansion laplacian}, we can write the following expansion of the first error of approximation
\begin{equation}\label{S(W_0)}
	\begin{split}
		S(W_0)&=\varepsilon^{2}(\Delta_{M}h_{1}^{1}+|\A_{M}|^{2}h_{1}^{1})\mathsf{V}_{1}+\varepsilon^{2}(\Delta_{M}h_{1}^{2})\mathsf{V}_{2}\\
	&+\varepsilon^{2}t^{1}\left|\A_{M}\right|^{2}\mathsf{V}_{1}+\varepsilon^{3}(t^{1}+h^{1})^{2}H_{3}\mathsf{V}_{1}+\varepsilon\bar{H}\mathsf{V}_{1}\\
	&+\varepsilon^{3}(t^{\beta}+h^{\beta})\left(D_{1,\beta}h^{1}\right)\mathsf{V}_{1}+\varepsilon^{3}(t^{\beta}+h^{\beta})\left(D_{1,\beta}h^{2}\right)\mathsf{V}_{2}\\
	&-\varepsilon^{2}a^{ij}_0\partial_{i}h_{0}^{\beta}\partial_{j}h_{0}^{\gamma}\nabla_{\beta\gamma,{U_0}}{U_0}-2\varepsilon^{2}a^{ij}_0\partial_{i}h_{0}^{\beta}\partial_{j}h_{1}^{\gamma}\nabla_{\beta\gamma,{U_0}}{U_0}\\
	&-\varepsilon^{2}a^{ij}_0\partial_{i}h_{1}^{\beta}\partial_{j}h_{1}^{\gamma}\nabla_{\beta\gamma,{U_0}}{U_0}-\varepsilon^{2}(t^\delta+h^\delta)a_{1,\delta}^{ij}\partial_{i}h_{0}^{\beta}\partial_{j}h_{0}^{\gamma}\nabla_{\beta\gamma,{U_0}}{U_0}\\
	&-2\varepsilon^{2}(t^\delta+h^\delta)a_{1,\delta}^{ij}\partial_{i}h_{0}^{\beta}\partial_{j}h_{1}^{\gamma}\nabla_{\beta\gamma,{U_0}}{U_0}-\varepsilon^{2}(t^\delta+h^\delta)a_{1,\delta}^{ij}\partial_{i}h_{1}^{\beta}\partial_{j}h_{1}^{\gamma}\nabla_{\beta\gamma,{U_0}}{U_0}\\
	&+O(\eps^3r^{-4}e^{-\sigma|t|})
	\end{split}
\end{equation}
where $\nabla_{\beta\gamma,{U_0}}$ is as in \eqref{nablabetagamma}.


\subsection{Improvement of approximation}\label{improvement subsub} The term \begin{equation}\label{to cancel}
	\mathcal{E}(y,t)=\varepsilon^{2}|\A_M(y)|^2t^{1}\sdf{V}_{1}(t) -\varepsilon^{2}a^{ij}_0(y)\partial_{i}h_{0}^{\beta}(y)\partial_{j}h_{0}^{\gamma}(y)\nabla_{\beta\gamma,{U_0}}{U_0}(t)
\end{equation} 
appearing in \eqref{S(W_0)}
is the largest term of the error not contributing to the projections, in the sense that 
\begin{equation}\label{ort mathcalE}
	\int_{\R^2}\mathcal{E}(y,t)\cdot\sdf{V}_\alpha(t) dt=0.
\end{equation} 
Next step will be to improve the approximation $W_0$ in order to eliminate \eqref{to cancel}. While it may not be strictly necessary, including it makes the expansions of both $h$ and the remainder $\varPhi$ more precise. The solvability conditions are indeed automatically satisfied in the terms quadratic in $\eps$.

\medskip

Thus we improve the approximation, and we do so by setting 
\begin{equation*}
	W_1(y,t)\coloneqq W_0(t)+\eps^2 \Lambda(y,t).
\end{equation*}
We remark that all terms $W_0$, $W_1$ and $\Lambda$ are defined only in a region close to $M$ as they are all defined through Fermi coordinates. The corresponding error can be written as
\begin{align*}
	S(W_1)&=S(W_0)+\eps^2S'(W_0)[\Lambda]+N_{0}(\eps^2\Lambda)\\
	&=S(W_0)+\eps^2\sdf{L}[\Lambda]+\eps^2(S'(W_0)-\sdf{L})[\Lambda]+N_{0}(\eps^2\Lambda)\quad \text{in }\N.
\end{align*}
We recall that $\sdf{L}$ is the two-dimensional linearised around ${U_0}$ given by \eqref{Lstraight}, and we defined 
\begin{equation*}
	N_0(\eps^2\Lambda)\coloneqq S(W_0+\eps^2\Lambda)-S(W_0)-\eps^2S'(W_0)[\Lambda].
\end{equation*} 
If we choose $\Lambda$ such that 
\begin{equation*}
	\sdf{L}[\Lambda]=-|\A_M(\eps y)|^2t^1\sdf{V}_1+a_0^{ij}(\eps y)\partial_{i}h_{0}^{\beta}(\eps y)\partial_{j}h_{0}^{\gamma}(\eps y)\nabla_{\beta\gamma,{U_0}}{U_0}
	\end{equation*}
then the term \eqref{to cancel} in the error of approximation will be erased. It is natural to look for a solution of the form 
\begin{equation*}
	\Lambda(y,t)=-|\A_M(\eps y)|^2\Lambda_1(t)+a_0^{ij}\partial_{i}h_{0}^1(\eps y)\partial_{j}h_{0}^{1}(\eps y)\Lambda_{\beta\gamma}(t)
\end{equation*}
where 
\begin{equation}\label{IoA equations}
	\sdf{L}[\Lambda_1]=t^1\sdf{V}_1(t),\quad \sdf{L}[\Lambda_{\beta\gamma}]=\nabla_{\beta\gamma,{U_0}}{U_0}(t)
\end{equation}
The existence of such solutions is given by Lemma \ref{invertibility lstr}, 
given that the right-hand sides satisfy the orthogonality conditions
\begin{equation*}
	\int_{\R^2}t^1\sdf{V}_1(t)\cdot\sdf{V}_\alpha(t) dt=0,\quad \int_{\R^2}\nabla_{\beta\gamma,{U_0}}{U_0}(t)\cdot\sdf{V}_\alpha(t) dt=0, \quad \alpha=1,2.
\end{equation*} 
Also, since the right hands sides in  \eqref{IoA equations} are both $O(e^{-|t|})$ for $|t|$ large a standard barrier argument along with the fact that $\sdf{L}\sim -\Delta+\mathrm{Id}$ at infinity ensures that 
\begin{equation*}
	\sup_{t\in\R^2}e^{\sigma|t|}|\Lambda(y, t)|<\infty,\quad \forall y\in M,
\end{equation*}
for any $0<\sigma<1$.
In this way, the biggest part of $S(W_0)$ not contributing to the projections is erased. We can estimate the error created in $\mathcal{N}$
\begin{align*}
	\eps^2\left(S'(W_0)-\mathsf{L}\right)[\Lambda]+N_0(\eps^2\Lambda)
	&=-\varepsilon^{4}\big(|\A_M|^2\nabla_{1,{U_0}}\Lambda_{1}-a_{0}^{ij}\partial_{i}h_{0}^{\beta}\partial_{j}h_{0}^{\gamma}\nabla_{1,{U_0}}\Lambda_{\beta\gamma}\big)(\Delta_{M}h_{1}^{1})\\
	&\ -\varepsilon^{4}\big(|\A_M|^2\nabla_{2,{U_0}}\Lambda_{1}-a_{0}^{ij}\partial_{i}h_{0}^{\beta}\partial_{j}h_{0}^{\gamma}\nabla_{2,{U_0}}\Lambda_{\beta\gamma}\big)(\Delta_{M}h_{1}^{2})+R_{0}
\end{align*}
with 
\begin{equation*}
	\left|R_{0}(y,t)\right|\leq C\varepsilon^{3}r(y)^{-4}e^{-|t|}.
\end{equation*}
Thus, the error $S(W_1)$ can be written as
\begin{equation}\label{S(U1)}
	\begin{split}
		S(W_1)
	&=\varepsilon^{2}(\Delta_{M}h_{1}^{1}+|\A_{M}|^{2}h_{1}^{1})\mathsf{V}_{1}+\varepsilon^{2}(\Delta_{M}h_{1}^{2})\mathsf{V}_{2}\\
	&-\varepsilon^{4}\big(|\A_M|^2\nabla_{1,{U_0}}\Lambda_{1}-a_{0}^{ij}\partial_{i}h_{0}^{\beta}\partial_{j}h_{0}^{\gamma}\nabla_{1,{U_0}}\Lambda_{\beta\gamma}\big)(\Delta_{M}h_{1}^{1})\\
	&-\varepsilon^{4}\big(|\A_M|^2\nabla_{2,{U_0}}\Lambda_{1}-a_{0}^{ij}\partial_{i}h_{0}^{\beta}\partial_{j}h_{0}^{\gamma}\nabla_{2,{U_0}}\Lambda_{\beta\gamma}\big)(\Delta_{M}h_{1}^{2})\\
	&-2\varepsilon^{2}a^{ij}_0\partial_{i}h_{0}^{\beta}\partial_{j}h_{1}^{\gamma}\nabla_{\beta\gamma,{U_0}}{U_0}-\varepsilon^{2}a^{ij}_0\partial_{i}h_{1}^{\beta}\partial_{j}h_{1}^{\gamma}\nabla_{\beta\gamma,{U_0}}{U_0}\\
	&-\varepsilon^{3}(t^\delta+h^\delta)a_{1,\delta}^{ij}\partial_{i}h_{0}^{\beta}\partial_{j}h_{0}^{\gamma}\nabla_{\beta\gamma,{U_0}}{U_0}-2\varepsilon^{3}(t^\delta+h^\delta)a_{1,\delta}^{ij}\partial_{i}h_{0}^{\beta}\partial_{j}h_{1}^{\gamma}\nabla_{\beta\gamma,{U_0}}{U_0}\\
	&-\varepsilon^{3}(t^\delta+h^\delta)a_{1,\delta}^{ij}\partial_{i}h_{1}^{\beta}\partial_{j}h_{1}^{\gamma}\nabla_{\beta\gamma,{U_0}}{U_0}+R_1(y,t,h,\nabla_{M}h)
	\end{split}
\end{equation}
where
\begin{equation*}
	\left|R_{1}(y,t,\iota,\varsigma)\right|+\left|\partial_{\iota}R_{1}(y,t,\iota,\varsigma)\right|+\left|\partial_{\varsigma}R_{1}(y,t,\iota,\varsigma)\right|\leq C\varepsilon^{3}\left(1+r(y)\right)^{-4}e^{-\sigma|t|}.
\end{equation*}
and we used that $\J[h_0]=0$.
Using \eqref{estimate A_M} we can estimate term by term the error \eqref{S(U1)}. For instance 
\begin{align*}
	\abs{\varepsilon^{2}(\Delta_{M}h_{1}^{1}+|A_{M}|^{2}h_{1}^{1})\mathsf{V}_{1}}&\leq C\eps^2\|D^2 h_1\|_{C^{0,\gamma}_4(M)}\|e^{\sigma|t|}\sdf{V}_1\|_{L^{\infty}(\R^2)}r^{-4}e^{-\sigma|t|}\\
	&\ +C\eps^2\|h_1\|_{\infty}\|e^{\sigma|t|}\sdf{V}_1\|_{L^\infty(\R^2)}r^{-4}e^{-\sigma|t|}
\end{align*}
or in other words 
\begin{equation*}
	\norm{\varepsilon^{2}(\Delta_{M}h_{1}^{1}+|A_{M}|^{2}h_{1}^{1})\mathsf{V}_{1}}_{C^{0,\gamma}_{4,\sigma}(M\times\R^2)}\leq C\eps^3.
\end{equation*}
Similar calculation can be carried out for all the other terms in the expression of $S(W_1)$. In all we get
\begin{equation}
	\norm{S(W_1)}_{C^{0,\gamma}_{4,\sigma}(M\times\R^2)}\leq C\eps^3
\end{equation}
Now that we know precisely the size of all the terms involved in the error \eqref{S(U1)}, all that is left to obtain an actual approximation is to extend it beyond the support of the Fermi coordinates. 

\subsection{The global approximation}\label{firstglobapp}
The approximation obtained so far is sufficient for our purposes when considered in a neighbourhood of $M$. Next step is to extend $W_1$ beyond the support of the Fermi coordinates in a way that keeps the error $S$ small in a norm that accounts decay along the ends.
We begin by considering an extension of the set $\N$ previously defined
where the Fermi coordinates are still well-defined and the error of approximation maintains the same size. Let 
\begin{equation*}
	\mathcal{O}'=\braces{(y,t)\in M\times\R^2\ :\ |t+h_1(y)|<\varrho(y)/\eps}
\end{equation*}
where 
\begin{equation*}
	\varrho(y)\coloneqq \delta+\frac{4\eps}\sigma \log(1+r(y))
\end{equation*}

and $\delta$ is a small positive number. Remark that $\eps|t+h_1|=|z-\eps h_0|$ where $z$ is the normal coordinate to the manifold. This means that setting
\begin{equation}
    \N'=X_{h}(\mathcal{O}'),
\end{equation}
where $X_{h}$ is given by \eqref{X_epsh}, then $\N'$ describes a tubular (expanding) neighbourhood centred on the $h_0$-shifted manifold
\begin{equation}\label{M_zero}
	M_0\coloneqq \braces{y+\eps h^1_0(y)\nu(y)\ :\ y\in M}
\end{equation}
but expanding quicker than $\N$ along the ends. Observe that $X_{h}$ is one-to-one in $\mathcal{O}'$ thanks to assumption \eqref{divendcond} and the fact that $h_1=O(\eps)$, which means that Fermi coordinates are still well defined in $\N'$. 

\medskip

Let now $\zeta$ be a smooth cut-off function such that $\zeta(s)=1$ if $s<1$ and $\zeta(s)=0$ if $s>2$. For $\delta>0$, we define 
\begin{equation}
	\zeta_\delta(x)\coloneqq 
	\begin{cases}
		\zeta(|t+h_1(y)|-\varrho(y)/\eps-3)&\text{if }x=X_{h}(y,t)\in\mathcal{N}'\\
		0&\text{otherwise.}
	\end{cases}
\end{equation}
We will define the approximate global solution as
\begin{equation*}
	W\coloneqq \zeta_\delta W_1+(1-\zeta_\delta)\boldsymbol{\Psi}
\end{equation*}
where $\boldsymbol{\Psi}=(\psi,d\psi/i\psi)$ is a pure gauge of the form \eqref{pure gauges v2}, for some $\mathbb{S}^1$-valued function $\psi$ which is smooth in the support of $(1-\zeta_\delta)$. This ensures that away form the manifold $W$ is a solution, precisely
\begin{equation}\label{Psi pure gauge}
	S(W)=S(\boldsymbol{\Psi})=0\quad\text{on }\ \{\zeta_\delta = 0\}
\end{equation}
independently on the choice of $\psi$, as long as it is regular enough.
On the other hand, we need to choose the function $\psi$ such that $\boldsymbol{\Psi}$ glues well with $W_1$ on the set $\{0<\zeta_\delta<1\}$, to avoid the creation of a large error. Let us recall that 
\begin{equation*}
	W_{1}(y,t)=\begin{pmatrix}f(\rho)e^{i\theta}\\
a(\rho)d\theta
\end{pmatrix}+O(\eps^{2}r(y)^{-4}e^{-\sigma|t|})
\end{equation*}
where $(\rho,\theta)$ are the polar coordinates relative to $(t^1,t^2)$. Also, remark that as $\rho=|t|$ grows large it holds
\begin{equation}\label{estim1-f1-a}
	1-f(\rho)=O(e^{-\rho}),\quad 1-a(\rho)=O(e^{-\rho})
\end{equation} 
therefore it looks natural to choose $\psi=e^{i\theta}$ (and consequently $d\psi/i\psi=d\theta$) in $\N'\setminus M_0$. In such case it holds 
\begin{equation*}
	|W_1(t)-\boldsymbol{\Psi}(t)|=|1-f(|t|)|+|t|\cdot|1-a(|t|)|+O(\eps^{2}r(y)^{-4}e^{-\sigma|t|})\quad \text{on }\ \{0<\zeta_\delta <1\}.
\end{equation*}
Now, the points in $\{0<\zeta_\delta <1\}$ are of the form $X_{h}(y,t)$, where $(y,t)$ satisfies
\begin{equation*}
	|t|>\frac{\delta}{\eps}+\frac{4}{\sigma}\log r(y)
\end{equation*}
for $\eps$ sufficiently small. Therefore

\begin{equation}\label{expest}
	e^{-\sigma|t|}\leq e^{-\frac{\sigma\delta}{\eps}+\log r^{-4}}\leq Cr^{-4}e^{-\frac{\delta\sigma}{\eps}}\quad\text{on }\ \{0<\zeta_\delta <1\}
\end{equation}
and hence exponentially decaying terms in $|t|$ gain extra decay in $r$ and exponential smallness in $\eps$ when considered in $\{0<\zeta_\delta <1\}$. We claim that the interpolation error term is exactly of this type. A lengthy but straightforward calculation gives 
\begin{equation*}
	S(W)=\zeta_\delta S(W_1)+(1-\zeta_\delta)S(\boldsymbol{\Psi})+\mathrm{E}
\end{equation*}
where, if we denote $W_1=(u_1,A_1)$ and $\mathrm{E}=(e,E)$, it holds 
\begin{align*}
	e&=i\eps^2\zeta_{\delta}(1-\zeta_{\delta})d^{*}(A_{1}-d\theta)(u_{1}-e^{i\theta})+i\eps^2d\zeta_{\delta}\cdot(A_{1}-d\theta)(\zeta_{\delta}u_{1}+(1-\zeta_{\delta})e^{i\theta})\\
	&-2i\eps^2\zeta_{\delta}(1-\zeta_{\delta})(A_{1}-d\theta)\cdot d(u_{1}-e^{i\theta})+\eps^2\zeta_{\delta}(1-\zeta_{\delta})^{2}|(A-d\theta)|^{2}u_{1}\\
	&+\eps^2(1-\zeta_{\delta})\zeta_{\delta}^{2}|(A_{1}-d\theta)|^{2}e^{i\theta}-2\eps^2\zeta_{\delta}(1-\zeta_{\delta})d\theta\cdot(A_{1}-d\theta)e^{i\theta}\\
	&+2\eps^2\zeta_{\delta}(1-\zeta_{\delta})A_{1}\cdot(A_{1}-d\theta)u_{1}-\eps^2(u_{1}-e^{i\theta})\Delta\zeta_{\delta}\\
	&-2\eps^2\nabla^{A_{1}}u_{1}\cdot d\zeta_{\delta}-2i\eps^2((1-\zeta_{\delta})u_{1}+\eps^2\zeta_{\delta}e^{i\theta})(A_{1}-d\theta)\cdot d\zeta_{\delta}\\
	&+\tfrac{1}{2}\eps^2\zeta_{\delta}(1-\zeta_{\delta})((1-\zeta_{\delta})|u_{1}-e^{i\theta}|^{2}+2\langle u_{1},u_{1}-e^{i\theta}\rangle)u\\
	&+\tfrac{1}{2}\eps^2\zeta_{\delta}(1-\zeta_{\delta})(\zeta_{\delta}|u_{1}-e^{i\theta}|^{2}-2\langle v,u_{1}-e^{i\theta}\rangle)e^{i\theta}
\end{align*}
and 
\begin{align*}
	E&=\eps^2d^{*}\left(d\zeta_{\delta}\wedge(A_{1}-d\theta)\right)-\eps^2d\zeta_{\delta}\cdot d(A_{1}-d\theta)+\eps^2\zeta_{\delta}(1-\zeta_{\delta})\left\langle \nabla^{A_{1}}u_{1},iu_{1}\right\rangle \\
	&-\eps^2(1-\zeta_{\delta})d\zeta_{\delta}\langle u_{1}-e^{i\theta},iv\rangle -\eps^2\zeta_{\delta}d\zeta_{\delta}\langle u_{1}-e^{i\theta},iu_{1}\rangle \\
	&-\eps^2\zeta_{\delta}(1-\zeta_{\delta})\langle \nabla^{A_{1}}u_{1},ie^{i\theta}\rangle -\eps^2\zeta_{\delta}(1-\zeta_{\delta})(A_{1}-d\theta)\langle u_{1},e^{i\theta}\rangle.
\end{align*}
As a general estimate, we can write 
\begin{equation*}
	|\mathrm{E}|\leq C\eps^2 \chi_{\{0<\zeta_\delta<1\}}\round{1+|W_1|}^2\round{\abs{W_1-\boldsymbol{\Psi}}+\abs{D(W_1-\boldsymbol{\Psi})}+|\nabla^{A_1}u_1|}.
\end{equation*}
Using \eqref{estim1-f1-a}, the fact that $f',a'=O(e^{-\rho})$ and that 
\begin{equation*}
	|\nabla^{A_1}u_1(t)|=2f'(|t|)+O(\eps^{2}r(y)^{-4}e^{-\sigma|t|})
\end{equation*}
we obtain
\begin{align*}
	|\mathrm{E}|&\leq C\eps^2\chi_{\{0<\zeta_\delta<1\}} e^{-\sigma|t|}\leq Cr^{-4}e^{-\frac{\delta\sigma}{\eps}}\quad\text{on }\ \{0<\zeta_\delta <1\}
\end{align*}
where we used \eqref{expest}. Hence, the interpolation error has the required smallness. 

\medskip

Finally, we define the pure gauge $\boldsymbol{\Psi}$ in the following way: consider the signed distance function from $M_0$, where $M_0$ is given by \eqref{M_zero}, given by 
\begin{equation*}
	\bar{\mathtt{d}}:\{\zeta_\delta>0\}\cap \R^3\to \R,\quad \bar{\mathtt{d}}(x')\coloneqq s(x')\mathrm{dist}(x',M_0)
\end{equation*}
where $s(x')=+1$ if $x'$ is in the region where $\nu$ points, and $-1$ otherwise. Let now $\mathtt{d}:\R^3\setminus M_0\to \R$ be any smooth extension of $\bar{\mathtt{d}}$ to $\R^3\setminus M_0$. We set $\psi\colon\R^3\setminus M_0\to \mathbb{S}^1$ as
\begin{equation*}
	\psi(x',x_4)\coloneqq \frac{\mathtt{d}(x')+ix_4}{|\mathtt{d}(x')+ix_4|}
\end{equation*}
In this way, the pair $\boldsymbol{\Psi}=(\psi,d\psi/i\psi)$ satisfies our requirements in the sense that 
\begin{equation*}
	\boldsymbol{\Psi}\big\vert_{\{0<\zeta_\delta<1\}}=\twovec{e^{i\theta}}{d\theta}
\end{equation*}
and it extends smoothly on $\{\zeta_\delta=0\}$.

\section{Proof of main result}\label{proofmainres}
In the previous section we built an approximate solution $W$ of the form
\begin{equation*}
	W=\zeta_\delta W_1+(1-\zeta_\delta)\boldsymbol{\Psi}.
\end{equation*}
We now look for a solution of the form $W+\vphi$, where $\vphi$ will be small in a suitable topology. In other words, we want to solve
\begin{equation}\label{pertubed_equation_S(W+Ups)=0}
	S(W+\vphi)=0.
\end{equation} 
We can rephrase \eqref{pertubed_equation_S(W+Ups)=0}, using \eqref{decomposition_of_L_cmp}, as 
\begin{equation}\label{projected_equation_compact}
	S(W+\vphi)=-\Theta_W\Theta_W^*[\vphi]+L_W[\vphi]+S(W)+N(\vphi)=0
\end{equation}
where
\begin{equation*}
	N(\vphi)=S(W+\vphi)-S(W)-S'(W)[\vphi].
\end{equation*} 
As explained in Section \ref{outline}, we will solve \eqref{projected_equation_compact} by first finding a solution $\vphi$ to 
\begin{equation}\label{lin - uncorrected}
	L_W[\vphi]+S(W)+N(\vphi)=0,
\end{equation}
which we will be able to do up to corrections (see Proposition \ref{invertibility L_W} below), and then showing that all corrections (including $-\Theta_W\Theta_W^*[\Phi]$) vanish thanks to the symmetries of the solution found. 

Rather than \eqref{lin - uncorrected}, we will solve the more general problem 
\begin{equation}\label{Perturb eq not inverted - corrected}
	L_W[\vphi]+S(W)+N(\vphi)=\zeta_2(y,t)b^\alpha(y)\sdf{V}_\alpha(t).
\end{equation}
where $\zeta_2$ is a cut-off function supported close to $M$ (defined below in \eqref{zeta_m}), $\sdf{V}_\alpha$ are as in \eqref{expression sdfVj} and $b^\alpha$ are functions defined on $M$ which are unknowns of the problem.
This correction is necessary to obtain good estimates for the solution. The principal operator in \eqref{Perturb eq not inverted - corrected}, $L_W$, has an approximate kernel given by $\sdf{V}_\alpha(t)$, $\alpha=1,2$, suitably cut-off away from the manifold $M$. The correction $\zeta_2b^\alpha(y)\sdf{V}_\alpha(t)$ has the role of ensuring that the right-hand side of the equation satisfies an orthogonality condition to the approximate kernel, yielding the \emph{a priori} estimates. 

\medskip

We solve \eqref{Perturb eq not inverted - corrected} by first proving the following result for its linear version. The adjustment on the right-hand side provides unique solvability in terms of $\vphi$ and $b=(b^1,b^2)$. In the sense of the following Proposition, which will be proved in Section \ref{Proof of prop inv LW}.
\begin{proposition}\label{invertibility L_W}
	Let $\mu\geq 0$ and $\gamma\in(0,1)$. For any $\Lambda\in C^{0,\gamma}_{\mu}(\R^4)$ there exists $b\in C^{0,\gamma}_\mu(M)$ and a unique solution $\vphi=\mathcal{G}(\Lambda)$ to 
	\begin{equation}\label{L_W equation}
		L_W[\vphi]=\Lambda+\zeta_2b^\alpha(y)\sdf{V}_\alpha(t)
	\end{equation}
	satisfying 
	\begin{equation*}
		\norm{\vphi}_{C^{2,\gamma}_{\mu}(\R^4)}+\norm{b}_{C^{0,\gamma}_\mu(M)}\leq C\norm{\Lambda}_{C^{0,\gamma}_{\mu}(\R^4)}
	\end{equation*}
	for some $C>0$.
\end{proposition}
Using Proposition \ref{invertibility L_W} we can write \eqref{Perturb eq not inverted - corrected} as a fixed point problem 
\begin{equation*}
	\vphi=-\mathcal{G}\round{S(W)+N(\vphi)}\eqqcolon\mathcal{P}[\vphi]
\end{equation*}
on the space
\begin{equation}\label{space for Upsilon}
	X_A=\braces{\vphi\in C^{2,\gamma}_{4}(\R^4)\ :\ \norm{\vphi}_{C^{2,\gamma}_{4}(\R^4)}\leq A\eps^3}
\end{equation}
Indeed, if $\vphi_1,\vphi_2\in X_A$ (see \eqref{N lip claim 1} below),
\begin{align*}
	\norm{N(\vphi_1)-N(\vphi_2)}_{C^{0,\gamma}_{4}(\R^4)}\leq C\eps^3\norm{\vphi_1-\vphi_2}_{C^{0,\gamma}_{4}(\R^4)}
\end{align*}
which implies
\begin{align*}
	\norm{\mathcal{P}[\vphi_1]-\mathcal{P}[\vphi_2]}_{C^{2,\gamma}_{4}(\R^4)}&\leq C \norm{N(\vphi_1)-N(\vphi_2)}_{C^{0,\gamma}_{4}(\R^4)}\leq C\eps^3\norm{\vphi_1-\vphi_2}_{C^{0,\gamma}_{4}(\R^4)}.
\end{align*}
Hence, up to enlarging $A$ enough, by contraction mapping principle we can find an unique $\vphi\in C^{2,\gamma}_{4}(\R^4)$ and with 
\begin{equation}\label{estimate vphi}
	\norm{\vphi}_{C^{2,\gamma}_{4}(\R^4)}\leq A\eps^3
\end{equation}
such that $W+\vphi$ is a solution of the corrected problem \eqref{Perturb eq not inverted - corrected}. 
An important fact about such $\vphi$ that will be proved in Section \ref{Proof_left} is its Lipshitz dependence on $h_1$, namely
\begin{equation*}
		\|\vphi(h_1)-\vphi(h_2)\|_{C_{4}^{2,\gamma}(\mathbb{R}^{4})}\leq C\eps^2\|h_1-h_2\|_*.
\end{equation*}  
It only remains to make this correction vanish, and to do so we will adjust the parameter $h_1$.

\section{Adjusting $h_1$ to make the projection vanish}\label{adj h1}

In this section we prove that there is a suitable choice of the function $h_1$ such that the quantities $b^\alpha(h_1)$ in \eqref{Perturb eq not inverted - corrected} vanish, up to a further correction term accounting for the degeneracies of the Jacobi operator. 
So far we have solved 
\begin{equation}\label{already solved}
	L_W[\vphi]+S(W)+N(\vphi)=\zeta_2b^\gamma(y)\sdf{V}_\gamma(t)\quad\text{on }\R^4
\end{equation}
Where $\zeta_2$ is defined as follows: let $\zeta$ be a cut-off function such that $\zeta(s)=1$ if $s<1$ and $\zeta(s)=0$ if $s>2$. For every positive integer $m$, define
\begin{equation}\label{zeta_m}
	\zeta_m(x)=
	\begin{cases}
		\zeta(|t+h_1(y)|-\delta/\eps-m)&\text{if }x=X_{h}(y,t)\in\mathcal{N},\\
		0&\text{otherwise}.
	\end{cases}
\end{equation} 
If we multiply \eqref{already solved} by $\zeta_4(y,t)\sdf{V}_\alpha(t)$ and integrate on $\R^2$, we find an expression for $b^\alpha(y)$,
\begin{equation}\label{ba}
	b^\alpha(y)=\frac{1}{\int_{\R^2}\zeta_2(y,t)|\sdf{V}_\alpha(t)|^2dt}\int_{\R^2}\zeta_4(y,t)\squared{S(W)+N(\vphi)+L_W[\vphi]}\cdot\sdf{V}_\alpha(t)dt.
\end{equation}
Now we recall that inside the support of $\zeta_4$ the quantity $S(W)$ equals the local error of approximation $S(W_1)$, which has the expression \eqref{S(U1)}. We define 
\begin{equation*}
	q_{m}(y)=\int_{\R^2}\zeta_m(y,t)|\sdf{V}_\alpha(t)|^2dt,\quad m=1,2,\dots
\end{equation*}
which is independent of $\alpha\in\{1,2\}$ and satisfies $q_m(y)\to \|\sdf{V}_\alpha\|_{L^2}^2$ as $\eps\to 0$. Then
\begin{align*}
	\int_{\R^2}\squared{\zeta_4S(W_1)}\cdot\sdf{V}_1(t)dt&=\eps^2q_{4}(\Delta_Mh_1^1+|\A_M|^2h_1^1)-\eps^2G^1_1(h_1),\\
	\int_{\R^2}\squared{\zeta_4S(W_1)}\cdot\sdf{V}_2(t)dt&=\eps^2q_{4}(\Delta_Mh_1^2)-\eps^2G^1_2(h_1).
\end{align*}
We also split the remaining operator in the following way
\begin{align*}
	G^2_\alpha(h_1)&\coloneqq -\frac{1}{\eps^2}\int_{\R^2}\zeta_4\squared{N(\vphi)+\sdf{B}[\vphi]}\cdot\sdf{V}_\alpha(t)dt,\\
	G^3_\alpha(h_1)&\coloneqq -\frac{1}{\eps^2}\int_{\R^2}\zeta_4L_{U_0}[\vphi]\cdot\sdf{V}_\alpha(t)dt,
\end{align*}
where $\sdf{B}$ is defined (inside the support of $\zeta_4$) as 
\begin{equation*}
	\sdf{B}[\Phi]\coloneqq \round{L_W-L_{U_0}}[\Phi].
\end{equation*}
Also, we set
\begin{equation}\label{G}
	G=\twovec{G_1}{G_2}=q_{4}^{-1}\sum_{k=1}^3\twovec{G^k_1}{G^k_2}.
\end{equation}
With this notation, the system $b^1=b^2=0$ can be rephrased as a fixed point problem 
\begin{equation}\label{fixed point jacobi}
	\J[h_1]=G(h_1)
\end{equation}
where $\J$ is the Jacobi operator \eqref{jacobi in codim 2} and as we will see $G$ is small and satisfies a suitable Lipschitz property (see Lemma \ref{Lipsh} below). Equation \eqref{fixed point jacobi} reads explicitly 
\begin{equation*}
	\begin{cases}
		\Delta_Mh_1^1+|\A_M|^2h_1^1=G_1(h_1),\\
		\Delta_Mh_1^2=G_2(h_1).
	\end{cases}
\end{equation*}
We recall that the assumption of non-degeneracy of the manifold $M$ implies that all bounded Jacobi fields are linear combinations of those generated by the rigid motions of the manifold. In our case the Jacobi operator is decoupled, this gives 5 independent Jacobi fields, 4 of which come from those of the immersion into $\R^3$, namely
\begin{align*}
	\z_j=\twovec{z_j}{0},\quad j=0,1,2,3.
\end{align*}
\begin{gather*}
	z_i(y)=\nu_1(y)\cdot e_i\quad i=1,2,3,\\
	z_0(y)=\nu_1(y)\cdot\round{-y_2,y_1,0,0}.
\end{gather*}
while the immersion in $\R^4$ produces 
\begin{equation*}
	\z_4=\twovec{0}{z_4},\quad z_4(y)=\nu_2(y)\cdot e_4=1
\end{equation*}
which accounts for the constants in the kernel of $\Delta_M$. In what follows we consider linear combinations
\begin{equation*}
	\hat{\z}_j=\sum_{i=0}^4 d_{ij}\z_i,\quad j=0,\dots,4,
\end{equation*} 
such that 
\begin{equation*}
	\int_M|\A_M|^2\hat{\z}_i\hat{\z}_j=\delta_{ij},\quad i,j=0,\dots,4.
\end{equation*}
The existence of a non-trivial kernel requires the presence of a correction in an invertibility theory for the Jacobi operator \eqref{fixed point jacobi}, of the form 
\begin{equation}\label{corrected jacobi system}
	\J(h_1)=G(h_1)-\sum_{j=0}^4c^j|\A_M|^2\hat{\z}_j.
\end{equation}
Indeed, a suitable choice of the $c^j$'s ensures that the right-hand side is orthogonal to the kernel, namely 
\begin{equation*}
	\int_M\Bigg(G(h_1)-\sum_{j=0}^4c^j|\A_M|^2\hat{\z}_j\Bigg)\cdot\z^k=0,\quad k=0,\dots,4,
\end{equation*}
allowing for an invertibility theory that carries \emph{a priori} estimates in the sense of the following result, the proof of which is postponed to Section \ref{invertjacob}.
\begin{proposition}\label{invertibility Jacobi}
	Let $f=(f^1,f^2)$ be a vector-valued function defined on $M$ such that $\|f\|_{C^{0,\gamma}_{4}(M)}<+\infty$.
Then there exist constants $c^0,\dots,c^4$ such that system \eqref{corrected jacobi system}, namely
	\begin{equation*}
		\begin{cases}
		\Delta_Mh^1+|\A_M|^2h^1=f^1-\sum_{j=0}^3c^j|\A_M|^2\hat{z}_j,\\
		\Delta_Mh^2=f^2-c^4|\A_M|^2,
	\end{cases}
	\end{equation*}
admits a solution $h=(h^1,h^2)=\mathcal{H}(f)$ satisfying 
\begin{equation*}
	\|h\|_*\leq C\|f\|_{C^{0,\gamma}_{4}(M)}.
\end{equation*}
\end{proposition}
\begin{remark}
    It might happen, for instance in the case where $M$ is a catenoid, that the Jacobi field $\z_0$ associated to rotation invariance is 0. In this case the orthogonality condition is automatically satisfied and we do not need an extra correction term. 
\end{remark}
Proposition \ref{invertibility Jacobi} allows us to rewrite problem \eqref{corrected jacobi system} as a fixed point problem
\begin{equation}\label{fixed point for h1}
	h_1=\mathcal H(G(h_1)).
\end{equation}
In order to find such an $h_1$ it suffices to show that the right-hand side is a contraction mapping. This follows from the following Lemma, which will be proved in Section \ref{Proof_left}.
\begin{lemma}\label{Lipsh}
	The map $G$ satisfies 
	\begin{equation*}
		\|G(0)\|_{C^{0,\gamma}_{4}(M)}\leq C\eps
	\end{equation*} 
	and the Lipschitz condition
	\begin{equation*}
		\|G(h_1)-G(h_2)\|_{C^{0,\gamma}_{4}(M)}\leq C\eps\|h_1-h_2\|_*.
	\end{equation*}
\end{lemma}
By the contraction mapping principle, combining the estimates provided by Proposition \ref{invertibility Jacobi} and Lemma \ref{Lipsh}, \eqref{fixed point for h1} admits an unique solution in the space 
\begin{equation*}
	B_K=\braces{h\in C^0(M)\ :\ \|h\|_*\leq K\eps}
\end{equation*}
and hence we found a solution of the corrected problem \eqref{corrected jacobi system}. 

\medskip

In the next Section we will show that the above corrections, together with the one coming from gauge invariance, are actually zero.

\section{Conclusion of the proof of Theorem \ref{main theorem}}
      
So far we have constructed a solution $U=W+\vphi$ of 
\begin{equation*}
	S(U)=-\zeta_4b^\alpha(y)\sdf{V}_\alpha(t)-\Theta_W\Theta_W^*[\vphi]
\end{equation*}
where we managed to adjust the parameter $h$ to obtain 
\begin{equation*}        
	b^1=-\eps^2q_{2}^{-1}q_4\sum_{j=0}^3c^j|\A_M|^2\hat{z}_j,\quad b^2=-\eps^2q_{2}^{-1}q_4c^4|\A_M|^2.
\end{equation*}
Let us define $q=\eps^2\zeta_4q_{2}^{-1}q_4$ and set 
\begin{equation*}
	\gamma\coloneqq -\Theta_W^*[\vphi]
\end{equation*}
so that $U$ satisfies 
\begin{equation}\label{expression S(U) with corrections}
S(U)=q|\A_M|^2\sum_{j=0}^3c^j\hat{z}_j\sdf{V}_1(t)+q|\A_M|^2c^4\sdf{V}_2(t)+\Theta_W[\gamma].
\end{equation}
Our claim is that the coefficient vector $\boldsymbol{c}=(c^0,\dots,c^4)$ and the function $\gamma$ are automatically 0, which will make $U$ the sought solution and conclude the proof. Consider the quantities
\begin{align*}
	Z_i&=\nabla_{x_i,U}U,\quad i=1,2,3,4\\
	Z_0&=x_1\nabla_{x_2,U}U-x_2\nabla_{x_1,U}U.
\end{align*}
We claim that 
\begin{align}
	Z_j&=\sdf{V}_1z_j+O(\eps r^{-1}),\quad j=0,1,2,3\label{expansion Zj}\\
	Z_4&=\sdf{V}_2+O(\eps r^{-1}).\label{expansion Z4}
\end{align}
Indeed, recall that in a region close to the manifold the solution $U$ so found satisfies
\begin{equation*}
	U(x)=U_0(t)+\Phi(y,t)
\end{equation*}
for some function $\Phi$ satisfying in this region
\begin{equation*}
	|\Phi(y,t)|+|D\Phi(y,t)|\leq C\eps^2r(y)^{-4}e^{-\sigma|t|}
\end{equation*}
and where $x=y+\eps(t^\beta+h^\beta)\nu_\beta$, with $|t|$ sufficiently small. An explicit calculation yields to 
\begin{equation*}
	\nabla_UU=\sdf{V}_\alpha\cdot \eps\nabla t^\alpha+\sdf{R}
\end{equation*}
where $|\sdf{R}|\leq C\eps^2r(y)^{-4}e^{-\sigma|t|}$.
Since $\eps t^\alpha = z^\alpha -\eps h^\alpha(y)$ and $\nabla z^\alpha=\nu_\alpha$ it holds that 
\begin{equation*}
	\nabla t^\alpha = \nu_\alpha-\eps\nabla h^\alpha(y),
\end{equation*}
then, recalling that on each end $M_k$ it holds 
\begin{gather*}
	\nabla h^1=(-1)^k\frac{\eps\lambda_k}{r}\hat{r}+O(\eps r^{-2})\\
	\nabla h^2=O(\eps r^{-2})
\end{gather*}
we obtain the following local behaviour for $Z_i$, $i=1,\dots,4$,
\begin{align*}
	Z_i&=\nabla_{x_i,U}U\\
	&=\mathsf{V}_{\alpha}(t)\left(\nu_{\alpha}-\varepsilon^2\nabla h^{\alpha}\left(y\right)\right)\cdot e_i+\sdf{R}_i\\
	&=\mathsf{V}_{1}(t)\left(\nu_{1}\cdot e_{i}-\varepsilon^2\left(-1\right)^{k}\frac{\lambda_{k}}{r}\hat{r}\cdot e_{i}+O\left(\varepsilon r^{-2}\right)\right)\\
	&\ +\mathsf{V}_{2}(t)\left(\nu_{2}\cdot e_{i}+O\left(\varepsilon^2 r^{-2}\right)\right)+\sdf{R}_i.
\end{align*}
which proves \eqref{expansion Zj}-\eqref{expansion Z4}, using also that $Z_0=z_1Z_2-z_2Z_1.$ 

\medskip

The next Lemma follows from gauge invariance, the invariances of $M$ under rigid motions and from the balancing condition \eqref{balancing condition}. The proof is postponed to Section \ref{Proof_left}.
\begin{lemma}\label{orthogonality Z_i}
	It holds
	\begin{equation*}
		\int_{\R^4}S(U)\cdot\Theta_U[\gamma]=0\quad\text{and}\quad\int_{\R^4}S(U)\cdot Z_i=0
	\end{equation*}
	for $i=0,\dots,4$.
\end{lemma}
To prove that the vector $(\boldsymbol{c},\gamma)$ vanishes we will show that it is mapped to zero by a positive, linear operator. Using Lemma \ref{orthogonality Z_i} and \eqref{expression S(U) with corrections} we obtain 
\begin{equation}\label{dotprod cj}
\begin{split}
	0&=\int_{\R^4}S(U)\round{\sum_{l=0}^4d_{li}Z_l} \\
	&=\sum_{j=0}^3c^j\int_{\R^4}q|\A_M|^2\hat{z}_j\round{\sum_{l=0}^4d_{li}Z_l}\sdf{V}_1(t)dt+c^4\int_{\R^4}q|\A_M|^2\hat{z}_4\round{\sum_{l=0}^4d_{li}Z_l}\sdf{V}_2(t)dt\\
	&+\sum_{l=0}^4d_{li}\int_{\R^4}\Theta_W[\gamma]Z_l
	\end{split}
\end{equation}
for $i=0,\dots,4$, and 
\begin{align*}
	\begin{split}
		0&=\int_{\R^4}S(U)\Theta_U[\gamma]\\
		&=\sum_{j=0}^3c^j\int_{\R^4}q|\A_M|^2\hat{z}_j\Theta_U[\gamma]\sdf{V}_1+c^4\int_{\R^4}q|\A_M|^2\hat{z}_4\Theta_U[\gamma]\sdf{V}_2\\
		&+\int_{\R^4}\Theta_W[\gamma]\Theta_U[\gamma].
	\end{split}
\end{align*}
Using the expansions \eqref{expansion Zj}-\eqref{expansion Z4} we find
\begin{align*}
	\int_{\R^4}q|\A_M|^2\hat{z}_j\round{\sum_{l=0}^4d_{li}Z_l}\sdf{V}_1(t)dt&=\int_{\R^2}|\sdf{V}_1(t)|^2\int_Mq|\A_M|^2\hat{z}_i\hat{z}_j+o(\eps)\\
	\int_{\R^4}q|\A_M|^2\hat{z}_4\round{\sum_{l=0}^4d_{li}Z_l}\sdf{V}_2(t)dt&=\int_{\R^2}|\sdf{V}_2(t)|^2\int_Mq|\A_M|^2\hat{z}_4\hat{z}_4+o(\eps).
\end{align*}
Moreover, it holds (recalling that $U=W+\vphi$) 
\begin{equation*}
\begin{split}
	\int_{\R^4}\Theta_W[\gamma]\Theta_{U}[\gamma]&=\int_{\R^4}\eps^2|\nabla\gamma|^2+\langle iw,i(w+\varphi)\rangle\gamma^2\\
	&=\int_{\R^4}\gamma [(\eps^2\Delta+|w|^2+\langle w,\varphi\rangle)\gamma]
\end{split}
\end{equation*}
where we denoted with $w$ and $\varphi$ respectively the firsts components of $W$ and $\vphi$, and we used the decay of $\gamma$ to justify the integration by parts. Observe that since $\vphi$ is small (in $\eps$) with respect to $W$, the operator 
\begin{equation*}
	\eps^2\Delta+|w|^2+\langle w,\varphi\rangle
\end{equation*}
is positive for $\eps$ small enough. 
Lastly, using the decay of $\gamma$ and $|\A_M|^2$ together with the fact that $\Theta_{U_0}^*[\sdf{V}_\alpha]=0$, it is direct to see that
\begin{gather*}
	\int_{\R^4}\Theta_W[\gamma]Z_l=o(\eps)\\
	\int_{\R^4}q|\A_M|^2\hat{z}_j\Theta_U[\gamma]\sdf{V}_1(t)dt=o(\eps)\\
	\int_{\R^4}q|\A_M|^2\hat{z}_4\Theta_U[\gamma]\sdf{V}_2(t)dt=o(\eps).
\end{gather*}
This computations show that an equation of the form $\mathcal{L}(\boldsymbol{c},\gamma)=0$ holds, where $\mathcal{L}$ is a linear operator which can be expressed as a small perturbation of the positive operator
\begin{equation*}
	(\boldsymbol{c},\gamma)\mapsto (\|\sdf{V}_\alpha\|_{L^2}^2\boldsymbol{c},(-\eps^2\Delta+|w|^2+\langle w,\varphi\rangle)\gamma) 
\end{equation*}
and, as a consequence, we have $(\boldsymbol{c},\gamma)=0$. We have thus found a solution to the system \eqref{Energy_4d} with the properties required from Theorem \ref{main theorem}.
\qed

\section{Invertibility theory for the gauge-corrected linearised}
\label{Proof of prop inv LW}
In this Section we prove Proposition \ref{invertibility L_W}, of which we recall the statement.

\medskip

	\noindent {\bf Proposition \ref{invertibility L_W}}.\ \emph{Let $\mu\geq 0$ and $\gamma\in(0,1)$. For any $\Lambda\in C^{0,\gamma}_{\mu}(\R^4)$ there exists a $b\in C^{0,\gamma}_\mu(M)$ and a unique solution $\vphi=\mathcal{G}(\Lambda)$ to 
	\begin{equation}\label{L_W equation}
		L_W[\vphi]=\Lambda+\zeta_2b^\alpha(y)\sdf{V}_\alpha(t)
	\end{equation}
	satisfying 
	\begin{equation*}
		\norm{\vphi}_{C^{2,\gamma}_{\mu}(\R^4)}+\norm{b}_{C^{0,\gamma}_\mu(M)}\leq C\norm{\Lambda}_{C^{0,\gamma}_{\mu}(\R^4)}
	\end{equation*}
	for some $C>0$.}

\medskip

To prove Proposition \ref{invertibility L_W} we develop an invertibility theory for the operator $L_W$ on a space of decaying functions.  We will use the fact that on a region close to the manifold the linearised $L_W$ can be approximated by $L_{U_0}$, namely the linearised operator on $M\times\R^2$ around the building block $U_0(y,t)\coloneqq {U_0}(t)$, while outside this region $L_W$ behaves like a positive operator. 
We aim to solve 
\begin{equation}\label{LW equation}
	L_W[\vphi]=-\eps^2\Delta_W\vphi+\vphi+T_W\vphi=\Lambda
\end{equation}
for a right-hand side $\Lambda=\Lambda(x)$ defined on $\R^4$. To this goal we look for a solution of the form 
\begin{equation}\label{decomposition varphi}
	\vphi(x) = \zeta_2(x)\Phi(y,t)+\Psi(x)
\end{equation}
being $\Phi$ a function defined on $M\times\R^2$ and $\Psi$ is defined on $\R^4$. We will develop an invertibility theory for $L_{U_0}$ and then apply it to find the inner function $\Phi$, but as we will see to do so we need a correction on the right-hand side. This correction is necessary for a general solvability theory and it is due to the fact that $L_{U_0}$ has a non-trivial kernel, precisely generated by $\{\sdf{V}_\alpha(t)\}_{\alpha=1,2}$. 
 We will invert $L_{U_0}$ for a family of right-hand sides satisfying an orthogonality condition with  $\sdf{V}_\alpha$, $\alpha=1,2$, for every fixed point of $M$. A way to obtain this condition for a general right-hand side $\Lambda$ is to replace it with
\begin{equation*}
	\Lambda+\zeta_2b^\alpha(y)\sdf{V}_\alpha(t)
\end{equation*}
in \eqref{L_W equation}, with $b^\alpha$ smooth functions defined on $M$. We observe that $\zeta_2b^\alpha\sdf{V}_\alpha$ is a well defined function on $\R^4$, since it is understood to be $0$ outside of $\N$. In what follows it will be useful the following definition
\begin{align*}
	\mathcal{R}_U[f,\Phi]\coloneqq -\Delta_U(f\Phi)+f\Delta_U\Phi.
\end{align*}
More precisely, if $\Phi=(\phi,\omega)$ and $U=(u,A)$, we have
\begin{equation}\label{definition of R}
	\mathcal{R}_U[f,\Phi]=\begin{pmatrix}-2\nabla^{A}\phi\cdot d f-\phi\Delta f\\
d^{*}\left(df\wedge\omega\right)+d^{*}\omega df-d\left(df\cdot\omega\right)-\omega\Delta f
\end{pmatrix}
\end{equation}
and 
\begin{equation*}
	\abs{\mathcal{R}_U[f,\Phi]}\leq C(|df|+|D^2 f|)(|\Phi|+|\nabla_U\Phi|).
\end{equation*}
Equation \eqref{LW equation} becomes 
\begin{align*}
	L_W[\vphi]&= \zeta_2L_{U_0}[\Phi]+L_W[\Psi]\\
	&\ +\zeta_2\round{L_W-L_{U_0}}[\Phi]+\eps^2\mathcal{R}_W[\zeta_2,\Phi]\\
	&=\zeta_2\round{\Lambda+b^\alpha(y)\sdf{V}_\alpha(t)}+(1-\zeta_2)\Lambda.
\end{align*}
Using \eqref{decomposition varphi}, we infer that such equation is solved if the pair $(\Phi,\Psi)$ solves the system
\begin{align}
	L_{U_0}[\Phi]+\round{L_W-L_{U_0}}[\Phi]+\zeta_1T_W\Psi=\Lambda+b^\alpha(y)\sdf{V}_\alpha(t)\quad\text{on supp}\,\zeta_2&\label{inner equation}\\
	-\eps^2\Delta_W\Psi+\Psi+(1-\zeta_1)T_W\Psi+\eps^2\mathcal{R}_W[\zeta_2,\Phi]=(1-\zeta_2)\Lambda\quad\text{on }\R^4&\label{outer equation}
\end{align}
where we used the fact that $\zeta_2\zeta_1=\zeta_1$. It is important to notice that the term $T_W$ does not contain any derivative. 
We start by solving \eqref{outer equation} with the following Lemma, whose proof is postponed to Section \ref{Proof_left}.
\begin{lemma}\label{outer invertibility}
	For every $\eps>0$ sufficiently small and any $\Gamma$ with $\|\Gamma\|_{C^{0,\gamma}_\mu(\R^4)}<\infty$ there exists $\Psi=\Psi(\Gamma)$, defining a linear map of $\Gamma$, satisfying 
	\begin{equation*}
 -\eps^2\Delta_W\Psi+\Psi+Q_\eps\Psi=\Gamma\quad\text{on }\R^4.
	\end{equation*}
where $Q_\eps=(1-\zeta_2)T_W$.
Moreover
	\begin{equation*}
		\norm{\Psi}_{C^{2,\gamma}_\mu(\R^4)}\leq C\|\Gamma\|_{C^{0,\gamma}_\mu(\R^4)}
	\end{equation*}
 for a constant $C>0$ independent of $\eps$.
\end{lemma}
Recalling that since $\|h_{1}\|_{\infty}\leq C\varepsilon$ we have, by choosing $\varepsilon$ sufficiently small, $|t|>\delta/\eps$ on $\mathrm{supp}\,\nabla\zeta_2$. Therefore
\begin{align*}
	\abs{\mathcal{R}_W[\zeta_2,\Phi]}&\leq e^{-\sigma|t|}r^{-\mu}\left(\|D\Phi\|_{C^{0,\gamma}_{\mu,\sigma}(M\times\R^2)}+\|\Phi\|_{C^{0,\gamma}_{\mu,\sigma}(M\times\R^2)}\right)\\
	&\leq Ce^{-\frac{\delta'}{\varepsilon}}r^{-\mu}\|\Phi\|_{C^{2,\gamma}_{\mu,\sigma}(M\times\R^2)}
\end{align*}
thus we find 
\begin{equation*}
	\norm{\mathcal{R}_W[\zeta_2,\Phi]}_{C^{0,\gamma}_\mu(\R^4)}\leq Ce^{-\frac{\delta'}{\varepsilon}}\|\Phi\|_{C^{2,\gamma}_{\mu,\sigma}(M\times \R^2)}.
\end{equation*}
Now, using Lemma \ref{outer invertibility}, we find a solution $\Psi=\Psi_1+\Psi_2$ of \eqref{outer equation}, where 
\begin{equation}\label{eq Delta_W}
	\begin{split}
		-\eps^2\Delta_W\Psi_1+\Psi_1+(1-\zeta_1)T_W\Psi_1&=-\eps^2\mathcal{R}_W[\zeta_2,\Phi],\\
	-\eps^2\Delta_W\Psi_2+\Psi_2+(1-\zeta_1)T_W\Psi_2&=(1-\zeta_2)\Lambda.
	\end{split}
\end{equation}
Moreover, it holds
\begin{align*}
	\left\Vert \Psi_1\right\Vert _{C^{2,\gamma}_\mu(\R^4)}&\leq C\eps^2\norm{\mathcal{R}_W[\zeta_2,\Phi]}_{C^{0,\gamma}_\mu(\R^4)}\leq C\eps^2 e^{-\frac{\delta'}{\varepsilon}}\|\Phi\|_{C^{2,\gamma}_{\mu,\sigma}(M\times \R^2)},\\
	\left\Vert \Psi_2\right\Vert _{C^{2,\gamma}_\mu(\R^4)}&\leq C\norm{(1-\zeta_2)\Lambda}_{C^{0,\gamma}_\mu(\R^4)}.
\end{align*}
This allows us to reduce the system \eqref{inner equation}-\eqref{outer equation} to a single equation depending just on $\Phi$. First, we extend the so found equation to a problem on entire $M\times\R^2$. Define 
\begin{equation}\label{definition of sdfB}
	\tilde{\sdf{B}}[\Phi]\coloneqq \zeta_4\sdf{B}[\Phi]=\zeta_4\round{L_W-L_{U_0}}[\Phi],\quad\tilde{\Lambda}=\zeta_4\Lambda
\end{equation}
and it's straightforward to check that $\tilde{\sdf{B}}$ is a linear operator of size $O(\eps)$.
With this notation, \eqref{inner equation} will be solved if we find a solution to 
\begin{equation}\label{extended inner equation}
	L_{U_0}[\Phi]+\tilde{\sdf{B}}[\Phi]+\zeta_1T_W\Psi=\tilde{\Lambda}+b^\gamma(y)\sdf{V}_\gamma(t)\quad\text{on }M\times\R^2.
\end{equation}
To solve \eqref{extended inner equation} we use the following result, which will be proved in Section \ref{Invertibility of LUbullet}.
\begin{proposition}\label{Inversion for LUbullet - theorem}
	Let $\mu\geq 0$, $\gamma\in (0,1)$ and $\sigma>0$ sufficiently small. Then, for every $\tilde\Lambda\in C^{0,\gamma}_{\mu,\sigma}(M\times\R^2)$ there exists $b\in C^{0,\gamma}_{\mu}(M)$ such that the problem 
	\begin{equation}
	L_{U_0}[\Phi]=\tilde\Lambda+b^\alpha(y)\sdf{V}_\alpha(t)\quad\text{on }M\times\R^2
\end{equation}
	admits a unique solution $\Phi=\mathcal{T}(\tilde\Lambda)$ satisfying 
	\begin{equation*}
		\norm{\Phi}_{C^{2,\gamma}_{\mu,\sigma}(M\times\R^2)}+\norm{b}_{C^{0,\gamma}_{\mu}(M)}\leq C\|\tilde\Lambda\|_{C^{0,\gamma}_{\mu,\sigma}(M\times\R^2)}
	\end{equation*}
	for some $C>0$.
\end{proposition}

Using Proposition \ref{Inversion for LUbullet - theorem}, we can rewrite \eqref{extended inner equation} as the linear problem
\begin{equation}\label{phi+gphi}
	\Phi+\mathcal{G}[\Phi]=\mathcal{H}
\end{equation}
where 
\begin{align*}
	\mathcal{G}[\Phi]&=\mathcal{T}\round{\tilde{\sdf{B}}[\Phi]+\zeta_1T_W\Psi_1[\Phi]},\\
	\mathcal{H}&=\mathcal{T}(\tilde\Lambda-\zeta_1T_W\Psi_2[\Lambda]),
\end{align*}
and the notations $\Psi_1[\Phi],\Psi_2[\Lambda]$ are chosen to highlight the dependencies from $\Phi,\Lambda$ arising from \eqref{eq Delta_W}. Remark that 
\begin{align*}
	\|\tilde{\sdf{B}}[\Phi]\|_{C^{0,\gamma}_{\mu,\sigma}(M\times\R^2)}&\leq C\eps\|\Phi\|_{C^{2,\gamma}_{\mu,\sigma}(M\times\R^2)}\\\|\zeta_1T_W\Psi_1[\Phi]\|_{C^{0,\gamma}_{\mu,\sigma}(M\times\R^2)}&\leq C \left\Vert \Psi_1[\Phi]\right\Vert _{C^{0,\gamma}_\mu(\R^4)}\leq Ce^{-\frac{\delta'}{\eps}}\|\Phi\|_{C^{2,\gamma}_{\mu,\sigma}(M\times \R^2)}\end{align*}
where we used that $\zeta_1T_W\sim e^{-|x|}\chi_{\{\zeta_1>0\}}(x)$	to control the exponential decay in the weighted norm. 
Therefore, we have 
\begin{align*}
	\norm{\mathcal{G}[\Phi]}_{C^{2,\gamma}_{\mu,\sigma}}&\leq C\round{\|\tilde{\sdf{B}}[\Phi]\|_{C^{0,\gamma}_{\mu,\sigma}}+\|\zeta_1T_W\Psi_1[\Phi]\|_{C^{0,\gamma}_{\mu,\sigma}}}\\
	&\leq C\round{\eps+e^{-\frac{\delta'}{\eps}}}\|\Phi\|_{C^{2,\gamma}_{\mu,\sigma}}
\end{align*}
and hence, choosing $\eps$ sufficiently small we find a unique solution to \eqref{phi+gphi}, from which it follows the existence of a unique solution $(\Phi,\Psi)$ to system \eqref{inner equation}-\eqref{outer equation}. Hence, $\vphi=\zeta_2\Phi+\Psi$ solves \eqref{LW equation} and  it follows directly that 
\begin{equation}
	\norm{\vphi}_{C^{2,\gamma}_{\mu}(\R^4)}\leq C\norm{\Lambda}_{C^{0,\gamma}_{\mu}(\R^4)}.
\end{equation}
which concludes the proof of Proposition \ref{invertibility L_W}.
\qed

\section{Invertibility theory on $M\times\R^2$}\label{Invertibility of LUbullet}

The aim of this Section is to prove Proposition \ref{Inversion for LUbullet - theorem}, which we now recall.

\medskip

\noindent{\bf Proposition \ref{Inversion for LUbullet - theorem}}.\ \emph{Let $\mu\geq 0$, $\gamma\in (0,1)$ and $\sigma>0$ sufficiently small. Then, for every $\tilde\Lambda\in C^{0,\gamma}_{\mu,\sigma}(M\times\R^2)$ there exists $b\in C^{0,\gamma}_{\mu}(M)$ such that the problem 
	\begin{equation}
	L_{U_0}[\Phi]=\tilde\Lambda+b^\alpha(y)\sdf{V}_\alpha(t)\quad\text{on }M\times\R^2
\end{equation}
	admits a unique solution $\Phi=\mathcal{T}(\tilde\Lambda)$ satisfying 
	\begin{equation*}
		\norm{\Phi}_{C^{2,\gamma}_{\mu,\sigma}(M\times\R^2)}+\norm{b}_{C^{0,\gamma}_{\mu}(M)}\leq C\|\tilde\Lambda\|_{C^{0,\gamma}_{\mu,\sigma}(M\times\R^2)}
	\end{equation*}
	for some $C>0$.}

\medskip

Recall that the operator $L_{U_0}$, is given by 
\begin{equation}\label{eqbul}
	L_{U_0}[\Phi]=-\Delta_{t,{U_0}}\Phi-\eps^2\Delta_{M}\Phi+\Phi+T_{{U_0}}(t)\Phi=\Psi\quad\text{on }M\times\R^2.
\end{equation}
We look for a solution to \eqref{eqbul} on a space of functions defined on $M\times\R^2$ with suitable decay. 
This cannot be done for every choice of right-hand side $\Psi$, hence instead of \eqref{eqbul} we aim to solve the projected problem
\begin{equation}\label{corrected projection - bullet}
\begin{cases}
	L_{U_0}[\Phi]=\Psi+b^\alpha(y)\sdf{V}_\alpha(t)&\text{on }M\times\R^2\\
	\int_{\R^2}\Phi(y,t)\cdot\sdf{V}_\alpha(t)dt=0&\alpha=1,2.
\end{cases}
\end{equation}
where $b = (b^1,b^2)$ is defined by
\begin{equation*}
	b^\alpha(y)\int_{\R^2}|\sdf{V}_\alpha(t)|^2dt=-\int_{\R^2}\Psi(y,t)\cdot\sdf{V}_\alpha(t)dt\quad\forall y\in M.
\end{equation*}
This variation will provide unique solvability, in the sense of Proposition \ref{Inversion for LUbullet - theorem}.
To prove the result, we will first prove the existence of an inverse of the same operator in the Euclidean space $\R^4=\R^2\times\R^2$ and use this to solve the problem locally, up to a small operator bounded by the size $\eps$ of the dilation. Subsequently, we find an actual solution by gluing and fixed point techniques.

\subsection{Solvability theory for the linearised in the flat space}\label{linflatspace}
To indicate a point in $\R^4$ we will use two coordinates $(y,t)\in\R^2\times\R^2$, so that the centre of the linearisation is ${U_0}={U_0}(t)$. Our aim is to solve 
\begin{equation}\label{lin R4}
\begin{cases}
	-\Delta_{t,{U_0}}\Phi-\eps^2\Delta_y \Phi+\Phi+T_{{U_0}}(t)\Phi=\Psi+b^\alpha(y)\sdf{V}_\alpha(t)&\text{in }\R^4,\\
	\int_{\R^2}\Phi(y,t)\cdot\sdf{V}_\alpha(t)dt=0&\alpha=1,2.
\end{cases}
\end{equation}
where 
\begin{equation}\label{balpha flat}
	b^\alpha(y)=-\frac{1}{\int_{\R^2}|\sdf{V}_\alpha(t)|^2dt}\int_{\R^2}\Psi(y,t)\cdot\sdf{V}_\alpha(t)dt,\quad \alpha=1,2.
\end{equation}
The following result holds.
\begin{proposition}\label{Solvability lin flat space}
	Let $\Psi\in L^2(\R^4)$ and let $b^\alpha$ be given by \eqref{balpha flat} for $\alpha=1,2$. Then problem \eqref{lin R4}
admits a unique solution $\Phi=\mathcal{T}(\Psi)\in H^{1}_{U_0}(\R^4)\times L^2(\R^2)$ satisfying 
\begin{equation}\label{orth Phi}
	\int_{\R^2}\Phi(y,t)\cdot\sdf{V}_\alpha(t)dt=0,\quad \alpha=1,2.
\end{equation}
Moreover, the estimate
\begin{equation}\label{global estimates}
	\norm{\Phi}_{H^{1}_{U_0}(\R^4)}\leq C\norm{\Psi}_{L^2(\R^4)}
\end{equation}
holds for some $C>0$. If besides $\Psi\in C^{0,\gamma}(\R^4)$, then
\begin{equation}\label{holder estimates}
	\norm{\Phi}_{C^{2,\gamma}(\R^4)}\leq C\norm{\Psi}_{C^{0,\gamma}(\R^4)}.
\end{equation}
\end{proposition}

Let us first prove existence. We observe that if $\Phi=(\phi,\omega)$ and $\Psi=(\psi,\eta)$ then we can split $\omega=\omega_tdt+\omega_ydy$ and $\eta=\eta_tdt+\eta_ydy$. In this way the first equation in system \eqref{lin R4} reduces to 
	\begin{equation}\label{system broken}
		\begin{cases}
-\eps^2\Delta_{y}\phi-\Delta_t^{{A_0}}\phi-\frac{1}{2}(1-3\left|{u_0}\right|^{2})\phi+2i\eps^2\nabla^{{A_0}}{u_0}\cdot\omega_{t}=\psi+b^\alpha\sdf{V}_\alpha^1\\
-\eps^2\Delta_{y}\omega_t-\Delta_t\omega_{t}+\left|{u_0}\right|^{2}\omega_{t}-2\bracket{\nabla^{{A_0}}{u_0},i\phi}=\eta_{t}+b^\alpha\sdf{V}_\alpha^2\\
-\eps^2\Delta_{y}\omega_y-\Delta_t\omega_{y}+\left|{u_0}\right|^{2}\omega_{y}=\eta_{y}
\end{cases}
	\end{equation}
where we have that the last equation is not coupled with the first two. Therefore, we need to solve two separate problems:
\begin{equation}\label{flat system}
	\begin{cases}
		-\eps^2\Delta_{y}\phi-\Delta_{t}^{{A_0}}\phi-\frac{1}{2}(1-3\left|{u_0}\right|^{2})\phi+2i\eps^2\nabla^{{A_0}}{u_0}\cdot\omega_t=\psi+b^\alpha\sdf{V}_\alpha^1\\
		-\eps^2\Delta_{y}\omega_t-\Delta_{t}\omega_t+\left|{u_0}\right|^{2}\omega-2\bracket{\nabla^{{A_0}}{u_0},i\phi}=b^\alpha\sdf{V}_\alpha^2
	\end{cases}
\end{equation}
where $\omega=\omega(x,y)dx$, and the second one is 
\begin{equation}\label{uncoupled equation}
	-\eps^2\Delta_y\omega_y-\Delta_x\omega_y+f^2\omega_y=\eta_y\quad\text{in }\R^4.
\end{equation}

\subsubsection{Solving system \eqref{flat system}}\label{solving flat system}
We apply Fourier transform to the whole system \eqref{flat system} on the $y$ variable, obtaining (after dropping the subscript $t$ from $\omega_t$ for simplicity)
\begin{equation*}
	\begin{cases}
		-\Delta_{t}^{{A_0}}\hat{\phi}+\eps^2\left|\xi\right|^{2}\hat{\phi}-\frac{1}{2}(1-3\left|{u_0}\right|^{2})\hat{\phi}+2i\eps^2\nabla^{{A_0}}{u_0}\cdot\hat{\omega}=\hat\psi+\hat b^\alpha\sdf{V}_\alpha^1\\
		-\Delta_{t}\hat{\omega}+\eps^2\left|\xi\right|^{2}\hat{\omega}+\left|{u_0}\right|^{2}\hat{\omega}-2\langle\nabla^{{A_0}}{u_0},i\hat\phi\rangle=\hat\eta+\hat b^\alpha\sdf{V}_\alpha^2,
	\end{cases}
\end{equation*}
where $\xi$ is the Fourier variable, which can be written in a compact way as 
\begin{equation}\label{compact-system}
	(\sdf{L}+\eps^2\left|\xi\right|^{2})\hat\Phi=\hat\Psi+\hat b^\alpha\sdf{V}_\alpha\quad\text{in }\mathbb{R}^{2}
\end{equation}
where $\sdf{L}$ is the gauge-orthogonal, two-dimensional linearised around ${U_0}(t)$ given by \eqref{Lstraight} and
\begin{equation}\label{balphaxi}
	\hat b^\alpha(\xi)=-\frac{1}{\int_{\R^2}|\sdf{V}_\alpha(t)|^2dt}\int_{\R^2}\hat\Psi(\xi,t)\cdot\sdf{V}_\alpha(t)dt,\quad \alpha=1,2.
\end{equation}
We now observe that \eqref{compact-system} can be solved since the bilinear form 
\begin{equation*}
	B_\xi[\Phi,\Psi]=\int_{\R^2}(\sdf{L}+\eps^2\left|\xi\right|^{2})\Phi\cdot\Psi
\end{equation*}
is coercive on $H^{1}_{U_0}(\R^2)\cap Z_{{U_0}}^\perp$, being a sum of a coercive operator (see \eqref{coercivity}) and a positive one. Precisely, choosing $\hat b^\alpha$s as in \eqref{balphaxi} ensures that the right-hand side belongs to $Z_{{U_0}}^\perp$ and in turn yields by Lax--Milgram theorem to the existence of a unique solution $\hat\Phi$ to problem \eqref{compact-system}. Besides, we have 
\begin{align*}
	\gamma\| \hat\Phi\| _{H^{1}_{U_0}}^{2}+\eps^2\left|\xi\right|^{2}\|\hat \Phi\| _{L^{2}}^{2}&\leq B_{\xi}[\hat\Phi,\hat\Phi]\\
	&=\langle \hat\Psi,\hat\Phi\rangle _{L^{2}}\\
	&\leq\| \hat\Psi\| _{L^{2}}\|\hat\Phi\| _{L^{2}}\\
	&\leq \frac{1}{2\delta^2}\| \hat\Psi\| _{L^{2}}^2+\frac{\delta^2}{2}\|\hat\Phi\| _{L^{2}}^2
\end{align*}
and hence choosing $\delta>0$ sufficiently small we get 
\begin{equation}\label{final transformed estimate}
	\int_{\mathbb{R}^{2}}\left(|\hat{\phi}|^{2}+|\nabla_{t}^{{A_0}}\hat{\phi}|^{2}+|\xi|^{2}|\hat{\phi}|^{2}+|\hat{\omega}|^{2}+|\nabla_{t}\hat{\omega}|^{2}+|\xi|^{2}|\hat{\omega}|^{2}\right)\leq C\int_{\mathbb{R}^{2}}\left(|\hat{\psi}|^{2}+|\hat{\eta}|^{2}\right)
\end{equation}
for some $C>0$. Now, applying inverse Fourier transform and Plancherel's Theorem we find a solution $\Phi\in H^{1}_{U_0}(\R^4)\times L^2(\R^2)$ to \eqref{flat system} satisfying \eqref{orth Phi} and
\begin{equation*}
	\left\Vert \Phi\right\Vert _{H^{1}_{U_0}\left(\mathbb{R}^{4}\right)}\leq C\left\Vert \Psi\right\Vert _{L^{2}\left(\mathbb{R}^{4}\right)}.
\end{equation*}
This proves the existence part for \eqref{flat system}.
\qed

\subsubsection{Solving equation \eqref{uncoupled equation}}\label{uncoupled eq} Here we mimic the procedure of \S\ref{solving flat system} and obtain a solution $\varphi$ of \eqref{uncoupled equation} satisfying 
\begin{equation*}
	\norm{\omega_y}_{H^1(\R^4)}\leq\norm{\eta_y}_{L^2(\R^4)}
\end{equation*}
finding in this way a solution to \eqref{lin R4} satisfying \eqref{global estimates}. The bilinear form related to the operator
\begin{equation*}
	B\left[\phi,\psi\right]\coloneqq\int_{\mathbb{R}^{2}}\phi\left(-\Delta\psi+f^{2}\psi\right).
\end{equation*}
is coercive in the following sense.
\begin{lemma}\label{coercivity simple}
	The operator $B$ is coercive on $H^{1}\left(\mathbb{R}^{2}\right)$, i.e. there exists a $\lambda>0$ such that 
	\begin{equation*}
	B\left[\phi,\phi\right]\geq\lambda\left\Vert \phi\right\Vert _{H^{1}}^{2}
\end{equation*}
for every $\phi\in H^1(\R^2)$.
\end{lemma}
We postpone the proof of Lemma \ref{coercivity simple} to Section \ref{Proof_left}. 
From the results of \S\ref{solving flat system} and \S\ref{uncoupled eq} we find a solution to \eqref{lin R4} satisfying estimates \eqref{global estimates}. Next, we proceed with the proof of the H\"older estimates \eqref{holder estimates}. The first step is to prove the following $L^\infty$ estimate in a cylinder. 

\subsubsection{$L^\infty$-estimates}\label{Linfty estimates subsub}
Let $\Phi=\mathcal{T}(\Psi)$ be a solution to \eqref{lin R4} with $\Psi\in L^\infty(\R^4)$. We claim that the following estimate holds
\begin{equation}\label{L-infty estimates}
	\norm{\Phi}_{H^{1}_{U_0}(B(p,1)\times\R^2)}\leq C\norm{\Psi}_{L^\infty(\R^4)}
\end{equation}
for every $p\in\R^2$. First, we prove that 
\begin{equation}\label{estimate-cilinder}
	\norm{\Phi}_{H^{1}_{U_0}(B(p,1)\times\R^2)}\leq C\norm{\Psi}_{L^2(B(p,1)\times\R^2)}\quad \forall p\in\R^2.
\end{equation}
Indeed, for such given $p$ we define the function $\varrho:\R^2\to\R$ given by 
\begin{equation*}
	\varrho(y)=\sqrt{1+\delta^2\abs{y-p}^2}
\end{equation*}
and let
\begin{equation*}
	\tilde\Phi:=\varrho^{-\nu}\Phi,\quad\tilde\Psi:=\varrho^{-\nu}\Psi,\quad\tilde b^\alpha=\varrho^{-\nu}b^\alpha.
\end{equation*}
The equation for $\tilde\Phi$ becomes
\begin{equation*}
	-\Delta_{t,{U_0}}\tilde\Phi-\eps^2\Delta_y \tilde\Phi+\tilde\Phi+T_{U_0}(x)\tilde\Phi=\tilde\Psi+\tilde b^\alpha\sdf{V}_\alpha+B_\delta[\Tilde\Phi]
\end{equation*}
with
\begin{equation*}
	\|B_\delta[\tilde\Phi]\|_{L^2(\R^4)}\leq C\delta\|\tilde\Phi\|_{H^{1}_{U_0}(\R^4)}
\end{equation*}
By estimates \eqref{global estimates} we obtain 
\begin{align*}
	\|\tilde\Phi\|_{H^{1}_{U_0}(\R^4)}&\leq C\|\tilde\Psi+
	B_\delta[\tilde\Phi]\|_{L^2(\R^4)}\\
	&\leq C\|\tilde\Psi\|_{L^2(\R^4)}+C\delta\|\tilde\Phi\|_{H^{1}_{U_0}(\R^4)},
\end{align*}
Therefore, choosing $\delta$ sufficiently small, we find 
\begin{equation}\label{estimates_tilda}
	\|\tilde\Phi\|_{H^{1}_{U_0}(\R^4)}\leq C\|\tilde\Psi\|_{L^2(\R^4)}.
\end{equation}
Next we observe that, if $\nu$ is sufficiently large, 
\begin{align*}
	\|\tilde\Psi\|_{L^2(\R^4)}&=\|\varrho^{-\nu}\Psi\|_{L^2(\R^4)}\\ &\leq C\|\Psi\|_{L^2(B(p,1)\times\R^2)}
\end{align*}
and 
\begin{align*}
	\|\Phi\|_{H^{1}_{U_0}(B(p,1)\times	\R^2)}&=\|\varrho^{\nu}\tilde\Phi\|_{H^{1}_{U_0}(B(p,1)\times	\R^2)}\\
	&\leq C \|\tilde\Phi\|_{H^{1}_{U_0}(\R^4)}
\end{align*}
From the last two estimates and \eqref{estimates_tilda}, by taking the supremum over $p\in\R^2$, \eqref{estimate-cilinder} immediately follows.
Now, to prove \eqref{L-infty estimates}, we write $\Phi=\Phi_0+\Phi_1$, where $\Phi_0$ is the unique bounded solution to 
\begin{equation*}
	-\Delta_{t,{U_0}}\Phi_0-\eps^2\Delta_y\Phi_0+\Phi_0=\Psi+b^\gamma\sdf{V}_\gamma
\end{equation*}
which satisfies $\norm{\Phi_0}_{L^\infty(\R^4)}\leq C\norm{\Psi}_{L^\infty(\R^4)}$. Then, $\Phi_1$ must satisfy
\begin{equation*}
	-\Delta_{t,{U_0}}\Phi_1-\eps^2\Delta_y \Phi_1+\Phi_1+T_{U_0}(t)\Phi_1=-T_{U_0}(t)\Phi_0.
\end{equation*}
Using \eqref{estimate-cilinder} we get 
\begin{align*}
	\norm{\Phi_1}_{H^{1}_{U_0}(B(p,1)\times\R^2)}&\leq C\norm{T_{U_0}(t)\Phi_0}_{L^2(B(p,1)\times\R^2)}\\
	&\leq\sup_{B(p,1)\times\R^2}|\Phi_0|\norm{T_{U_0}}_{L^2(\R^2)}\\
	&\leq C\norm{\Phi_0}_{L^\infty(\R^4)}\\
	&\leq C\norm{\Psi}_{L^\infty(\R^4)}
\end{align*}
and \eqref{L-infty estimates} follows.
\subsubsection{H\"older estimates}\label{holder subsub} Now we prove the H\"older estimates \eqref{holder estimates}.
Elliptic interior regularity yields that for every point $p=(y,t)\in\R^2\times\R^2$ it holds 
\begin{align*}
	\norm{\Phi}_{C^{2,\gamma}(B(p,1/2))}\leq C\round{\norm{\Phi}_{H^{1}_{U_0}(B(y,1)\times B(t,1))}+\norm{\Psi}_{C^{0,\gamma}(B(y,1)\times B(t,1))}}
\end{align*}
On the other hand, using the $L^\infty$ estimates found in \eqref{Linfty estimates subsub}
\begin{align*}
	\norm{\Phi}_{H^{1}_{U_0}(B(p,1))}&\leq\norm{\Phi}_{H^{1}_{U_0}(\R^2\times B(t,1))}\\ 
	&\leq C\norm{\Psi}_{L^\infty(\R^4)}\\
	&\leq C\norm{\Psi}_{C^{0,\gamma}(\R^4)}
\end{align*}
and since $C$ doesn't depend on the point $p$, taking the supremum we find \begin{equation*}
	\norm{\Phi}_{C^{2,\gamma}(\R^4)}\leq C\norm{\Psi}_{C^{0,\gamma}(\R^4)},
\end{equation*}
that is the sought estimates. The proof of Proposition \ref{Solvability lin flat space} is concluded. 

\subsection{Proof of Proposition \ref{Inversion for LUbullet - theorem}}
We now use the theory developed in \S\ref{linflatspace} to prove the invertibility for the linearised $L_{U_0}$ on $M\times\R^2$. For each point $p\in M$ we can find a local parametrization
\begin{equation*}
	Y_p:B(0,1)\subset\R^2\mapsto M\subset\R^4
\end{equation*}
onto a neighbourhood $\mathcal{U}_p$ of $p$ in $M$, so that writing 
\begin{equation*}
	g_{ij}(\xi):=\langle \partial_iY_p,\partial_jY_p\rangle = \delta_{ij}+\theta_p(\xi),\quad \xi \in B(0,1),
\end{equation*}
then we may assume that $\theta_p$ is smooth, $\theta_p(0)=0$ and 
\begin{equation*}
	|D^2\theta_p|\leq C\quad \text{in }B(0,1)
\end{equation*}
with $C$ independent of $p$. We represent the Laplace--Beltrami operator by 
\begin{align*}
		\Delta_{M}&=\frac{1}{\sqrt{\det g(\xi)}}\partial_{i}\round{\sqrt{\det g(\xi)}g^{ij}(\xi)\partial_j}\\
		&\eqqcolon \Delta_\xi+B_{p}
\end{align*}
where 
\begin{equation*}
	B_{p}=b^{ij}(\xi)\partial_{ij}+b^{j}(\xi)\partial_{j},\quad |\xi|<1.
\end{equation*}
Now, recall that we can parametrise each end $M_k$ as $\varphi_k(B_{R_0}^c)$, where $B_{R_0}^c=\braces{\xi \in \R^2\,:\,|\xi|>R_0}$ and 
\begin{equation*}
	\varphi_k(\xi)=\xi^je_j+F_k(\xi)e_3,
\end{equation*}
where $F_k$ is given by \eqref{expansion F_k}. Let now $r(\xi)=|\xi|$. According to \eqref{expansion F_k}, we find 
\begin{equation*}
	\partial_i\varphi_k=e_i+O(r^{-1})e_3
\end{equation*}
and hence the metric on the end $M_k$ expands as 
\begin{equation*}
	g^0_{ij}(\xi)=\partial_i\varphi_k\cdot\partial_j\varphi_k=\delta_{ij}+O(r^{-2})
\end{equation*}
In this way we can compute
\begin{equation}\label{laplacian on Mk}
	\Delta_M=a^{ij}_0\partial_{ij}+c^j_0\partial_j=\Delta_\xi+O(r^{-2})\partial_{ij}+O(r^{-3})\partial_{j}\quad\text{on }M_k. 
\end{equation}
thus, we conclude that the coefficients $b_{ij},{b_i}$ and their derivatives are uniformly bounded. 
\subsubsection{Proof of Proposition \ref{Inversion for LUbullet - theorem} - existence}
Let us now fix a small $\delta>0$. We have that 
\begin{equation*}
	|b_{ij}|\leq C\delta,\quad |Db_{ij}|+|Db_i|\leq C\delta.\quad \xi\in B(0,\delta).
\end{equation*}
Secondly, we choose a sequence of points $(p_j)_{j\in\mathbb{N}}\subset M$ such that, defining 
\begin{equation*}
	\V_k=Y_{p_k}\round{B(0,\delta/2)},\quad k=1,2,\dots
\end{equation*}
then $M$ is covered by the union of $\V_k$ and so that each $\V_j$ intersects at most a finite, uniform number of $\V_k$, with $k\ne j$. 
Consider now a smooth cut-off function $\eta$  such that $\eta(s)=1$ for $s<1$ and $\eta(s)=0$ for $s>2$. We define the following set of cut-off functions on $M$
\begin{equation*}
	\eta_{m}^k(y)=\eta\round{\frac{|\xi|}{m\delta}}\quad y=Y_{p_k}(\xi)
\end{equation*}
which are supported in 
\begin{equation*}
	\mathcal{U}_{k,m}\coloneqq Y_{p_k}\left(\braces{\xi\in\R^2\ :\ |\xi|\leq m\delta}\right)
\end{equation*}
and extended as $0$ outside of $\mathcal{U}_{k,m}$. Remark that $\eta^k_1\eta^k_2=\eta^k_1$ since 
\begin{equation*}
	\{\eta^k_2=1\}\supseteq\mathcal{U}_{k,1}.
\end{equation*}
Now, our choice of $\{\mathcal{V}_k\}$ and the fact that $\mathcal{V}_k\subset \{\eta^k_1=1\}$ guarantee that there is a constant $C>0$ such that 
\begin{equation}\label{bounds of eta1}
	1\leq \eta_1:=\sum_{k=1}^\infty\eta_{1}^k\leq C.
\end{equation}
Also, we remark the estimate
\begin{equation}\label{estimates derivatives cut off}
	|\nabla_{M}\eta_1^{k}|+|\Delta_{M}\eta_1^{k}|\leq C.
\end{equation}
At this point, we look for a solution of \eqref{corrected projection - bullet} of the form
\begin{equation*}
	\Phi=\Phi_0+\eta_1^{k}\Phi_k,\quad b^\alpha=\eta_1^kb_k^\alpha,\quad\alpha=1,2,
\end{equation*}
where we used Einstein summation convention. With this ansatz, \eqref{eqbul} can be written as 
\begin{align*}
	-\Delta_{t,{U_0}}\Phi-\eps^2\Delta_{M}\Phi+\Phi+T_{{U_0}}(t)\Phi=&-(\Delta_{t,{U_0}}+\eps^2\Delta_{M})\Phi_0+\Phi_0+T_{U_0}(x)\Phi_0\\
	&+\eta_1^{k}\squared{-(\Delta_{t,{U_0}}+\eps^2\Delta_{M})\Phi_k+\Phi_k+T_{U_0}(t)\Phi_k}\\&+\eps^2\mathcal{R}_{{U_0}}[\eta_1^k,\Phi_k]=\ \Psi+\eta_1^kb^\gamma_k\sdf{V}_\gamma.
\end{align*}
where $\mathcal{R}$ is as in \eqref{definition of R}.
The above equation is satisfied if we solve the system
\begin{align}
	-\Delta_{t,{U_0}}\Phi_k-\eps^2\Delta_{M}\Phi_k+\Phi_k+T_{U_0}(t)\Phi_k=-\eta_1^{-1}T_{U_0}(t)\Phi_0+b^\gamma_k\sdf{V}_\gamma\quad&\text{in }\mathcal{U}_{k,2}\times\R^2\label{inner_part_extension_to_MxR2}\\
	-\Delta_{t,{U_0}}\Phi_0-\eps^2\Delta_{M}\Phi_0+\Phi_0=\Psi-\eps^2\mathcal{R}_{U_0}[\eta_1^k,\Phi_k]
	\quad&\text{in }M\times\R^2.\label{outer_part_extension_to_MxR2}
\end{align}
To this end, we need the following result, which will be proved in Section \ref{Proof_left}.
\begin{lemma}\label{lemma1}
	There exists a constant $C>0$ independent on $\eps$ such that for every pair $H\in C^{0,\gamma}(M\times\R^2)$ there exists a solution $\Phi=\Phi(H)$ to the equation 
	\begin{equation*}
		-\Delta_{t,{U_0}}\Phi-\eps^2\Delta_{M}\Phi+\Phi=H\quad \text{in }M\times\R^2
	\end{equation*}
defining a linear operator in $H$, such that
\begin{equation*}
\norm{\Phi}_{C^{2,\gamma}(M\times\R^2)}\leq C\norm{H}_{C^{0,\gamma}(M\times\R^2)}.
\end{equation*}
\end{lemma}

Assume now that $\boldsymbol{\Phi}=(\Phi_k)_{k\in\mathbb{N}}\in\ell^\infty\round{C^{2,\gamma}(\R^4)}$. Using Lemma \ref{lemma1} and \eqref{estimates derivatives cut off} we can find a solution $\Phi_0=\Phi_0^1(\Psi)+\Phi_0^2(\boldsymbol\Phi)$ to \eqref{outer_part_extension_to_MxR2}, with
\begin{align*}
	-\Delta_{t,{U_0}}\Phi_0^1-\eps^2\Delta_{M}\Phi_0^1+\Phi_0^1&=\Psi\\
	-\Delta_{t,{U_0}}\Phi_0^2-\eps^2\Delta_{M}\Phi_0^2+\Phi_0^2&=-\eps^2\mathcal{R}_{U_0}[\eta_1^k,\Phi_k]
\end{align*} 
and satisfying
\begin{align*}
	\|\Phi_0\|_{C^{2,\gamma}(M\times\R^2)}&\leq \|\Psi\|_{C^{0,\gamma}(M\times\R^2)}+\eps^2\|\mathcal{R}_{U_0}[\eta_1^k,\Phi_k]\|_{C^{0,\gamma}(M\times\R^2)}\\
	&\leq\|\Psi\|_{C^{0,\gamma}(M\times\R^2)}+\eps^2\|\boldsymbol{\Phi}\|_{\ell^\infty\round{C^{2,\gamma}(\R^4)}}
\end{align*}
where we used \eqref{estimates derivatives cut off}.
We can plug such $\Phi_0$ into \eqref{inner_part_extension_to_MxR2},
writing it in coordinates
\begin{equation*}
	-\Delta_{t,{U_0}}\Phi_k-\eps^2\Delta_{\xi}\Phi_k+B_{p_k}\Phi_k+T_{U_0}(t)\Phi_k=-\eta^{-1}_1T_{U_0}(t)\Phi_0+b^\gamma_k\sdf{V}_\gamma\quad \text{in }B(0,2\delta)\times\R^2
\end{equation*}
and then extend it to the whole space using $\eta_2^k$
\begin{equation}\label{inner_in_coordinates}
	-\Delta_{t,{U_0}}\Phi_k-\eps^2\Delta_{\xi}\Phi_k-\eta_2^kB_{p_k}\Phi_k+T_{U_0}(t)\Phi_k=-\eta^{-1}_1\eta_2^kT_{U_0}(t)\Phi_0+b^\gamma_k\sdf{V}_\gamma\quad \text{in }\R^4.
\end{equation}

For every $k$, we set 
\begin{equation}
	b_k^\alpha(y)=-\frac{1}{\int_{\R^2}|\sdf{V}_\alpha(t)|^2dt}\int_{\R^2}\round{-\eta^{-1}_1\eta_2^kT_{U_0}\Phi_0+\eta_2^kB_{p_k}\Phi_k}\cdot\sdf{V}_\alpha(t)dt,\quad \alpha=1,2.
\end{equation}
and apply the inversion operator $\mathcal{T}$ of Proposition \ref{Solvability lin flat space} to get 
\begin{equation*}
	\Phi_k=\mathcal{T}(-\eta^{-1}_1\eta_2^kT_{U_0}\Phi_0+\eta_2^kB_{p_k}\Phi_k)
\end{equation*}
Then we can formulate \eqref{inner_in_coordinates} as 
\begin{equation}\label{fixed point problem}
	\boldsymbol{\Phi}+\mathcal{S}(\boldsymbol{\Phi})=\mathcal{W}
\end{equation}
where $\mathcal{S}$ is defined as 
\begin{equation*}
	\mathcal{S}(\boldsymbol{\Phi})_k
:=\mathcal{T}(\eta^{-1}_1\eta_2^kT_{U_0}\Phi_0^2(\boldsymbol\Phi)-\eta_2^kB_{p_k}\Phi_k).
\end{equation*}
and
\begin{equation*}
	\mathcal{W}_k=-\mathcal{T}(\eta^{-1}_1\eta_2^kT_{U_0}\Phi_0^1(\Psi)).
\end{equation*}
We readily check that
\begin{align*}
	\norm{\mathcal{S}(\boldsymbol\Phi)_k}_{C^{2,\gamma}(\R^4)}\leq &\ C\norm{\eta_2^kB_{p_k}(\boldsymbol\Phi)}_{C^{0,\gamma}(\R^4)}\\
	&+C\norm{\eta^{-1}_1\eta_2^kT_{U_0}\Phi_0^2(\boldsymbol{\Phi})}_{C^{0,\gamma}(\R^4)}\\
	\leq &\  C(\delta+\eps^2) \norm{\boldsymbol\Phi}_{\ell^\infty\round{C^{2,\gamma}(\R^4)}}.
\end{align*}
Taking the supremum over $k$ and $\delta,\eps$ sufficiently small we find a unique solution to \eqref{fixed point problem}, from which we then find $\Phi=\Phi_0+\eta_1^{k}\Phi_k$. Moreover, looking back at how the solution was built we infer that 
\begin{equation*}
	\norm{\Phi}_{C^{2,\gamma}(M\times\R^2)}\leq C\norm{\Psi}_{C^{0,\gamma}(M\times\R^2)}.
\end{equation*}
and the proof is concluded. 

\subsubsection{Proof of Proposition \ref{Inversion for LUbullet - theorem} - weighted estimates}\label{Proof of Proposition invlubul we}

Recall that we defined, for $\mu\geq0$ and $0<\gamma,\sigma<1$, the norm
\begin{equation*}
	\norm{\Phi}_{C^{2,\gamma}_{\mu,\sigma}(M\times\R^2)}=\|\varrho^{-1}\Phi\|_{C^{2,\gamma}(M\times\R^2)}.
\end{equation*}
where
\begin{equation*}
	\varrho(y,t)=r(y)^{-\mu} e^{-\sigma|t|}.
\end{equation*}
Let $\Phi=\varrho\tilde\Phi$. In terms of $\tilde\Phi$, equation \eqref{corrected projection - bullet} reads

\begin{equation}\label{perturbed operator holder estimates}
	-\Delta_{U_0}\tilde\Phi+\varrho^{-1}\mathcal{R}_{U_0}[\varrho,\tilde\Phi]+\tilde\Phi+T_{U_0}(t)\tilde\Phi=\tilde\Psi+\tilde b^\gamma\sdf{V}_\gamma
\end{equation}
where $\mathcal{R}_{U_0}$ is as in \eqref{definition of R}, and 
\begin{equation*}
	\tilde b^\gamma=\varrho^{-1}b^\gamma,\qquad\tilde\Psi=\varrho^{-1}\tilde\Psi.
\end{equation*}
Remark that, for some $C>0$,
\begin{equation*}
	\norm{\varrho^{-1}D^2\varrho}_{C^{0,\gamma}(M\times\R^2)}+\norm{\varrho^{-1}D\varrho}_{C^{0,\gamma}(M\times\R^2)}\leq C(\eps+\sigma)
\end{equation*}
and hence
\begin{equation*}
	\|\mathcal{R}_{U_0}[\varrho,\tilde\Phi]\|_{C^{0,\gamma}(M\times\R^2)}\leq C(\eps+\sigma)\|\tilde\Phi\|_{C^{2,\gamma}(M\times\R^2)}.
\end{equation*}
Now, observe that  up to reducing $\sigma$ and $\eps$,
we can solve \eqref{perturbed operator holder estimates} by fixed point arguments with the inversion theory provided by Proposition \ref{Inversion for LUbullet - theorem}, with the corresponding estimates for $\tilde\Phi$, namely
\begin{equation}\label{estimate1 holder}
	\|\tilde\Phi\|_{C^{2,\gamma}(M\times\R^2)}\leq C\|\tilde\Psi\|_{C^{0,\gamma}(M\times\R^2)}
\end{equation}
on the other hand, 
\begin{equation}\label{estimate2 holder}
	\norm{\varrho^{-1}D^2\Phi}_{C^{0,\gamma}(M\times\R^2)}+\norm{\varrho^{-1}D\Phi}_{C^{0,\gamma}(M\times\R^2)}+\norm{\varrho^{-1}\Phi}_{C^{0,\gamma}(M\times\R^2)}\leq C\|\tilde\Phi\|_{C^{2,\gamma}(M\times\R^2)}
\end{equation}
thus, putting together \eqref{estimate1 holder} and \eqref{estimate2 holder}, we obtain the sought estimates and the proof is complete.

\section{Invertibility theory for the Jacobi operator}\label{invertjacob}

In this section we prove Proposition \ref{invertibility Jacobi}, namely we solve the system
\begin{equation}\label{jacobi system}
		\begin{cases}
		\Delta_Mh^1+|\A_M|^2h^1=f^1-\sum_{j=0}^3c^j|\A_M|^2\hat{z}_j,\\
		\Delta_Mh^2=f^2-c^4|\A_M|^2
	\end{cases}
\end{equation}
for a right-hand side $f=(f^1,f^2)^T$ of class $C^{0,\gamma}_{4}(M)$. We recall the precise statement.

\medskip

\noindent{\bf Proposition \ref{invertibility Jacobi}}.\ \emph{Let $f\in C^{0,\gamma}_{4}(M,\R^2)$.
Then there exist constants $c^0,\dots,c^4$ such that system \eqref{jacobi system} admits a solution $h=\mathcal{H}(f)$ satisfying 
\begin{equation*}
	\|h\|_*\leq C\|f\|_{C^{0,\gamma}_{4}(M)}.
\end{equation*}}

We begin with an invertibility theory for the Laplace-Beltrami operator in the space 
\begin{equation}\label{definition of H}
	H=\braces{\psi\in H^1_{\text{loc}}(M)\ :\ \|\psi\|_H<\infty,\ \int_M|\A_M|^2\psi=0}
\end{equation}
where 
\begin{equation*}
	\|\psi\|_{H}^{2}\coloneqq \int_{M}|\nabla \psi|^{2}+\int_{M}\left|\A_{M}\right|^{2}|\psi|^{2}.
\end{equation*}
and for a right-hand side having mean zero. The following result holds.
\begin{lemma}\label{invertibility laplace}
	Let $g$ be a function defined on $M$ such that 
	\begin{equation}\label{mean zero for g}
		\int_Mg=0
	\end{equation}
	and $\||\A_M|^{-1}g\|_{L^2(M)}<+\infty$. Then there exists a solution $\psi\in H\cap L^{\infty}(M)$ to
	\begin{equation*}
		\Delta_M\psi=g\quad\text{on }M
	\end{equation*}
	such that
	\begin{equation*}
		\|\psi\|_H+\|\psi\|_\infty\leq \||\A_M|^{-1}g\|_{L^2(M)}.
	\end{equation*}
	Moreover, if $\|g\|_{C^{0,\gamma}_{4}(M)}<\infty$, then such $\psi$ satisfies also 
	\begin{equation}\label{laplastimate}
		\|\psi\|_*\coloneqq \|\psi\|_\infty+\|D_M\psi\|_{C^{0,\gamma}_2(M)}+\|D^2_M\psi\|_{C^{0,\gamma}_{4}(M)}\leq C\|g\|_{C^{0,\gamma}_{4}(M)}
	\end{equation}
	for some $C>0$
\end{lemma}
\begin{proof}
	We consider the weak formulation of the problem
	\begin{equation}\label{weak formulation}
		B[\psi,\phi]:=\int_M\nabla\phi\cdot\nabla\psi=\int_Mg\phi\eqqcolon \langle\ell_g,\phi\rangle,\quad \forall\phi\in C^{\infty}_c(M).
	\end{equation}
	We claim that $B$ is coercive on $H$, namely
	\begin{equation}\label{coercivity lb}
		B[\phi,\phi]\geq c\|\phi\|_H,\quad \forall\phi\in H
	\end{equation}
	for some $c>0$. Let us accept this for a moment and observe that 
	\begin{equation*}
		|\langle\ell_g,\phi\rangle|\leq\||\A_M|^{-1}g\|_{L^2(M)}\|\phi\|_H
	\end{equation*}
	hence we have the continuity of $\ell_g$ on $H$ and, by Riesz theorem, the existence of a solution $\psi\in H$ to 
	\begin{equation*}
		B[\psi,\tilde\phi]=\langle\ell_g,\tilde\phi\rangle, \quad \forall\tilde\phi\in H
	\end{equation*}
	satisfying 
	\begin{equation}\label{first estimates lb}
		\|\psi\|_H\leq C\||\A_M|^{-1}g\|_{L^2(M)}.
	\end{equation} 
	Now, to solve the actual weak problem \eqref{weak formulation} we observe that a $\phi\in C^\infty_c(M)$ can be written as $\phi=\tilde\phi+\alpha$ for some constant $\alpha$ and $\tilde\phi\in H$. Condition \eqref{mean zero for g} implies then 
	\begin{equation*}
		B[\psi,\phi]=B[\psi,\tilde\phi]=\langle\ell_g,\tilde\phi\rangle=\langle\ell_g,\phi\rangle,\quad \forall\phi\in C^{\infty}_c(M),
	\end{equation*}
	that is, $\psi$ is the sought solution. We now prove the coercivity of $B$. Observe that \eqref{coercivity lb} is satisfied if 
	\begin{equation*}
		\int_{M}\left|\nabla\phi\right|^{2}\geq c\int_{M}|\A_{M}|^{2}\left|\phi\right|^{2}\quad\forall\phi\in H
	\end{equation*}
	holds for some $c>0$. If not, there would be a sequence $\{\phi_n\}\subset H$ satisfying 
	\begin{gather}
		\int_M|\nabla \phi_n|^2\to0,\quad \text{as }n\to+\infty\label{gradient 2 0}\\
		\int_M|\A_M|^2|\phi_n|^2=1,\quad \forall n\geq0.\label{norm at 1}
	\end{gather}
	Notice that \eqref{gradient 2 0} and \eqref{norm at 1} imply that, passing to a subsequence, $\phi_n\to a$ in $L^2_{\text{loc}}(M)$ for some $a\in \R$. We split
	\begin{equation}\label{splitt}
		\int_M|\A_M|^2|\phi_n-a|=\int_{B_R\cap M}|\A_M|^2|\phi_n-a|+\int_{B_R^c\cap M}|\A_M|^2|\phi_n-a|
	\end{equation}
	where $R>0$ is sufficiently big. We claim that the left-hand side of \eqref{splitt} converges to 0. First we notice that 
	\begin{equation*}
		\int_{B_R\cap M}|\A_M|^2|\phi_n-a|\to 0\quad \text{as }n\to+\infty
	\end{equation*}
	because of $L^2_{\text{loc}}$-convergence and Cauchy--Schwartz inequality. Secondly, by triangle inequality,
	\begin{equation*}
		\int_{B_{R}^{c}\cap M}\left|\A_{M}\right|^{2}\left|\phi_{n}-a\right|\leq\int_{B_{R}^{c}\cap M}\left|\A_{M}\right|^{2}\left|\phi_{n}\right|+\int_{B_{R}^{c}\cap M}\left|\A_{M}\right|^{2}\left|a\right|.
	\end{equation*}
	We use that 
	\begin{equation*}
		\int_{B_{R}^{c}\cap M}\left|\A_{M}\right|^{2}\left|\phi_{n}\right|\leq\underbrace{\left(\int_{B_{R}^{c}\cap M}\left|\A_{M}\right|^{2}\left|\phi_{n}\right|^{2}\right)^{\frac{1}{2}}}_{\leq1}\underbrace{\left(\int_{B_{R}^{c}\cap M}\left|\A_{M}\right|^{2}\right)^{\frac{1}{2}}}_{O(1/R)}=O\left(\frac{1}{R}\right)
	\end{equation*}
	and also
	\begin{equation*}
		\int_{B_{R}^{c}\cap M}\left|\A_{M}\right|^{2}\left|a\right|=O\left(\frac{1}{R}\right)
	\end{equation*}
	to conclude that
	\begin{equation*}
		\int_{ M}|\A_M|^2|\phi_n-a|\to 0\quad \text{as }n\to+\infty
	\end{equation*}
	and hence 
	\begin{equation*}
		\int_{M}\left|\A_{M}\right|^{2}\phi_{n}\to\int_{M}\left|\A_{M}\right|^{2}a
	\end{equation*}
	but $\phi_n\in H$ implies $\int_{M}\left|\A_{M}\right|^{2}\phi_{n}=0$ therefore it must be $a=0$. We will show that this is incompatible with \eqref{gradient 2 0} and \eqref{norm at 1}. Recall that outside of a large cylinder the manifold $M$ decomposes into its ends $M_k$, each of which resembles a plane. We claim that 
	\begin{equation}\label{conv end 0}
		\int_{M_k}|\A_M|^2|\phi_n|^2\to 0\quad \forall k=1,\dots, m
	\end{equation}
	as $n\to\infty$. This, together with the fact that $\phi_n\to 0$ in $L^2_{\mathrm{loc}}(M)$ contradicts \eqref{norm at 1}, thus proving \eqref{coercivity lb}.
	
	\medskip
	
	Let us fix $k$ and let $\delta>0$. Let $\xi\colon B_{1/\delta}^c(0)\subset\R^2\to M_k$ be the coordinates defined by \eqref{expansion F_k}.  Since the end $M_k$ is planar, up to choosing $\delta$ sufficiently small we get 
	\begin{equation*}
		\int_{B^c_{1/\delta}}|\nabla \phi_n(\xi)|^2d\xi\to 0\quad \text{as }n\to\infty,
	\end{equation*}
	and, recalling that $|\A_M|^2$ decays as $r^{-4}$,
	\begin{equation*}
		\int_{B^c_{1/\delta}}|\xi|^{-4}| \phi_n(\xi)|^2d\xi\leq C
	\end{equation*}
	for some $C$ uniform in $n$. Consider now the Kelvin transform of $\phi_n$
	\begin{equation*}
		\widetilde\phi_n(y)\coloneqq \phi_n(y/|y|^2),\quad y\in B_\delta(0).
	\end{equation*}
	A direct calculation with the change of variable $\xi=y/|y|^2$ shows that 
	\begin{equation*}
		\int_{B_\delta}|\nabla \widetilde\phi_n(y)|^2dy=\int_{B^c_{1/\delta}}|\nabla \phi_n(\xi)|^2d\xi\to 0
	\end{equation*}
	and 
	\begin{equation*}
		\int_{B_\delta}|\widetilde\phi_n(y)|^2dy=\int_{B^c_{1/\delta}}|\xi|^{-4}| \phi_n(\xi)|^2d\xi\leq C
	\end{equation*}
	for some constant $C>0$ independent on $n$. From this argument follows that $\widetilde\phi_n\to \widetilde a$ in $L^2(B_\delta)$ for some constant $\widetilde a\in \R$. On the other hand, on every ball $B'\subset B_\delta$ not containing the origin we already know, from the fact that $\phi_n\to 0$ in $L^2_{\mathrm{loc}}$, that $\widetilde\phi_n\to 0$, thus $\widetilde a=0$. This implies that 
	\begin{equation}\label{conv phitilde}
		\|\phi_n(1+|x|^2)^{-1}\|^2_{L^2(B_{1/\delta}^c)}\leq C\|\widetilde\phi_n\|^2_{L^2(B_{\delta})}\to 0.
	\end{equation}
	The claim is proved.

\medskip

	The only thing left to do is to prove that the solution $\psi$ found satisfies 
	\begin{equation}\label{estimatepsi with AM}
		\|\psi\|_\infty\leq C\||\A_M|^{-1}g\|_{L^2(M)}.
	\end{equation}
	The claim is easily proved on compact sets by using local $L^2$-elliptic estimates and Sobolev embedding, along with \eqref{first estimates lb}, so we go on proving it on the ends of $M$. Recall that $|\A_M|=O(r^{-2})$ for $r$ large and that on each end $M_k$  the expression \eqref{laplacian on Mk} for $\Delta_M$ holds, namely
	\begin{equation*}
		\Delta_M=\Delta_\xi+b^{ij}\partial_{ij}+b^{j}\partial_{j}\quad\text{on }M_k. 
	\end{equation*}
	where $\xi$ are the Euclidean coordinates mapped on the surface via the chart $\varphi_k$ and
	\begin{equation*}
		b^{ij}=O(r^{-2}),\quad b^{j}=O(r^{-3})\quad\text{for $r=|\xi|$ large.}
	\end{equation*}
	Since the manifold $M$ resembles a plane on its ends, we use again the Kelvin transform on $\R^2$. Let 
	\begin{equation*}
		\tilde\psi(x)=\psi(x/|x|^2)
	\end{equation*}
	where $x\in B(0,\delta)$ for some $\delta$ sufficiently small. It's straightforward to verify the existence of coefficients $c^{ij}$ and $c^j$, defined in $B(0,\delta)$, such that 
	\begin{align}
		(\Delta_x+c^{ij}(x)\partial_{ij}+c^j(x)\partial_j)\tilde\psi&=\frac{1}{|x|^4}(\Delta_\xi+\tilde{b}^{ij}\partial_{ij}+\tilde b^{j}\partial_{j})\psi \round{\frac{x}{|x|^2}}\label{kelvintransformed}\\
		&=\frac{1}{|x|^4}g\round{\frac{x}{|x|^2}}\nonumber\\
		&=\frac{1}{|x|^4}\tilde g(x)\nonumber
	\end{align}
	and such that 
	\begin{equation*}
		c^{ij}=O(|x|^2),\quad c^j=O(|x|),\quad\text{for $|x|$ small}.
	\end{equation*}
	Given that $\tilde\psi$ solves the above equation and using the Sobolev embedding $W^{2,2}(B_\delta)\subset C^{0,\gamma}(B_\delta)$, by elliptic estimates,
	\begin{equation*}
		\|\tilde\psi\|_{C^{0,\gamma}(B_\delta)}\leq C\round{\||x|^{-4}\tilde g\|_{L^2(B_{2\delta})}+\|\tilde\psi\|_{L^2(B_{2\delta})}}.
	\end{equation*}
	A simple change of variable shows then that 
	\begin{gather*}
		\int_{B_{2\delta}}|x|^{-8}\abs{\tilde g(x)}^2dx=\int_{\{r>1/2\delta\}}|\xi|^4|g(\xi)|^2d\xi,\\
		\int_{B_{2\delta}}|\tilde \psi(x)|^2dx=\int_{\{r>1/2\delta\}}|\xi|^{-4}|\psi(\xi)|^2d\xi
	\end{gather*}
	and also, for $|p|>1/\delta$
	\begin{align*}
		[\psi]_{\gamma,B(p,1)}&=\sup_{\xi,\eta\in B(p,1)}\frac{|\psi(\xi)-\psi(\eta)|}{|\xi-\eta|^{\gamma}}\\
		&=\sup_{\xi,\eta\in B(p,1)}\frac{\left|\tilde{\psi}\left(\frac{\xi}{|\xi|^{2}}\right)-\tilde{\psi}\left(\frac{\eta}{|\eta|^{2}}\right)\right|}{|\xi-\eta|^{\gamma}}\\
		&\leq C(\delta)[\tilde\psi]_{\gamma,B(p^{-1},1)}
	\end{align*}
	where we used that, for $|\xi|,|\eta|>1/\delta-1$,
	\begin{equation*}
		\left|\frac{\xi}{|\xi|^{2}}-\frac{\eta}{|\eta|^{2}}\right|\leq C(\delta)|\xi-\eta|.
	\end{equation*} 
	Hence, we find
	\begin{equation*}
		\|\psi\|_{C^{0,\gamma}(r>1/\delta)}\leq C\round{\|r^2 g\|_{L^2(r>1/2\delta)}+\|r^{-2}\psi\|_{L^2(r>1/2\delta))}}.
	\end{equation*}
	Thus
	\begin{equation*}
		\|\psi\|_{H}+\|\psi\|_{\infty}\leq C\||\A_{M}|^{-1}g\|_{L^{2}(M)}
	\end{equation*}
	and the first part of the proof is concluded. Next we consider the case where $\|g\|_{C^{0,\gamma}_{4}(M)}<\infty$. We claim that
	\begin{equation*}
		\|\psi\|_*\leq C\round{\|g\|_{C^{0,\gamma}_{4}(M)}+\|\psi\|_\infty}
	\end{equation*} 
	Again, by local elliptic estimates we only need to prove the desired behaviour on the ends. Consider the Kelvin transformed formulation of the equation \eqref{kelvintransformed}. By elliptic regularity, 
	\begin{equation*}
		\|\tilde\psi\|_{C^{2,\gamma}(B_\delta)}\leq C\round{\||x|^{-4}\tilde g\|_{C^{0,\gamma}(B_{2\delta})}+\|\tilde\psi\|_{L^\infty(B_{2\delta})}}.
	\end{equation*}
	and by inversion of the Kelvin transform we find 
	\begin{equation*}
		\|\psi\|_{C^{0,\gamma}(r>1/\delta)}+\|D_M\psi\|_{C_2^{0,\gamma}(r>1/\delta)}+\|D^2_M\psi\|_{C_4^{0,\gamma}(r>1/\delta)}\leq C\|g\|_{C^{0,\gamma}_4(r>1/2\delta)}
	\end{equation*}
	where we also used \eqref{estimatepsi with AM} and the fact that 
	\begin{equation*}
		\||\A_{M}|^{-1}g\|_{L^{2}(M)}\leq C\|g\|_{C^{0,\gamma}_4(r>1/2\delta)}
	\end{equation*}
	and the proof is concluded.
\end{proof}
The invertibility theory for the Laplace--Beltrami operator just developed, together with the decay assumption $f\in C^{0,\gamma}_{4}(M)$, gives us a solution for the second component of system \eqref{jacobi system} satisfying the predicted estimate, simply by setting 
\begin{equation*}
	c^4=\frac{1}{\int_M|\A_M|^2}\int_M f^2.
\end{equation*} 
Moreover, the same theory allows us to solve the first component, in the sense of the following result.

\begin{lemma}\label{inversion jacobi}
	Given $f\in C^{0,\gamma}_{4}(M)$ there exists a unique solution $h$ to the problem
	\begin{equation}\label{jacobi with corrections}
		\Delta_Mh+|\A_M|^2h=f-c^j|\A_M|^2\hat{z}_j\quad\text{on }M,
	\end{equation}
	where $c^j$s are constants defined as 
	\begin{equation}\label{cjs}
		c^j=\frac{1}{\int_M|\A_M|^2\hat{z}^2_j}\int_M f\hat{z}_j.
	\end{equation}
	 Such solution satisfies the estimate
	 \begin{equation*}
	 	\|h\|_*\leq C\|f\|_{C^{0,\gamma}_{4}(M)}
	 \end{equation*}
	 for some $C>0$.
\end{lemma}
\begin{proof}
	The strategy of the proof relies on the following formulation of problem \eqref{jacobi with corrections} 
	\begin{equation*}
		h+\Delta_M^{-1}(|\A_M|^2h)=\Delta_M^{-1}(f-c^j|\A_M|^2\hat{z}_j),
	\end{equation*}
	obtain by applying the inverse Laplacian, and on Fredholm alternative. According to Lemma \ref{invertibility laplace}, in order to do so we need both the quantities $|\A_M|^2h$ and $f$ to have mean zero, bearing in mind that  the mean of the correction automatically vanishes
	\begin{equation*}
		\int_M|\A_M|^2\hat{z}_j=-\int_M\Delta_M\hat{z}_j=0.
	\end{equation*}
	Even if those conditions look restrictive compared to the general formulation \eqref{jacobi with corrections} we can actually recover a solution to the original problem. Indeed, if we solve the modified problem
	\begin{equation*}
		\Delta_{M}h+|A_{M}|^{2}h-|A_{M}|^{2}c_1=f-c^j|\A_M|^2\hat{z}_j-|A_{M}|^{2}c_2
	\end{equation*}
	where the constants
	\begin{align*}
		c_{1}&=\frac{1}{\int_M|\A_{M}|^{2}}\int_M|\A_{M}|^{2}h_{1},\\
		c_2&=\frac{1}{\int_M|\A_{M}|^{2}}\int_M f
	\end{align*}
	guarantee the mean zero conditions, then $\bar h=h-c_1+c_2$ solves the original problem. Hence, without loss of generality we can look for a solution of \eqref{jacobi with corrections} in the space 
	\begin{equation*}
		H^\sharp:=H\cap L^\infty(M)\cap \braces{h\in L^\infty\ :\ \int_M|\A_M|^2h\hat{z}_j=0}
	\end{equation*}
	where $H$ is as in \eqref{definition of H} and for a right-hand side $f\in C^{0,\gamma}_4(M)$ such that
	\begin{equation*}
		\int_M f=0.
	\end{equation*}
	Let us define $\|h\|_\sharp=\|h\|_H+\|h\|_\infty$. In this setting, we can invert the Laplace-Beltrami operator and obtain 
	\begin{equation}\label{fredholm formulation}
		h+T(h)=\Delta_M^{-1}(f-c^j|\A_M|^2\hat{z}_j),
	\end{equation}
	where we define $T(h)=\Delta_M^{-1}(|\A_M|^2h)$. We claim that a solution exists by Fredholm alternative, given that the right-hand side belongs to $H^\sharp$ if the coefficients $c^j$ are as in \eqref{cjs}. To apply the Theorem, we need to verify the validity of the following:
	\begin{itemize}
		\item[(1)] $T$ is a self map in $H^\sharp$, namely $T(H^\sharp)\subset H^\sharp$;
		\item[(2)] If $h\in H^\sharp$ satisfies $h+T(h)=0$, then $h=0$;
		\item[(3)] $T:H^\sharp\to H^\sharp$ is compact.
	\end{itemize}
	We begin with point $(1).$ Remark that, since $T(h)$ solves 
	\begin{equation*}
		\Delta_MT(h)=|\A_M|^2h\quad\text{on }M
	\end{equation*}
	by Lemma \ref{invertibility laplace}. 
	\begin{equation*}
		\int_M |\A_M|^2T(h)=0
	\end{equation*}
	and
	\begin{align*}
		\|T(h)\|_\sharp&=\|\Delta_M^{-1}(|\A_M|^2h)\|_\sharp\\
		&\leq C\||\A_M|h\|_{L^2(M)}\\
		&\leq C\|h\|_\sharp.
	\end{align*}
	Also, 
	\begin{equation*}
		\int_M T(h)|\A_M|^2\hat{z}_j=-\int_M T(h)\Delta_M\hat{z}_j=-\int_M|\A_M|^2h\hat{z}_j,
	\end{equation*}
	where we used the self-adjointness of $\Delta_M$. These conditions together guarantee that $T$ is a self map. To prove $(2)$, we simply apply the Laplace--Beltrami operator on the whole equation and find that 
	\begin{equation*}
		\Delta_Mh+|\A_M|^2h=0\quad\text{on }M,
	\end{equation*}
	that is, $h$ is a bounded Jacobi field. By hypothesis of non-degeneracy, it must hold $h=\sum_{k=0}^3\alpha^k\hat{z}_k$ and since $h\in H^\sharp$ we have
	\begin{equation*}
		0=\int_M |\A_M|^2h\hat{z}_j=\alpha^k\delta_{kj}=\alpha^j,\quad j=0,\dots,3,
	\end{equation*}
	hence $h=0$. Finally, to prove $(3)$ we consider $\{h_n\}$ be a sequence of functions such that 
	\begin{equation*}
		\sup_n\|h_n\|_\sharp\leq 1,
	\end{equation*}
	and we show that $\{g_n\}=\braces{T(h_n)}$ has a convergent subsequence in $H^\sharp.$ Local elliptic estimates imply uniform bounds for $g_n$ in the norm of $C^{0,\gamma}$ and in turn, by Arzel\`a--Ascoli theorem, the existence of a subsequence converging uniformly over compact subsets to a limit $g$. We claim that $g_n\to g$ also in $H^\sharp$. By completeness and pointwise convergence, it will suffice to show that for every $\epsilon>0$
	\begin{equation}\label{cauchy condition}
		\|g_n-g_m\|_\sharp<\epsilon
	\end{equation}
	up to taking $n,m$ sufficiently large. Let us consider a sufficiently large ball $B_R=M\cap \{r<R\}$. We have 
	\begin{equation}
		\|g_n-g_m\|_{H^\sharp(B_R)}\leq C(R)\|g_n-g_m\|_{L^\infty(B_R)}\to 0
	\end{equation}
	as $m,n\to\infty$ by local uniform convergence. On the other hand 
	\begin{align*}
		\|g_n-g_m\|_{H^\sharp(B_R^c)}&=\|\Delta_M^{-1}(|\A_M|^2(h_n-h_m))\|_{H^\sharp(B_R^c)}\\
		&\leq C\||\A_M|(h_n-h_m)\|_{L^2(B_R^c)}\\
		&\leq CR^{-1}\|h_n-h_m\|_\infty\\
		&\leq CR^{-1}
	\end{align*}
	which is small up to enlarging $R$. This proves $(3)$ and hence the existence of a solution to \eqref{fredholm formulation} by Fredholm alternative. The estimates follow from the formulation \eqref{fredholm formulation} and \eqref{laplastimate}.
\end{proof}

\section{Proofs of Lemmas 2---7}\label{Proof_left}

This Section contains all the proofs of Lemmas that have been postponed so far. 

\subsection{Lemma \ref{behaviour h0}} \emph{For any real numbers $\lambda_1,\dots,\lambda_m$ satisfying 
\begin{equation*}
	\lambda_1\leq \lambda_2\leq\cdots\leq\lambda_m,\qquad\sum_{j=1}^m\lambda_j=0.
\end{equation*}
there exists a smooth function $h_0^1$, defined on $M$, such that
	\begin{equation*}
		\Delta_Mh_0^1+|\A_M|^2h_0^1=0\quad\text{on }M
	\end{equation*}
	and such that on each end $M_j$
	\begin{equation*}
		h_0^1(y)=(-1)^j\lambda_j\log r+\eta\quad \text{on }M_j,
	\end{equation*}
	where $\eta$ satisfies 
	\begin{equation*}
		\|\eta\|_{L^{\infty}(M)}+\|r^2D\eta\|_{L^{\infty}(M)}<+\infty.
	\end{equation*}}

\noindent{\bf Proof.} We start by noting that, defining $\psi$ as any continuous function satisfying 
\begin{equation*}
	\psi=(-1)^j\lambda_j\log r\quad \text{on }M_j
\end{equation*}
for every $j=1,\dots,m$, then the Lemma follows by proving the existence of a solution $\eta$ to 
\begin{equation*}
	J(\eta)=-J(\psi)
\end{equation*}
where $J=\Delta_M+|\A_M|^2$ is the Jacobi operator of the immersion $M\subset\R^3$. By Lemma \ref{inversion jacobi}, we can find a unique solution with the required properties to the corrected equation
\begin{equation*}
	J(\eta)=-J(\psi)-c^j|\A_M|^2\hat{z}_j
\end{equation*}
where
\begin{equation}\label{orthogonality J(psi)}
	c^j=-\frac{1}{\int_M|\A_M|^2\hat{z}^2_j}\int_M J(\psi)\hat{z}_j.
\end{equation}
We claim that
\begin{equation}\label{orthJ(psi)}
	\int_M J(\psi)z_i=0\quad i=0,\dots,3
\end{equation}
and hence that all the constants in \eqref{orthogonality J(psi)} vanish, concluding the proof of the result. Consider a cylinder $C_R=\{x\in\R^3\ :\ r(x)<R\}$. Then, using coordinates \eqref{expansion F_k} we find 
\begin{align}
	\int_{M\cap C_R}J(\psi)z_idV&=\sum_{j=1}^m\int_{M_j\cap \partial C_R}(z_i\partial_r\psi-\psi\partial_r z_i)d\sigma\label{Int on cyl}\\
	&=\sum_{j=1}^m(-1)^j\squared{\frac{\lambda_j}{R}\int_{|\xi|=R}z_id\sigma(\xi)-\lambda_j\log R\int_{|\xi|=R}\partial_rz_id\sigma(\xi)}.\nonumber
\end{align}
for $i=0,\dots,3$. Now, on each end $M_j$ we have
\begin{align*}
	z_3&=\nu_1\cdot e_3=\frac{(-1)^j}{\sqrt{1+|\nabla F_k|^2}}=(-1)^j+O(r^{-2}),\quad\partial_r z_3=O(r^{-3})
\end{align*}
and hence we get 
\begin{equation*}
	\int_{M\cap C_R}J(\psi)z_3dV=2\pi\sum_{j=1}^m\lambda_j+O(R^{-1}).
\end{equation*}
Passing to the limit as $R\to\infty$ and using the balancing condition on the $\lambda_j$'s we find \eqref{orthJ(psi)} for $i=3$. Similarly,
 \begin{align*}
	z_i&=(-1)^ka_k\frac{x_i}{r^2}+O(r^{-3}),\quad\partial_r z_i=(-1)^{k+1}a_k\frac{x_i}{r^2}+O(r^{-3}),\quad\text{on }M_j
\end{align*}
for $i=1,2$. This implies again, using \eqref{a_ks}, \eqref{Int on cyl} and taking the limit as $R\to+\infty$,
\begin{equation*}
	\int_MJ(\psi)z_idV=0,\quad i=1,2.
\end{equation*}
Finally, if $i=0$ on $M_j$
\begin{equation*}
	(-1)^jz_0=-\xi_2\partial_2F_j+\xi_1\partial_1F_j=\frac{1}{r^2}(b_{j1}\xi_2-b_{j2}\xi_1)+O(r^{-2})
\end{equation*}
and $\partial_r z_0=O(r^{-2})$. Thus, \eqref{orthJ(psi)} holds also for $i=0$ and the claim is proved. This concludes the proof.
\qed

\subsection{Lemma \ref{Lipsh}} \emph{
The map $G$ defined in \eqref{G} satisfies 
	\begin{equation*}
		\|G(0)\|_{C^{0,\gamma}_{4}(M)}\leq C\eps
	\end{equation*} 
	and the Lipschitz condition
	\begin{equation*}
		\|G(h_1)-G(h_2)\|_{C^{0,\gamma}_{4}(M)}\leq C\eps\|h_1-h_2\|_*.
	\end{equation*}
}

\noindent{\bf Proof.}
Consider at first $G^1$ - we claim that 
\begin{gather}
	\|G^1(0)\|_{C^{0,\gamma}_{4}(M)}\leq C\eps\label{G1 point}\\
	\|G^1(h_1)-G^1(h_2)\|_{C^{0,\gamma}_{4}(M)}\leq C\eps\|h_1-h_2\|_*\label{G_1 lip}.
\end{gather}
To see this, let us take for instance $G_1^1$. It holds that
\begin{align*}
	G_1^1(0)&=-\eps\partial_ih_0^\beta\partial_jh_0^\delta\int_{\R^2}\zeta_4a^{ij}_{1,\delta}(t^\delta+h_0^\delta)\nabla_{\beta\gamma,{U_0}}{U_0}\cdot\sdf{V}_1\\
	&+\eps^{-2}\int_{\R^2}\zeta_4R_1(y,x,0,0)\cdot\sdf{V}_1
\end{align*}
from which follows \eqref{G1 point}. Lipschitz dependence \eqref{G_1 lip} can be checked directly from the expression 
\begin{align*}
	G_1^1(h_1)	&=-\varepsilon^{2}\Delta_{M}h^{\delta}_{1}|\A_M|^2\int_{\R^2}\zeta_4\nabla_{\delta,{U_0}}\tilde{U}_{1}\cdot\sdf{V}_1\\
	&+\eps^2\Delta_{M}h^{\delta}_{1}a_{0}^{ij}\partial_{i}h_{0}^{\beta}\partial_{j}h_{0}^{\gamma}\int_{\R^2}\zeta_4\nabla_{\delta,{U_0}}\tilde{U}_{\beta\gamma}\cdot\sdf{V}_1\\
	&-\eps\partial_{i}h^{\beta}\partial_{j}h^{\gamma}\int_{\R^2}\zeta_4(t^\delta+h^\delta)a_{1,\delta}^{ij}\nabla_{\beta\gamma,{U_0}}{U_0}\cdot\sdf{V}_1\\
	&+\eps^{-2}\int_{\R^2}\zeta_4R_1(y,t,h,\nabla_{M}h)\cdot\sdf{V}_1
\end{align*}
for instance, 
\begin{align*}
	\norm{\varepsilon^{2}\Delta_{M}(h^{\delta}_{1}-h_2^\delta)|\A_M|^2\int_{\R^2}\zeta_4\nabla_{\delta,{U_0}}\tilde{U}_{1}\cdot\sdf{V}_1}_{C^{0,\gamma}_{4}(M)}&\leq C\eps^2\|\A_M\|_\infty^2\|D^2_Mh_1-D^2_Mh_2\|_{C^{0,\gamma}_{4}}\\
	&\leq C\eps^2\|h_1-h_2\|_*
\end{align*}
and the other terms are similarly checked. Now let's consider $G^2$---for ease of notation, define 
\begin{equation*}
	\sdf{N}(h_1)=\eps^{-2}\zeta_4\pi_{Z_{W,g}^\perp}N(\vphi(h_1))+\tilde{\sdf{B}}(h_1,\vphi(h_1))
\end{equation*}
so that
\begin{equation*}
	G^2_\alpha(h_1)=\int_{\R^2}\sdf{N}(h_1)\cdot\sdf{V}_\alpha.
\end{equation*}
and where we made explicit the two dependencies on $h_1$ of this term, one direct and one through $\vphi$. We claim that both these dependencies are Lipschitz with small constant. In particular, we claim that
\begin{gather}
	\|N(\vphi_{1})-N(\vphi_{2})\|_{C_{4}^{0,\gamma}(\mathbb{R}^{4})}\leq C\eps^3\|\vphi_{1}-\vphi_{2}\|_{C_{4}^{2,\gamma}(\mathbb{R}^{4})},\label{N lip claim 1}\\
	\|\tilde{\sdf{B}}(h,\vphi_{1})-\tilde{\sdf{B}}(h,\vphi_{2})\|_{C_{4}^{0,\gamma}(\mathbb{R}^{4})}\leq C\eps^3\|\vphi_{1}-\vphi_{2}\|_{C_{4}^{2,\gamma}(\mathbb{R}^{4})},\label{N lip claim 1.5}\\
	\|\tilde{\sdf{B}}(h_{1},\vphi)-\tilde{\sdf{B}}(h_{2},\vphi)\|_{C_{4}^{0,\gamma}(\mathbb{R}^{4})}\leq C\eps^4\|h_1-h_2\|_*\|\vphi\|_{C^{2,\gamma}_{4}(\R^4)}\label{N lip claim 2}
\end{gather}
also, we claim that 
\begin{equation}\label{N lip claim 3}
	\|\vphi_{1}-\vphi_{2}\|_{C_{4}^{2,\gamma}(\mathbb{R}^{4})}\leq C\eps^2\|h_1-h_2\|_*.
\end{equation}
Let us accept estimates \eqref{N lip claim 1}-\eqref{N lip claim 3} for a moment. We readily check that, denoting $\vphi_l=\vphi(h_l)$
\begin{align*}
	\norm{G_\alpha^2(h_1)-G_\alpha^2(h_2)}_{C^{0,\gamma}_{4}(M)}&= \norm{\int_{\R^2}[\sdf{N}(h_1)-\sdf{N}(h_2)]\cdot\sdf{V}_\alpha}_{C^{0,\gamma}_{4}(M)}\\
	&\leq C\varepsilon^{-2}\|N(\vphi_{1})-N(\vphi_{2})\|_{C_{4}^{0,\gamma}(\mathbb{R}^{4})}+\\
 &+C\varepsilon^{-2}\|B(h_1,\vphi_{1})-B(h_2\vphi_{2})\|_{C_{4}^{0,\gamma}(\mathbb{R}^{4})}\\
	&\leq C\eps\|\vphi_{1}-\vphi_{2}\|_{C_{4}^{2,\gamma}(\mathbb{R}^{4})}+ C\varepsilon^{-2}\|B(h_{1},\vphi_{1})-B(h_{1},\vphi_{2})\|_{C_{4}^{0,\gamma}(\mathbb{R}^{4})}\\
	&+ C\varepsilon^{-2}\|B(h_{1},\vphi_{2})-B(h_{2},\vphi_{2})\|_{C_{4}^{0,\gamma}(\mathbb{R}^{4})}\\
	&\leq C\eps\|\vphi_{1}-\vphi_{2}\|_{C_{4}^{2,\gamma}(\mathbb{R}^{4})}+C\eps^2\|h_1-h_2\|_*\|\vphi\|_{C^{2,\gamma}_{4}(\R^4)}\\
	&\leq C(\eps^3+\eps^5)\|h_1-h_2\|_*.
\end{align*}
We now prove all the Lipschitz dependencies, starting with \eqref{N lip claim 1}. This follows automatically from the definition of $N$ as a second order expansion of $S$ in the direction of $\vphi$, from where we find
\begin{align*}
	\|N(\vphi_1)-N(\vphi_2)\|_{C^{0,\gamma}_{4}(\R^4)}&\leq C \round{\|\vphi_1\|_{\infty}+\|\vphi_2\|_{\infty}}\|\vphi_1-\vphi_2\|_{C^{2,\gamma}_{4}(\R^4)}\\
	&\leq C\eps^3\|\vphi_1-\vphi_2\|_{C^{2,\gamma}_{4}(\R^4)}
\end{align*} 
where we used that $\|\vphi_\ell\|_{C^{2,\gamma}_4(\R^4)}=O(\eps^3)$, $\ell=1,2$. Next we prove \eqref{N lip claim 3} and to do so we recall that $\vphi$ solves \eqref{Perturb eq not inverted - corrected}. The term $\mathcal{R}(h)=-S(W)$ has a Lipschitz dependence 
\begin{equation*}
	\|\mathcal{R}(h_1)-\mathcal{R}(h_2)\|_{C^{0,\gamma}_{4}(\R^4)}\leq C\eps^2\|h_1-h_2\|_*
\end{equation*}
as shown by calculations analogous to those for the Lipschitz behaviour of $G^1$. The estimates from Proposition \ref{invertibility L_W} yield to 
\begin{align*}
	\|\vphi_1-\vphi_2\|_{C^{2,\gamma}_{4}(\R^4)}&\leq C\|\mathcal{R}(h_1)-\mathcal{R}(h_2)\|_{C^{0,\gamma}_{4}(\R^4)}+\|N(\vphi_1)-N(\vphi_2)\|_{C^{0,\gamma}_{4}(\R^4)}\\
	&\leq C\eps^2\|h_1-h_2\|_*+C\eps^3\|\vphi_1-\vphi_2\|_{C^{2,\gamma}_{4}(\R^4)}
\end{align*}
and hence, up to choosing $\eps$ sufficiently small, we have \eqref{N lip claim 3}. Lastly, we observe that \eqref{N lip claim 1.5} Is a direct consequence of the smallness in $\eps$ of all the coefficients of the second order operator $B$, while \eqref{N lip claim 2} is a consequence of the mild Lipschitz dependence in $h_1$ of such coefficients, for instance 
\begin{equation*}
	a^{ij}_{1,\delta}(y, \eps(t^\delta + h_0^\delta + h_1^\delta)).
\end{equation*}
Finally, we prove that the term
\begin{equation*}
	G^3_\alpha(h_1)=-{\eps^{-2}}\int_{\R^2}\zeta_4L_{U_0}[\vphi]\cdot\sdf{V}_\alpha
\end{equation*}
is Lipschitz with a constant that is exponentially small in $\eps$. This is straightforward using the self-adjointness of $L_{U_0}$ and the fact that $\sdf{V}_\alpha\in\ker L_{U_0}$. Indeed, integrating by parts yields to 
\begin{align*}
	\int_{\mathbb{R}^{2}}\zeta_{4}L_{U_0}\left[\vphi\right]\cdot\mathsf{V}_{\alpha}dt&=\int_{\mathbb{R}^{2}}\zeta_{4}\left(-\Delta_{t,{U_0}}\vphi-\Delta_{M_{\varepsilon}}\vphi+\vphi+T_{{U_0}}\vphi\right)\cdot\mathsf{V}_{\alpha}dt\\&=-\int_{\mathbb{R}^{2}}\Delta_{t}\zeta_{4}\mathsf{V}_{\alpha}\cdot\vphi dt-\int_{\mathbb{R}^{2}}\left(\nabla_{t}\zeta_{4}\cdot\nabla_{t,{U_0}}\mathsf{V}_{\alpha}\right)\vphi dt
\end{align*}
where we also used the orthogonality between $\vphi$ and $\sdf{V}_\alpha$. Using the exponential decay of $\sdf{V}_\alpha$ and the usual argument with the support of $\zeta_4$ we obtain 
\begin{align*}
	\|G_\alpha^3(h_1)-G_\alpha^3(h_2)\|_{C^{0,\gamma}_{4}(\R^4)}&\leq C\eps^{-2}e^{-\frac\delta\eps}\|\vphi_1-\vphi_2\|_\infty\\
	&\leq Ce^{-\frac\delta\eps}\|h_1-h_2\|_*.
\end{align*}
Finally, we can put all the estimates just found together to obtain that $G$ has an $O(\eps)$-Lipschitz constant. 
\qed

\subsection{Lemma \ref{orthogonality Z_i}} \emph{
It holds
	\begin{equation*}
		\int_{\R^4}S(U)\cdot\Theta_U[\gamma]=0\quad\text{and}\quad\int_{\R^4}S(U)\cdot Z_i=0
	\end{equation*}
	for $i=0,\dots,4$.
}

\medskip

\noindent{\bf Proof.} We start by proving that 
\begin{equation*}
	\int_{\R^4}S(U)\cdot\Theta_U[\gamma]=0,\quad \gamma=-\Theta_W^*[\Phi] .
\end{equation*}
First, remark that by \eqref{estimate vphi}
\begin{equation*}
	\|\gamma\|_{C^{0,\alpha}_4(\R^4)}<\infty
\end{equation*}
and thus $\gamma$ decays at infinity. 

\medskip

Let us consider a cylinder $C_R=\{x\in \R^4\ :\ x_1^2+x_2^2<R^2\}$. It holds
\begin{align*}
	\int_{C_R}S(U)\cdot\Theta_U[\gamma]&=\int_{C_R}\bracket{-\eps^2\Delta^{A}u-\tfrac{1}{2}\left(1-|u|^{2}\right)u,iu\gamma}+\eps^2\int_{C_{R}}\left(\eps^2d^{*}dA-\bracket{\nabla^Au,iu}\right)d\gamma\\
	&=\int_{C_R}\bracket{-\eps^2\Delta^Au,iu}\gamma+\eps^2\int_{C_R}\eps^2d^*dA\cdot d\gamma-\eps^2\int_{C_R}\bracket{\nabla^Au,iu}d\gamma\\
	&=\eps^2\int_{\partial C_R}d\gamma(\eps^2d^*dA-\bracket{\nabla^Au,iu}).
\end{align*}
Using the decay of $\gamma$, we see that the last quantity vanishes as $R\to+\infty$, proving the claim.
Now, we prove the second part of the lemma, namely that 
\begin{equation*}
	\int_{\R^4}S(U)\cdot Z_i=0\quad i=0,\dots,4.
\end{equation*}
Recall that, if $U=(u, A)^T$,
\begin{equation*}
	Z_i=\nabla_{x_i,U}U=\nabla_{x_i}{U}-\Theta_{U}[A_i]
\end{equation*}
for $i=1,\dots,4$. We're going to prove that the integration on $C_R$ of $S(U)$ against both of these quantities produces boundary terms that, together, vanish as $R\to+\infty$. For instance consider $i=3$. It holds
\begin{align*}
	\int_{C_{R}}S(U)\cdot\nabla_{x_{3}}U&=\int_{C_{R}}\bracket{-\eps^2\Delta^{A}u-\tfrac{1}{2}\left(1-|u|^{2}\right)u,\partial_{x_{3}}u}+\eps^2\int_{C_{R}}\left(\eps^2d^{*}dA-\text{Im}\left(\bar{u}\nabla^{A}u\right)\right)\cdot\partial_{x_{3}}A\\
	&=\int_{\partial C_{R}}\eps^2\bracket{\nabla u\cdot\hat{r},\partial_{3}u}-\int_{\partial C_{R}}\eps^2\bracket{iAu\cdot\hat{r},\partial_{3}u}+\eps^2\int_{\partial C_{R}}\eps^2\partial_{3}A_{j}\left[dA(\cdot,e_{j})\cdot\hat{r}\right]\\
	&\ -\frac{\eps^2}{2}\int_{C_{R}}\partial_{x_{3}}\left(|\nabla^{A}u|^{2}+\eps^2|dA|^{2}+\frac{1}{4\eps^2}\left(1-|u|^{2}\right)^{2}\right)\\
	&=\int_{\partial C_{R}}\eps^2\bracket{\nabla^Au\cdot\hat{r},\partial_{3}u}+\eps^2\int_{\partial C_{R}}\eps^2\partial_{3}A_{j}\left[dA(\cdot,e_{j})\cdot\hat{r}\right].
\end{align*}
Where we used the invariance of the cylinder $C_R$ in the $z_3$-direction. Next, we compute
\begin{equation*}
\int_{C_{R}}S(U)\cdot\Theta_{U}\left[A_{3}\right]=\int_{C_{R}}\begin{pmatrix}-\Delta^{A}u-\frac{1}{2}\left(1-|u|^{2}\right)u\\
d^{*}dA-\text{Im}\left(\bar{u}\nabla^{A}u\right)
\end{pmatrix}\cdot\begin{pmatrix}iuA_{3}\\
dA_{3}
\end{pmatrix},
\end{equation*}
and as above an integration by parts yields to 
\begin{align*}
	\int_{C_{R}}S(U)\cdot\Theta_{U}\left[A_{3}\right]&=-\int_{\partial C_R}\eps^2\bracket{\nabla^Au\cdot\hat{r},iuA_3}-\eps^2\int_{\partial C_{R}}\eps^2\partial_{j}A_{3}\left[dA(\cdot,e_{j})\cdot\hat{r}\right]\\
	& =-\int_{\partial C_R}\eps^2\bracket{\nabla^Au\cdot\hat{r},iuA_3}-\eps^2\int_{\partial C_{R}}\eps^2\partial_{j}A_{3}\left[dA(\cdot,e_{j})\cdot\hat{r}\right]
\end{align*}
where we used the gauge invariance of the energy $E$. In all we get 
\begin{align*}
	\int_{C_{R}}S(U)\cdot \nabla_{x_{3},U}U&=\int_{\partial C_{R}}\eps^2\bracket{\nabla^Au\cdot\hat{r},\partial_{3}^Au}+\eps^2\int_{\partial C_{R}}\eps^2\left(\partial_{3}A_{j}-\partial_{j}A_{3}\right)\left[dA(\cdot,e_{j})\cdot\hat{r}\right]\\
 &=\eps^2\int_{\partial C_{R}}\left(\nabla_{U}U\cdot\hat{r}\right)\nabla_{x_{3},U}U.
\end{align*}
We claim that 
\begin{equation}\label{limit boundary terms to 0}
	\lim_{R\to+\infty}\int_{\partial C_{R}}\left(\nabla_{U}U\cdot\hat{r}\right)\cdot \nabla_{x_{3},U}U=0.
\end{equation}
We begin by expanding
\begin{equation*}
	\left(\nabla_{U}U\cdot\hat{r}\right)\nabla_{x_{3},U}U=|\mathsf{V}_{1}|^{2}\left(\nu_{1}\cdot e_{3}\right)\left(\nu_{1}-\left(-1\right)^{k}\frac{\eps\lambda_{k}}{r}\hat{r}\right)\cdot\hat{r}+O\left(\varepsilon^2 r^{-2}\right)
\end{equation*}
and by observing that, since the $k$-th end of the manifold reads
\begin{equation*}
	x_{3}=F_{k}\left(x_{1},x_{2}\right)=\left(a_{k}\log r+b_{k}+O(r^{-1})\right)
\end{equation*}
the normal vector satisfies
\begin{equation*}
	(-1)^{k}\nu_{1}=\frac{1}{\sqrt{1+|\nabla F_{k}|^{2}}}\left(\nabla F_{k},-1\right)=a_{k}r^{-1}\hat{r}-e_{3}+O(r^{-2}).
\end{equation*}
On the portion of $C_R$ near this end we have 
\begin{equation*}
	\left(\nu_{1}\cdot e_{3}\right)\left(\nu_{1}-\left(-1\right)^{k}\frac{\eps\lambda_{k}}{r}\hat{r}\right)\cdot\hat{r}=-\frac{a_{k}+\eps\lambda_{k}}{r}+O(r^{-2}).
\end{equation*}
Now, remark that
\begin{equation*}
	t_{1}=\left(x_{3}-F_{k}\left(x_{1},x_{2}\right)-\eps\lambda_{k}\log r+O\left(1\right)\right)\left(1+O\left(R^{-2}\right)\right)
\end{equation*}
and hence we get
\begin{equation*}
	\int_{-\infty}^{+\infty}dx_{4}\int_{F_{k}(x')+\eps\lambda_{k}\log r-\rho}^{F_{k}(x')+\eps\lambda_{k}\log r+\rho}|\mathsf{V}_{1}|^{2}dx_{3}=\int_{\mathbb{R}^{2}}|\mathsf{V}_{1}|^{2}dt+O\left(R^{-2}\right).
\end{equation*}
This follows from the fact that on $\partial C_R$ the distance between ends is greater than $\rho\coloneqq2\gamma\log R$. 
Now, we use the fact that for large $r$ the covariant gradient of $U$ is exponentially small with the distance from each end of the manifold. Precisely, it holds
\begin{equation}\label{estimates gradient ends manifold}
	\abs{\nabla_UU}\leq C\sum_{j=1}^me^{-\sigma|x_3-(F_j(x')+\eps\lambda_j\log r)|}.
\end{equation}
This fact can be proved precisely as in Lemma 9.4 in \cite{del2013entire}, namely through characterization \eqref{expression S(U) with corrections} of $U$ and barriers, using the fact that the main order of the linearised behaves roughly as $-\Delta+1$ on each component. 
Using \eqref{estimates gradient ends manifold}, we get that far from the ends of the manifold
\begin{equation*}
	\int_{\cap_{k}\{|x_{3}-F_{k}|>\rho\}}\left(\nabla_{U}U\cdot\hat{r}\right)\cdot\nabla_{x_{3},U}U=O\left(R^{-2}\right).
\end{equation*}
Finally, we infer 
\begin{equation*}
	\int_{\mathbb{R}^{2}}\left(\nabla_{U}U\cdot\hat{r}\right)\cdot\nabla_{x_{3},U}U=-\frac{1}{R}\sum_{k=1}^{m}\left(a_{k}+\eps\lambda_{k}\right)\int_{\mathbb{R}^{2}}|\mathsf{V}_{1}|^{2}dt+O\left(R^{-2}\right)
\end{equation*}
and thus
\begin{equation*}
	\int_{\partial C_{R}}\left(\nabla_{U}U\cdot\hat{r}\right)\cdot\nabla_{x_{3},U}U=-2\pi\sum_{k=1}^{m}\left(a_{k}+\eps\lambda_{k}\right)\int_{\mathbb{R}^{2}}|\mathsf{V}_{1}|^{2}dt+O\left(R^{-1}\right)
\end{equation*}
but since $\sum_{k=1}^{m}a_{k}=\sum_{k=1}^{m}\lambda_{k}=0$ we get \eqref{limit boundary terms to 0} and the proof is complete for $i=3$. For $i=4$ the proof is almost identical, with the only difference being that the terms $\lambda_k$ don't appear. For $i=2$ is similar, but the integration of $S(U)\cdot\nabla_{x_2}U$ on $C_R$ produces an extra boundary term 
\begin{align*}
	\int_{C_{R}}S(U)\cdot\nabla_{x_{2}}U
	&=\int_{\partial C_{R}}\eps^2\left\langle\nabla^{A}u\cdot\hat{r},\partial_{2}u\right\rangle+\eps^2\int_{\partial C_{R}}\eps^2\partial_{2}A_{j}\left[dA(\cdot,e_{j})\cdot\hat{r}\right]\\
	&\ -\frac{\eps^2}{2}\int_{\partial C_{R}}\left(|\nabla^{A}u|^{2}+\eps^2|dA|^{2}+\frac{1}{4\eps^2}\left(1-|u|^{2}\right)^{2}\right)n_2
\end{align*}
due to the fact that the domain is bounded in the $x_2$ direction. We denoted $n_2=x_2/r$. By the expansion of the gradient $\nabla_UU$, we get
\begin{equation*}
	|\nabla^{A}u|^{2}+\eps^2|dA|^{2}=|\sdf{V}_1|^2+|\sdf{V}_2|^2+O(\eps r^{-2}e^{-\sigma|t|})
\end{equation*}
and hence, since $\int_{r=R}n_2=0$, 
\begin{equation*}
	\lim_{R\to+\infty}\int_{\partial C_{R}}\left(|\nabla^{A}u|^{2}+\eps^2|dA|^{2}\right)n_{2}=\lim_{R\to+\infty}2\int_{\mathbb{R}^{2}}|\mathsf{V}_{1}|^{2}\int_{\{r=R\}}n_{2}+O(R^{-1})=0
\end{equation*}
and a similar argument yields to 
\begin{equation*}
	\lim_{R\to+\infty}\int_{\partial C_{R}}\left(1-|u|^2\right)^2n_{2}=0.
\end{equation*}
The other boundary terms are treated exactly as before and of course the same proof holds also for $i=1$. We are only left with the case $i=0$, for which is convenient to switch to cylindrical coordinates. We have 
\begin{equation*}
	\int_{C_{R}}S(U)\cdot\left(x_{2}\nabla_{x_{1},U}U-x_{1}\nabla_{x_{2},U}U\right)=\int_{C_{R}}S(U)\nabla_{\theta,U}U
\end{equation*}
and just as in the previous cases an integration by part produces a vanishing term (due to the rotational symmetry of the energy) plus a boundary term, namely
\begin{align*}
	\int_{C_{R}}S(U)\cdot\nabla_{\theta,U}U&=R\int_{\mathbb{R}^{2}}\int_{0}^{2\pi}\nabla_{r,U}U\cdot\nabla_{\theta,U}U\vert_{r=R}\\
	&=\int_{C_R}\nabla_{r,U}U\cdot\nabla_{\theta,U}U
\end{align*}
but from the expansion of the gradient we obtain
\begin{align*}
	\nabla_{r,U}U\cdot\nabla_{\theta,U}U=|\sdf{V}_1|^2O(R^{-2})+O(R^{-2}e^{-\sigma|t|})
\end{align*}
thus
\begin{equation*}
	\int_{C_R}\nabla_{r,U}U\cdot\nabla_{\theta,U}U=O(R^{-1})
\end{equation*}
and letting $R\to+\infty$ we obtain the claim also for $i=0$. The proof is concluded.
\qed\subsection{Lemma \ref{coercivity simple}} \emph{
	There exists a $\lambda>0$ such that 
	\begin{equation*}
	B\left[\phi,\phi\right]\coloneqq \int_{\mathbb{R}^{2}}\nabla\psi\cdot\nabla\phi+f^{2}\psi\phi\geq\lambda\left\Vert \phi\right\Vert _{H^{1}}^{2}
\end{equation*}
for every $\phi\in H^1(\R^2)$.
}

\medskip

\noindent{\bf Proof.}	By contradiction suppose that we can find a collection of functions $\left\{ \phi_{j}\right\}$  such that $\left\Vert \phi_{j}\right\Vert _{H^{1}}=1$ for every $j$ and
	\begin{equation}\label{contradiction}
		B\left[\phi_{j},\phi_{j}\right]\to0
	\end{equation} 
	as $j\to\infty$. Extracting a subsequence, we have that 
	\begin{equation*}
		\phi_{j}\rightharpoonup\phi^{*}\quad\text{in }H^{1}
	\end{equation*}
	and by weak lower semicontinuity of $B$, using \eqref{contradiction}, we infer 
	\begin{equation*}
		0\leq B\left[\phi^{*},\phi^{*}\right]\leq\liminf B\left[\phi_{j},\phi_{j}\right]=0
	\end{equation*}
	that is 
	\begin{equation*}
		B\left[\phi^{*},\phi^{*}\right]=\int_{\mathbb{R}^{2}}\left(\left|\nabla\phi^{*}\right|^{2}+f^{2}\left|\phi^{*}\right|^{2}\right)=0
	\end{equation*}
	which in turn implies $\phi^{*}=0$. Now, observe that 
	\begin{equation}\label{key}
		\left\Vert \phi_{j}\right\Vert _{H^{1}}^{2}=B\left[\phi_{j},\phi_{j}\right]+\int_{\mathbb{R}^{2}}\left(1-f^{2}\right)\left|\phi_{j}\right|^{2}.
	\end{equation}
	Now, given any $\epsilon>0$ we can find $r(\epsilon)>0$ such that 
	\begin{equation*}
		1-f^2<\epsilon\quad\text{in }B_{r(\epsilon)}^c,
	\end{equation*}
	therefore, since $\left\Vert \phi_{j}\right\Vert _{H^{1}}^{2}=1$, 
	\begin{equation*}
		\int_{B_{r\left(\epsilon\right)}^{c}}\left(1-f^{2}\right)\left|\phi_{j}\right|^{2}\leq\epsilon\left\Vert \phi_{j}\right\Vert _{L^{2}}^{2}<\epsilon.
	\end{equation*}
	On the other hand, by Rellich lemma $\phi_j\to0$ in $L^2(B_{r(\epsilon)})$, which means that by choosing $j$ big enough we have 
	\begin{equation*}
		\int_{B_{r\left(\epsilon\right)}}\left(1-f^{2}\right)\left|\phi_{j}\right|^{2}<\epsilon.
	\end{equation*}
	Finally, from equation \eqref{key}, we find 
	\begin{align*}
		1=&\left\Vert \phi_{j}\right\Vert _{H^{1}}^{2}\\
		=&B\left[\phi_{j},\phi_{j}\right]+\int_{B_{r\left(\epsilon\right)}}\left(1-f^{2}\right)\left|\phi_{j}\right|^{2}+\int_{B_{r\left(\epsilon\right)}^{c}}\left(1-f^{2}\right)\left|\phi_{j}\right|^{2}\\
		<&B\left[\phi_{j},\phi_{j}\right]+2\epsilon
	\end{align*}
	from which we draw a contradiction using \eqref{contradiction}.
 \qed
\subsection{Lemma \ref{lemma1}} \emph{
There exists a constant $C>0$ independent on $\eps$ such that for every pair $H\in C^{0,\gamma}(M\times\R^2)$ there exists a solution $\Phi=\Phi(H)$ to the equation 
	\begin{equation*}
		-\Delta_{t,{U_0}}\Phi-\eps^2\Delta_{M}\Phi+\Phi=H\quad \text{in }M\times\R^2
	\end{equation*}
defining a linear operator in $H$, such that
\begin{equation*}
\norm{\Phi}_{C^{2,\gamma}(M\times\R^2)}\leq C\norm{H}_{C^{0,\gamma}(M\times\R^2)}.
\end{equation*}
}

\medskip

\noindent{\bf Proof.}
Observe that we can embed $M\times\R^2\subset\R^5$. Let $B_R$ be the ball of radius $R$ centred in the origin of $\R^5$ and consider, for $k=1,2,\dots$, the following problem
\begin{equation}\label{cut-system}
 	\begin{cases}
 		-\Delta_{t,{U_0}}\Phi-\eps^2\Delta_{M}\Phi+\Phi=H&\text{in }(M\times\R^2)\cap B_k\\
 		\Phi=0&\text{on }(M\times\R^2)\cap \partial B_k.
 	\end{cases}
 \end{equation}
 For every $k$ problem \eqref{cut-system} admits a weak solution $\Phi_k$. This follows from Riez's theorem, along with the fact that the corresponding linear operator defines a positive, symmetric bilinear form on $H_{U_0,0}^1((M\times\R^2)\cap B_k)$, namely the closure on $C^{\infty}_c((M\times\R^2)\cap B_k)$ pairs under the $H_{U_0}^1$ norm. Moreover, the linear operator in \eqref{cut-system} satisfies on each component maximum principle, thus using $\|H\|_\infty$ as a barrier we obtain the uniform bound
 \begin{equation*}
 	\|\Phi_k\|_\infty\leq \|H\|_\infty\quad \forall k\in\mathbb{N}.
 \end{equation*}
Thus, using elliptic estimates we obtain the presence of a subsequence of $\Phi_k$ that converges uniformly over compact sets as $k \to \infty$ to a limit $\Phi$ that solves  
\begin{equation*}
	-\Delta_{t,{U_0}}\Phi-\eps^2\Delta_{M}\Phi+\Phi=H\quad\text{in }M\times\R^2,
\end{equation*}
and the estimate
\begin{equation*}
	\norm{\Phi}_{C^{2,\gamma}(M\times\R^2)}\leq C\norm{H}_{C^{0,\gamma}(M\times\R^2)}
\end{equation*}
for some constant $C>0$ independent on $\eps$.
\qed

\subsection{Lemma \ref{outer invertibility}}
\emph{
For every $\eps>0$ sufficiently small and any $\Gamma$ with $\|\Gamma\|_{C^{0,\gamma}_\mu(\R^4)}<\infty$ there exists a $\Psi=\Psi(\Gamma)$, defining a linear map of $\Gamma$, satisfying 
	\begin{equation*}
 -\eps^2\Delta_W\Psi+\Psi+Q_\eps\Psi=\Gamma\quad\text{on }\R^4.
	\end{equation*}
where $Q_\eps=(1-\zeta_2)T_W$.
Moreover
	\begin{equation*}
		\norm{\Psi}_{C^{2,\gamma}_\mu(\R^4)}\leq C\|\Gamma\|_{C^{0,\gamma}_\mu(\R^4)}
	\end{equation*}
 for a constant $C>0$ independent of $\eps$.
}

\medskip

\noindent{\bf Proof.}
This proof follows the same lines as the one of Lemma \ref{lemma1} above. We use again a barrier using the fact that the operator satisfies maximum principle. The weighted estimates follow from the same argument used in the proof of Proposition \ref{Inversion for LUbullet - theorem}.
\qed

\medskip

{\bf Acknowledgements:}
	The authors have been supported  by Royal Society Research Professorship RP-R1-180114, United Kingdom.

\bibliography{MGLbibl} 
\bibliographystyle{siam}

\end{document}